\documentclass[10pt]{amsart}

\usepackage{amsmath}
\usepackage{amsthm}
\usepackage{amsopn}
\usepackage{amssymb}
\usepackage[all]{xy}
\usepackage{lscape,xcolor}
\usepackage{graphicx}
\usepackage{hyperref}
\usepackage{mathtools}
\usepackage{soul}
\usepackage{stmaryrd}
\usepackage{verbatim}
\usepackage{rotating}
\usepackage{tikz}

\usepackage[normalem]{ulem}
\newcommand{\stkout}[1]{\ifmmode\text{\sout{\ensuremath{#1}}}\else\sout{#1}\fi}

\DeclareRobustCommand\circled[1]{\tikz[baseline=(char.base)]{
   \node[shape=circle,draw,inner sep=0pt] (char) {#1};}}

\definecolor{darkspringgreen}{rgb}{0.09, 0.45, 0.27}

\allowdisplaybreaks[1]

\newcommand{\mmod}{\! \sslash \!}

\newcommand{\mc}[1]{\mathcal{#1}}
\newcommand{\ull}[1]{\underline{#1}}
\newcommand{\mb}[1]{\mathbb{#1}}
\newcommand{\mr}[1]{\mathrm{#1}}

\newcommand{\abs}[1]{\lvert #1 \rvert}
\newcommand{\norm}[1]{\lVert #1 \rVert}
\newcommand{\bra}[1]{\langle #1 \rangle}
\newcommand{\br}[1]{\overline{#1}}

\newcommand{\td}[1]{\widetilde{#1}}

\newcommand{\ZZ}{\mathbb{Z}}

\newcommand{\QQ}{\mathbb{Q}}

\newcommand{\FF}{\mathbb{F}}

\newcommand{\tmf}{\mathrm{tmf}}
\newcommand{\bo}{\mathrm{bo}}
\newcommand{\bsp}{\mathrm{bsp}}

\def \HF2{\mr{H}\FF_2}

\newcommand{\KO}{\mr{KO}}

\def \AA0{\br{A \mmod A(0)}_*}

 \newtheorem{thm}[equation]{Theorem}
 \newtheorem{cor}[equation]{Corollary}
 \newtheorem{lem}[equation]{Lemma}
 \newtheorem{prop}[equation]{Proposition}
 \newtheorem{obs}[equation]{Observation}
  \newtheorem{rem}[equation]{Remark}
 
 \newtheorem*{thm*}{Theorem}
 \newtheorem*{cor*}{Corollary}
 \newtheorem*{lem*}{Lemma}
 \newtheorem*{prop*}{Proposition}
  \newtheorem*{not*}{Notation}

 \theoremstyle{definition}
 \newtheorem{defn}[equation]{Definition}

 \newtheorem{rmk}[equation]{Remark}

\newtheorem*{defn*}{Definition}
\newtheorem*{ex*}{Example}
\newtheorem*{exs*}{Examples}
\newtheorem*{rmk*}{Remark}
\newtheorem*{claim*}{Claim}

\numberwithin{equation}{section}
\numberwithin{figure}{section}
\DeclareMathOperator{\Ext}{Ext}

\newcommand{\s}{\wedge}
\newcommand{\Si}{\Sigma}

\newcommand{\E}[2]{\prescript{#1}{#2}{E}}

\usepackage[linecolor=black]{todonotes}
\usepackage{blindtext}
\usepackage{tcolorbox}

\definecolor{limegreen}{rgb}{0.2, 0.8, 0.2}
\definecolor{darkmagenta}{rgb}{0.55, 0.0, 0.55}
\definecolor{lavenderrose}{rgb}{0.91, 0.33, 0.5}
\definecolor{goldenpoppy}{rgb}{0.99, 0.76, 0.0}
\definecolor{seagreen}{rgb}{0.11, 0.35, 0.02}
\definecolor{maroon}{RGB}{128,0,0}
\definecolor{darkviolet}{RGB}{148,0,211}
\definecolor{twelve}{RGB}{100,100,170}
\definecolor{thirteen}{RGB}{100,150,50}
\definecolor{fourteen}{RGB}{200,0,0}
\definecolor{fifteen}{RGB}{0,200,0}
\definecolor{sixteen}{RGB}{0,0,200}
\definecolor{seventeen}{RGB}{200,0,200}
\definecolor{eighteen}{RGB}{0,200,200}

\newcommand{\Ex}{\Ext_{A_*}}

\def\gdm#1#2#3#4#5{
{\color{#1}(#2,#3:#4)_{#5}}
}

\def\ev#1#2#3{
{\color{#1}(#2,#3:#3)^{ev}}
}

\def\evm#1#2#3#4{
{\color{#1}(#2,#3:#3)^{ev}_{#4}}
}

\def\gdr#1#2#3#4{
{\color{#4}(#1,#2:#3)}
}

\def\evr#1#2#3{
{\color{#3}(#1,#2:#2)^{ev}}
}

\def\evrm#1#2#3#4{
{\color{#3}(#1,#2:#2)^{ev}_{#4}}
}

\title{On the $E_2$-term of the $\bo$-Adams spectral sequence}

\author{A.~Beaudry}\address{University of Colorado Boulder}\email{agnes.beaudry@colorado.edu}
\author{M.~Behrens}\address{University of Notre Dame}\email{mbehren1@nd.edu}
\author{P.~Bhattacharya}\address{University of Virginia}\email{pb9wh@virginia.edu}
\author{D.~Culver}\address{ University of Illinois Urbana-Champaign}\email{dculver@illinois.edu}
\author{Z.~Xu}\address{Massachusetts Institute of Technology}\email{xuzhouli@mit.edu}

\thanks{This material is based upon work supported by the National Science Foundation under Grant No. DMS-1050466/1611786 and DMS-1612020/1725563}

\begin{document}

\begin{abstract}
The $E_1$-term of the ($2$-local) $\bo$-based Adams spectral sequence for the sphere spectrum decomposes into a direct sum of a $v_1$-periodic part, and a $v_1$-torsion part.
Lellmann and Mahowald completely computed the $d_1$-differential on the $v_1$-periodic part, and  the corresponding contribution to the $E_2$-term.  The $v_1$-torsion part is harder to handle, but with the aid of a computer it was computed through the $20$-stem by Davis.  Such computer computations are limited by the exponential growth of $v_1$-torsion in the $E_1$-term.  In this paper, we introduce a new method for computing the contribution of the $v_1$-torsion part to the $E_2$-term, whose input is the cohomology of the Steenrod algebra.  We demonstrate the efficacy of our technique by computing the $\bo$-Adams spectral sequence beyond the $40$-stem.
\end{abstract}

\maketitle	

\tableofcontents


\section{Introduction}\label{sec:intro}

Let $\KO$ denote the ($2$-local) real $K$-theory spectrum, and let $\bo$ denote its connective cover. The $\bo$-based Adams
spectral sequence (ASS) for the sphere takes the form: 
$$ \E{bo}{}_1^{n,t} = \pi_t \bo \s \br{\bo}^{\s n} \Rightarrow (\pi^{s}_{t-n})_{(2)}. $$ 
Here, $\br{\bo}$ denotes the cofiber of the unit 
$$ S \rightarrow \bo \rightarrow \br{\bo} $$
and $\pi_*^s$ denotes the stable homotopy groups of spheres.

Many researchers have studied aspects of the $\bo$-ASS, most notably Mahowald in his paper \cite{Mahowaldbo}, where this
spectral sequence is used to prove the $2$-primary $v_1$-periodic telescope conjecture. However, a systematic study of this
spectral sequence as a tool to perform low-dimensional computations of the stable homotopy groups of spheres was not
undertaken until \cite{LM} and \cite{Davisbo}. Early on, the structure of the $E_1$-term was known \cite{Mahowaldbo}, \cite{Milgram}. Namely,
there is a splitting $$ \bo \s \bo \simeq \bigvee_i \Si^{4i} \bo \wedge B_i $$ where $B_i$ is the $i$th integral
Brown-Gitler spectrum \cite{CDGM}. This splitting iterates to yield a splitting 
$$ \bo \wedge \br{\bo}^n \simeq \bigvee_{i_1, \ldots, i_n}\Si^{4(i_1+ \cdots + i_n)} \bo \s B_{i_1} \s \cdots \s B_{i_n} $$
where each index satisfies $i_l > 0$.
The homotopy type of the wedge summands in the above splitting were also determined in \cite{Mahowaldbo}, \cite{Milgram}: for a multi-index $I = (i_1, \ldots, i_n)$, there is an equivalence
$$ \bo \s B_{i_1} \s \cdots \s B_{i_n} \simeq b_I \vee HV_I $$
where $b_I$ is a certain Adams cover of $\bo$ or $\bsp$ (see Section~\ref{sec:E1}), and $HV_I$ is the Eilenberg-MacLane spectrum associated to a graded $\FF_2$-vector space $V_I$.

There is a corresponding decomposition
\begin{align*}
\E{bo}{}_1^{n,*} & = \bigoplus_{I= (i_1, \ldots, i_n)} \Si^{4(i_1, \ldots, i_n)} \pi_*  b_I \oplus \Si^{4(i_1, \ldots, i_n)}  V_I.
\end{align*}
The subspaces
$$ V^{n,*} := \bigoplus_{I=( i_1, \cdots ,i_n)} \Si^{4(i_1, \ldots, i_n)} V_I $$
are closed under the $d_1$-differential.  Following \cite{LM}, \cite{Davisbo}, we define $\mc{C^{*,*}}$ to be the quotient complex
$$ 0 \rightarrow V^{*,*} \rightarrow \E{bo}{}_1^{*,*} \rightarrow \mc{C}^{*,*} \rightarrow 0 $$
so that
$$ \mc{C}^{n,*} \cong \bigoplus_{I= (i_1, \ldots, i_n)} \Si^{4(i_1, \ldots, i_n)}  \pi_*  b_I, $$
resulting in a long exact sequence
\begin{equation}\label{eq:LES}
\cdots \rightarrow H^{*,*}(V) \rightarrow \E{bo}{}_2^{*,*} \rightarrow H^{*,*}(\mc{C}) \xrightarrow{\partial} H^{*+1, *}(V) \rightarrow \cdots.
\end{equation}

Historically, most of the applications of the $\bo$-ASS have rested on analyzing $H^{*,*}(\mc{C})$ while bounding the effects of $H^{*,*}(V)$.  In fact, 
a complete computation of $H^{*,*}(\mc{C})$ was accomplished in \cite{LM}.  By contrast, the complex $V^{*,*}$, and its cohomology, have proven to be the major obstacle to using the $\bo$-ASS to perform low-dimensional computations.  Carlsson computed $V^{1,*}$ explicitly in \cite{Carlsson}, but the result is rather unwieldy.  Davis \cite{Davisbo} used computer computations to compute $V^{n,t}$, and its cohomology, through the range $t \le 25$.  He observed there that a rapid exponential growth in the dimension of $V^{n,t}$ severely limits this approach.  Furthermore, Davis's approach does not include a method for computing the map $\partial$ in (\ref{eq:LES}).  In essence, $V^{*,*}$ resembles the cobar complex for the Steenrod algebra (but seemingly without a good conceptual description), and as such, is too large for extensive computer computations.

Motivated by this disparity of computability, we shall refer to elements of $\E{bo}{}_2^{*,*}$ with non-trivial image in $H^{*,*}(\mc{C})$ as \emph{good}, and we shall refer to non-zero elements of $\E{bo}{}_2^{*,*}$ which are in the image of the map from $H^{*,*}(V)$ as \emph{evil}.  The purpose of this paper is to present a method of computing the evil part of $\E{bo}{}_2^{*,*}$.  This involves (1) computing $H^{*,*}(V)$, and (2) understanding the long exact sequence (\ref{eq:LES}).

The basic idea is to use an algebraic analog of the $\bo$-resolution (called the Mahowald spectral sequence in \cite{Miller}), studied by Davis, Mahowald, Miller and others, which we will refer to as the $\bo$-MSS.  The $\bo$-MSS is a spectral sequence of the form
$$ \E{bo}{alg}_1^{n,s,t} = \Ext^{s,t}_{A(1)}(\br{A\mmod A(1)}^{n}, \FF_2) \Rightarrow \Ext^{s+n,t}_{A}(\FF_2, \FF_2). $$
Analogous to the topological decomposition, there is a decomposition
$$ \E{bo}{alg}_1^{*,*,*} = \mc{C}^{*,*,*}_{alg} \oplus V^{*,*,*}_{alg} $$
and thus a long exact sequence
\begin{equation}\label{eq:algLES}
\cdots \rightarrow H^{*,*,*}(V_{alg}) \rightarrow \E{bo}{alg}_2^{*,*,*} \rightarrow H^{*,*,*}(\mc{C}_{alg}) \xrightarrow{\partial_{alg}} H^{*+1, *,*}(V) \rightarrow \cdots.
\end{equation}

The key observation (implicit in \cite{Davisbo}) is:
\begin{obs}
The terms $V_{alg}^{n,s,t}$ are zero unless $s = 0$, and there is an isomorphism of cochain complexes:
$$ V_{alg}^{n,0,t} \cong V^{n,t} $$
and therefore an isomorphism
$$ H^{*,*}(V) \cong H^{*,0,*}(V_{alg}). $$
\end{obs}
 In fact, the map $\partial$ turns out to be computable in terms of the map $\partial_{alg}$.  In short, a complete understanding of the evil part of $\E{bo}{}_2$ can be extracted from a complete understanding of the evil part of $\E{bo}{alg}_2$.  

Our methods rest on the idea that the groups 
$\Ext^{*,*}_{A}(\FF_2, \FF_2)$ are easily understood (through a range) by computer computations using minimal resolutions \cite{Bruner}.  Furthermore, analogous to the topological situation, the groups $H^{*,*,*}(\mc{C}_{alg})$ can be completely computed.  One can then deduce $H^{*,*,*}(V_{alg})$ and $\partial_{alg}$ from the existence of the $\bo$-MSS, using knowledge of $\Ext^{*,*}_{A}(\FF_2, \FF_2)$,  \emph{provided one has a means of determining which elements of $\Ext$ are detected by good classes and which are detected by evil classes in the $\bo$-MSS}.  

We shall say a non-trivial class $x \in \Ext^{*,*}_{A}(\FF_2, \FF_2)$ is \emph{evil} if there exists an evil class in the $\bo$-MSS which detects it; otherwise we shall say $x$ is \emph{good}.
Our main theorem establishes a precise relationship between being good, and being $v_1$-periodic (Lemma~\ref{lem:onethird} and Theorem~\ref{thm:dichotomy}):

\begin{thm*}[Dichotomy Principle]
Suppose $x$ is a non-trivial class in $\Ext^{s,t}_{A}(\FF_2, \FF_2)$.
\begin{enumerate}
\item 
If $x$ is $v_1$-torsion, it is evil.

\item If $x$ lies above the 1/3 line ($t-s < 3s$), then it is good and $v_1$-periodic.

\item 
Suppose $x$ is $v_1$-periodic.  There is an $M \gg 0$ for which $y = v_1^{M} x$ is defined and lies above the $1/3$-line.  Suppose $y$ has $\bo$-filtration $n$.  The class $x$ is good if and only if 
$$ s \ge n. $$
\end{enumerate}
\end{thm*}

Our technique of using knowledge of $\Ext^{s,t}_{A}(\FF_2, \FF_2)$ and $H^{*,*,*}(\mc{C}_{alg})$ to deduce $H^{*,*}(V)$ via the notions of good and evil is encoded more fluently in a refinement of the $\bo$-MSS we develop, called the \emph{agathokakological spectral sequence}
$$ \E{akss}{alg}^{*,*,*}_1 = \mc{C}_{alg}^{*,*,*} \oplus V^{*,*} \Rightarrow \Ext^{s,t}_{A}(\FF_2, \FF_2). $$
The indexing in this spectral sequence is non-standard, but is set up in such a manner that good and evil classes actually live in distinct tri-degrees.

Mahowald's proof of the $v_1$-periodic telescope conjecture \cite{Mahowaldbo} rested on two fundamental theorems: the \emph{Bounded Torsion Theorem}, and the \emph{Vanishing Line Theorem}.  We note that our perspective on the $\bo$-ASS $E_2$-term organically produces these results (Corollary~\ref{cor:BTT} and Theorem~\ref{thm:vanishing}).

We demonstrate the efficacy of our technique by computing the $\bo$-ASS through the $40$-stem.  The previous computations of \cite{LM} and \cite{Davisbo} only run through the $20$-stem.  In principle, our method could be employed to compute the $\bo$-ASS through a larger range, but it quickly becomes apparent that as a function of the stem,  increasingly large portions of $\pi_*^s$ are detected by evil classes in the $\bo$-ASS, and beyond our range it appears that the $\bo$-ASS is not much more effective than the classical ASS.

Besides being an attempt to streamline and extend previous work on the $\bo$-ASS, this paper represents a developmental platform for methods the authors hope to employ to study the $\tmf$-resolution (extending the program of \cite{BOSS}).

It should be noted that Gonz\'alez's work on the odd primary $BP\bra{1}$-resolution \cite{Gonzalez2} suggests an interesting alternative to the methodology of this paper. Namely, Gonz\'alez shows that the analog of the complex $V^{*,*}$ for the Adams cover $S^{\bra{1}}$ of the sphere  spectrum is isomorphic to a cobar complex for a relative $\Ext$ group. A similar observation carries over to the $2$-primary case.  It would be an interesting project to attempt to compute these relative Ext groups with some variant of the May spectral sequence or a minimal resolution technique.

\subsection*{Organization of the paper}

The paper is organized as follows. In Section~\ref{sec:E1} we review the structure of the $E_1$-page of the $\bo$-ASS,
together with the good component of the $d_1$-differential. In Section~\ref{sec:wss} we describe the Lellmann-Mahowald
weight spectral sequence, a spectral sequence for computing $H^{*,*}(\mc{C})$.  In Section~\ref{sec:alg}, we describe the algebraic $\bo$ resolution (the $\bo$-MSS), and compute the corresponding weight spectral sequence converging to $H^{*,*,*}(\mc{C}_{alg})$. 
In Section~\ref{sec:v1}, we review the notion of $v_1$-periodicity in $\Ext$, define the $v_1$-periodic $\bo$-resolution, and discuss its convergence.  
In Section~\ref{sec:AA0}, we discuss the $\bo$-MSS for $\br{A \mmod A(0)}$.  The Ext groups of $\br{A \mmod A(0)}$ are identical to those of $\FF_2$ away from the $0$-stem, but the $v_1$-periodic behavior of the former is much more manageable.
In Section~\ref{sec:rules}, we define the agathokakological spectral sequence (AKSS) and prove the Dichotomy Principle.  We also introduce a topological analog of the AKSS, which refines the $\bo$-ASS.  In Section~\ref{sec:algcomp}, we use the results of the previous section, together with knowledge of $\Ext_A(\FF_2, \FF_2)$, to completely compute the algebraic $\bo$-resolution through a large range.  The information this affords us about $H^{*,*}(V)$, and the connecting homomorphism $\partial$, is then used to compute the $\bo$-ASS through the same range in Section~\ref{sec:topcomp}.

\subsection*{Conventions} In the remainder
of this paper everything is implicitly localized at the $p = 2$.  

Homology $H_*$ will always be implicitly taken with $\FF_2$ coefficients.
We let $A$ denote the $2$-primary Steenrod algebra, and $A_*$ its dual.  For a sub-Hopf algebra $B \subseteq A$, we will use $B_*$ to denote its dual, and we will use $A\mmod B_*$ to denote the dual of the Hopf-algebra quotient $A\mmod B$.  We let $\zeta_i$ denote the conjugate of the element $\xi_i \in A_*$.  We shall sometimes use the abbreviation
$$ \Ext^{*,*}_{A_*}(M) := \Ext^{*,*}_{A_*}(\FF_2, M). $$
For any connective ($2$-local) finite type spectrum $X$, we let $X^{\bra{s}}$ denote its $s$th Adams cover, so that
$$ X \leftarrow X^{\bra{1}} \leftarrow X^{\bra{2}} \leftarrow \cdots $$
is a minimal Adams resolution of $X$.
The associated classical Adams spectral sequence (ASS) will be denoted
$$ \E{ass}{}^{s,t}_2(X) = \Ext_{A_*}^{s,t}(H_*X) \Rightarrow \pi_{t-s} X. $$
For a non-zero element $x \in \pi_*X$, we let 
$$ [x] \in \E{ass}{}_\infty(X) $$
denote the class in $E_\infty$ which detects it.  Cycles in the $E_r$-page of the ASS will be automatically regarded as representing elements of the $E_{r+1}$-page.  Thus, given an element $y \in \E{ass}{}_2^{*,*}$, the expression
$$ y = [x] $$
means that $y$ detects $x$ in the ASS.  For a general ring spectrum $R$, we let $\{ \E{R}{}^{*,*}_r(X) \}$ denote the $R$-based Adams spectral sequence for $X$.

For a non-negative integer $i$, we let $\alpha(i)$ denote the sum of the digits of its base $2$ expansion.
For a multi-index 
$$ I = (i_1, \ldots, i_n) $$
we define
\begin{align*}
\abs{I} & := n, \\
\norm{I} & := i_1 + \cdots + i_n, \\
\alpha(I) & := \alpha(i_1) + \cdots + \alpha(i_n).
\end{align*}

\subsection*{Acknowledgments}  The authors thank the referee for their valuable suggestions and comments, and John Rognes for pointing out an error in the statement of a computation.


\section{The $E_1$-term of the $\bo$-ASS}\label{sec:E1}

Let $B_i$ denote the $i$th integral Brown-Gitler spectrum (see, for example, \cite{CDGM}).  The integral Brown-Gitler spectra are a sequence of finite complexes whose colimit is the integral Eilenberg-Maclane spectrum:
$$ S = B_0 \subset B_1 \subset \cdots \subset H\ZZ. $$
The image of their homology in 
$$ H_* H\ZZ \cong \FF_2[\zeta_1^2, \zeta_2, \ldots ] \subset A_* $$
is the $A_*$-subcomodule spanned by monomials $m$ of weight $wt(m) \le 2i$, where
$$ wt(\zeta_i) := 2^{i-1}. $$

\begin{thm}[Mahowald \cite{Mahowaldbo}, Milgram \cite{Milgram}]\label{thm:splitting}
	We have
	$$ \bo \wedge \br{\bo}^n \simeq \bigvee_{\abs{I} = n} \Si^{4\norm{I}} \bo \s B_I $$
	 where each term of the multi-index $I = (i_1, \ldots i_n)$ satisfies $i_l > 0$ and 
	 $$ B_I := B_{i_1} \s \cdots \s B_{i_n}. $$
\end{thm}

Mahowald and Milgram also determined the homotopy type of the wedge summands above.

\begin{thm}[Mahowald \cite{Mahowaldbo}, Milgram \cite{Milgram}]
	There are equivalences
	$$ \bo \wedge B_I \simeq b_I \vee HV_I $$
	where
	$$ b_I := 
	\begin{cases}
	\bo^{\bra{2\norm{I}-\alpha(I)}}, & \norm{I} \: \mr{even}, \\
	\bsp^{\bra{2\norm{I}-\alpha(I)-1}}, & \norm{I} \: \mr{odd}. \\
	\end{cases}
	$$
	and $V_I$ is a graded $\FF_2$-vector space of finite type.
\end{thm}

As indicated in the introduction, we will define 
\begin{align*}
V^{n,*} &:= \bigoplus_{\abs{I} = n} \Si^{4\norm{I}} V_I, \\
\mc{C}^{n,*} &:= \bigoplus_{\abs{I} = n} \Si^{4\norm{I}} \pi_* b_I.
\end{align*}
The subspace $V^{*,*} \subseteq \E{bo}{}_1^{*,*} $ is a subcomplex with respect to the $d_1$ differential \cite{LM}, and we endow $\mc{C}^{*,*}$ with the structure of the quotient complex.
The direct sum decomposition
$$ \E{bo}{}^{*,*}_1 = \mc{C}^{*,*} \oplus V^{*,*} $$
results in a decomposition of elements $x \in \E{bo}{}_1^{*,*}$ into \emph{good} and \emph{evil} components:
$$ x = x_{good} + x_{evil}. $$ 
We shall refer to the $d_1$ differential, restricted to $V^{*,*}$, as $d^{evil}_1$.  The induced differential on $\mc{C}^{*,*}$ will be denoted $d_1^{good}$. 

Since the spectra $b_I$ are Adams covers, it is useful to understand $\pi_* b_I$, and the associated differential $d_1^{good}$ in the complex $\mc{C}^{*,*}$, in terms of the Adams spectral sequence for $\pi_* b_I$.
To this end, we will denote elements of $\mc{C}^{*,*}$
using notation
$ x(I)$
for $x$ an element of $\pi_* b_I$.  
We denote the generators of $\E{ass}{}_2^{*,*}(b_I)$
using notation
$$ v_0^k w^m h_1^a $$
where $v_0 = [2]$, $w$ is the formal expression 
$$ w := v_1^2/v_0^2 $$
(here $v_1^4$ detects the Bott element in $\pi_8(\bo)$) and $h_1 = [\eta]$.
For example, in the case of $b_{(2)} = \bo^{\bra{3}}$, generators of $\E{ass}{}_2^{*,*}(b_{2})$ are labeled in Figure~\ref{fig:bIexample}.
\begin{figure}
\centering
\includegraphics[width=0.7\linewidth]{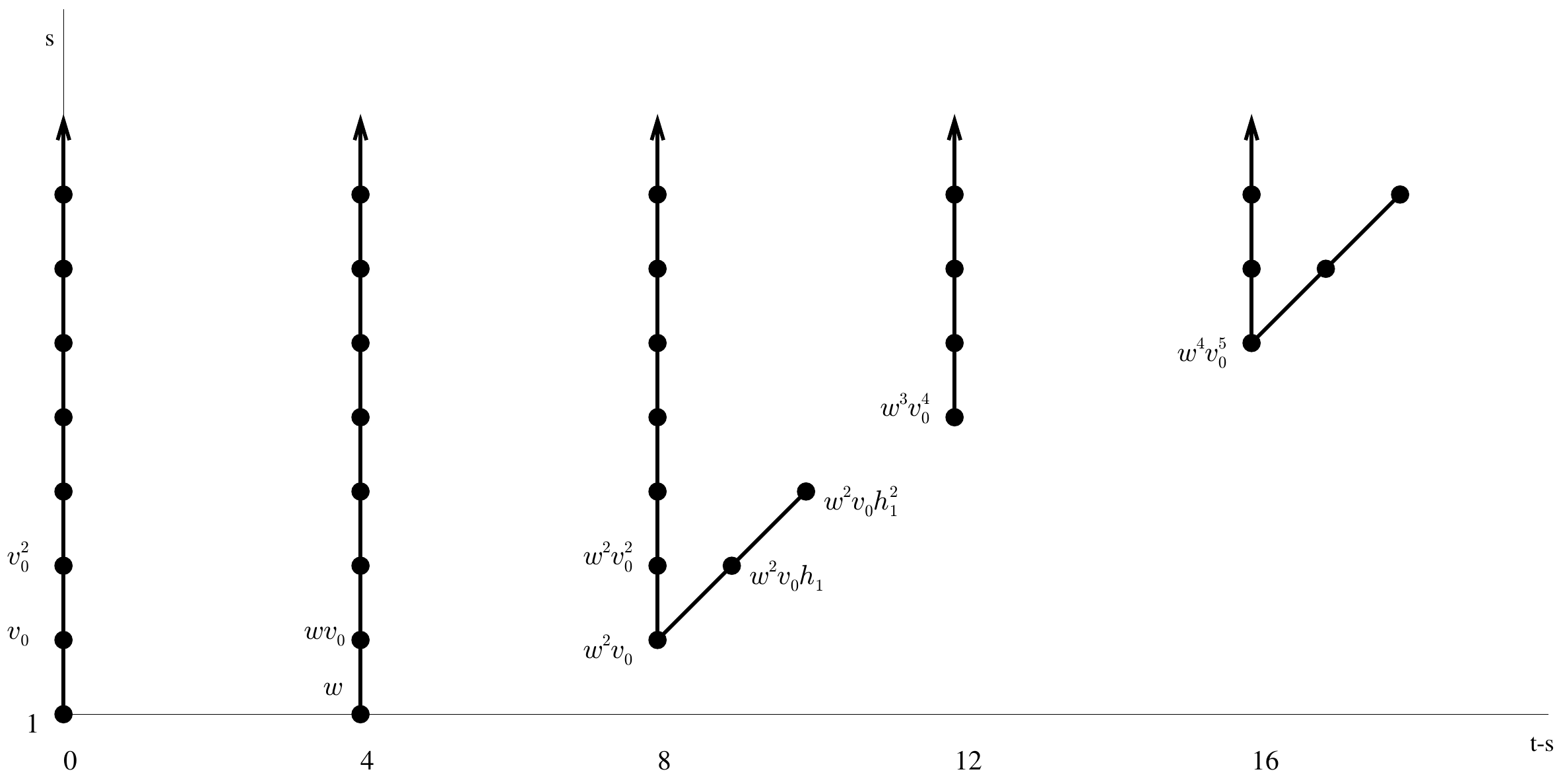}
\caption{An example of the labeling system for generators in the Adams spectral sequence for $b_I$ ($I=(2)$).}
\label{fig:bIexample}
\end{figure}
With respect to this notation, up to terms of higher Adams filtration, the differential in the complex $\mc{C}^{*,*}$ is determined by the following theorem.

\begin{thm}[Mahowald \cite{Mahowaldbo}, Lellmann-Mahowald \cite{LM}]\label{thm:d1}
	The differential $d_1^{good}$ in the complex $\mc{C}^{*,*}$ is determined modulo elements of higher Adams filtration by the formula
	\begin{multline*}
	 [d_1^{good}(w^m(i_1, \ldots, i_n))] = \sum_{m' + m'' = m} v_0^{\alpha(m')+\alpha(m'')-\alpha(m)} w^{m'}(m'', i_1, \ldots, i_n) \\ + 
	\sum_k \sum_{i_k'+i_k'' = i_k} v_0^{\alpha(i_k') + \alpha(i_k'')-\alpha(i_k)} w^m(i_1, \ldots, i_k', i_k'', \ldots, i_n) \end{multline*}
	and the fact that the differential is $v_0$ and $h_1$-linear.  (In the above formula, we are saying that for each multi-index $J$, the $J$-component of the differential $d_1^{good}$ is detected by the indicated element in the Adams spectral sequence for $b_J$.)
\end{thm}

\begin{rmk}
	Lellmann and Mahowald actually compute the differential $d_1^{good}$ more precisely (i.e., not up to ambiguity of higher Adams filtration) but our presentation in Theorem~\ref{thm:d1} is precise enough to contain all of the relevant features of this differential.
\end{rmk}

We present a new proof of Theorem~\ref{thm:d1} which we think is much simpler than the approach of \cite{LM} (based on Adams operations and numerical polynomials).

\begin{proof}[Proof of Theorem~\ref{thm:d1}]
Consider the zig-zag of complexes:
$$ C^*_{BP_*BP}(BP_*[1/2]) = v_0^{-1}\E{BP}{}^{*,*}_1 \xrightarrow{f} v_0^{-1} \E{E(1)}{}^{*,*}_1 \xleftarrow{g} \mc{C}^{*,*} $$
(where $C^*_{BP_*BP}(BP_*[1/2])$ denotes the normalized cobar complex for $BP_*BP$ --- with $2$ inverted) induced by the zig-zag of ring spectra
$$ BP [1/2] \rightarrow E(1)[1/2] \leftarrow \bo. $$
Since the map $g$ is an injection in internal degrees $t \equiv 0 \mod 4$, it will suffice to lift the image of $w^m(I)$ to the localized cobar complex for $BP_*BP$, and verify the formula for the differential there. 
Define
$$ \tau := t_1^2 + v_1t_1 \in BP_*BP; $$
this primitive element detects $\nu$ in the Adams-Novikov spectral sequence.
Let $e_i$ denote the generator of $\pi_0(b_{(i)})$.  Under the zig-zag
$$ BP_*BP \xrightarrow{f} E(1)_*E(1) \xleftarrow{g} \mc{C}^{1,*} $$
we have
$$ [f(\tau^i)] = [g(e_i)]. $$
This follows from \cite[Sec.~3.4]{BOSS} (where $e_i$ is instead denoted ``$e_{4i}$''), together with the fact that $[\tau] = [t_1^2]$ (since $v_1$ has Adams filtration 1).
It follows that we have  
$$ \left[f\left(\frac{v_1^{2m}}{2^{2m}} [\tau^{i_1}| \cdots | \tau^{i_n}]\right)\right]
= [g(w^m(i_1, \ldots, i_n))]. $$
Observe that in the Hopf algebroid $(BP_*[1/2], BP_*BP[1/2])$ we have the formulas:
\begin{align*}
\eta_R\left(\frac{v_1^2}{4}\right) = \frac{v_1^2}{4} + \tau, \\
\psi(\tau) = \tau \otimes 1 + 1 \otimes \tau.
\end{align*}
By the binomial theorem, together with the fact that
$$ \nu_2 \binom{a}{b} = \alpha(b)+\alpha(a-b)-\alpha(a) $$
we have
\begin{align*}
\left[\eta_R\left(\frac{v_1^{2m}}{2^{2m}}\right)\right] &= \sum_{m = m' + m''}v_0^{\alpha(m')+\alpha(m'') - \alpha(m)}\frac{v_1^{2m'}}{2^{2m'}}\tau^{m''}, \\
[\psi(\tau^i)] &= \sum_{i'+i'' = i} v_0^{\alpha(i')+\alpha(i'') - \alpha(i)}\tau^{i'} \otimes \tau^{i''}.
\end{align*}
We deduce 
\begin{multline*}
	 \left[d\left(\frac{v_1^{2m}}{2^{2m}} [\tau^{i_1}| \cdots | \tau^{i_n}] \right)\right] = \sum_{m' + m'' = m} v_0^{\alpha(m')+\alpha(m'')-\alpha(m)} \frac{v_1^{2m'}}{v_0^{2m'}} [\tau^{m''}| \tau^{i_1}| \cdots | \tau^{i_n}] \\ + 
	\sum_k \sum_{i_k'+i_k'' = i_k} v_0^{\alpha(i_k') + \alpha(i_k'')-\alpha(i_k)} \frac{v_1^{2m}}{v_0^{2m}}[\tau^{i_1}|\cdots|\tau^{i_k'}|\tau^{i_k''}| \cdots |\tau^{i_n}]. \end{multline*}
The result follows.
\end{proof}


\section{The weight spectral sequence}\label{sec:wss}

In order to compute $H^{*,*}(\mc{C})$, Lellmann and Mahowald \cite{LM} introduced a \emph{weight spectral sequence}.  This section is a summary of their work.  Endow $\mc{C}^{*,*}$ with a decreasing filtration by weight ({$wt$}), where we define
$$ wt(x(I)) = \norm{I}. $$
We see from Theorem~\ref{thm:d1} that $d_1^{good}$ does not decrease weight.  There is a resulting spectral sequence
$$ \E{wss}{}_0^{n,*,w} := [\mc{C}^{n,*}]_{wt=w} \Rightarrow H^{n,*}(\mc{C}) $$
with differentials
$$ d_r^{wss}: \E{wss}{}_r^{n,*,w} \rightarrow \E{wss}{}_r^{n+1, *, w+r}. $$
From Theorem~\ref{thm:d1}, we see that $d_0^{wss}$ is given (modulo terms of higher Adams filtration) by
\begin{equation}\label{eq:d0wss} 
[d_0^{wss}(x(i_1, \ldots, i_n))] = \sum_k \sum_{i_k'+i_k'' = i_k} v_0^{\alpha(i_k') + \alpha(i_k'')-\alpha(i_k)} x(i_1, \ldots, i_k', i_k'', \ldots, i_n). 
\end{equation}
Even with this explicit differential, the calculation of $\E{wss}{}_1$ is not immediate.  

To this end, paraphrasing \cite{LM}, we will introduce an \emph{Adams filtration spectral sequence} $\{\E{AF}{}_r\}_{r \ge 0}$ to compute $\E{wss}{}_1$.  We will be able to compute $\E{AF}{}_1$ directly, but the computation of $\E{AF}{}_2$ will be clarified by the use a \emph{lexographical filtration spectral sequence} $\{ \E{LF}{}_\alpha\}$ (technically speaking, this spectral sequence is indexed on a totally ordered set).  The situation is summarized by the following:
$$ \E{wss}{}_0 = \E{AF}{}_0 \rightsquigarrow \E{AF}{}_1 = \E{LF}{}_0\Rightarrow \E{AF}{}_2 = \E{wss}{}_1 \Rightarrow H(\mc{C}). $$

\subsection*{The Adams filtration spectral sequence}

Endow $\E{wss}{}_0^{*,*,*}$ with a decreasing filtration by Adams filtration (AF), where we define
$$ AF(x(I)) := AF(x). $$
There is a resulting \emph{Adams filtration spectral sequence} with
$$ \E{AF}{}_0^{n,*,w,a} := [\E{wss}{}_0^{n,*,w}]_{AF = a} \Rightarrow \E{wss}{}_1^{n,*,w} $$
with differentials
$$ d_r^{AF}: \E{AF}{}_r^{n,*,w,a} \rightarrow \E{AF}{}_r^{n+1, *, w, a+r}. $$

\begin{prop}[Lellman-Mahowald \cite{LM}]
	An additive basis for $\E{AF}{}_1$ is given by elements
	$$ x h_2^{k_2} h_3^{k_3}h_4^{k_4} \ldots $$
	indexed by $K = (k_2, k_3, \ldots)$, detected by $x(I[K])$ where 
	$$ I[K] = (\underbrace{1, \ldots ,1}_{k_2}, \underbrace{2, \ldots ,2}_{k_3}, \underbrace{4, \ldots ,4}_{k_4}, \ldots). $$
	Here, the elements $x$ run through a basis of $\E{ass}{}_2(b_{I[K]})$.
\end{prop}
	
\begin{proof}
By (\ref{eq:d0wss}), 
\begin{equation}\label{eq:d0AF} 
[d_0^{AF}(x(i_1, \ldots, i_n))] = \sum_k \sum_{\substack{i_k'+i_k'' = i_k \\ \alpha(i_k')+\alpha(i_k'') = \alpha(i_k)}} x(i_1, \ldots, i_k', i_k'', \ldots, i_n). 
\end{equation}
	The essential observation is that if $i = i' + i''$ then
	$$ \alpha(i) = \alpha{(i')} + \alpha{(i'')}, $$
	if and only if $i'$ and $i''$ have dyadic expansions with $1$'s in complementary places. 
	Consider the subcomplex $\mc{B}^n$ of $\E{AF}{}_0^{n,*,*,*}$ spanned by the terms $1(I)$.  There is an isomorphism
	$$ \mc{B}^* \cong \bar{C}^{*}(\FF_2, E, \FF_2) $$
	(the normalized cobar complex for the primitively generated exterior algebra $E = \Lambda_{\FF_2}[e_0, e_1, \ldots]$) given by 
	$$ 1(i_1, \ldots, i_n) \mapsto [e(i_1)|\ldots|e(i_n)]. $$
	Here, for $i$ expressed dyadically as 
	$$ i = \epsilon_0 + \epsilon_1 2 + \epsilon_2 4 + \cdots $$
	(with $\epsilon_j \in \{0,1\}$), the element $e(i)$ denotes the element
	$$ e(i) := e_0^{\epsilon_0} e_1^{\epsilon_1} \cdots \in E. $$
	We deduce
	\begin{align*} H^*(\mc{B}) \cong & \Ext^*_{E}(\FF_2, \FF_2) \cong
		 \FF_2[h_2, h_3, \cdots ]
    \end{align*}
	(where $h_i$ is represented by the cocycle $[e_{i-2}]$ in the cobar complex).  The structure of $\E{AF}{}_1$ follows from the fact that for 
	$$ I = (i_1, \ldots, i_n) $$ 
	and
	$$ I' = (i_1, \ldots, i_k', i''_k, \ldots i_n) $$
	with $i_k = i'_k + i''_k$ and
	$$ \alpha(i_k) = \alpha(i_k')+\alpha(i_k'') $$
	we have
	$$ b_I \simeq b_{I'}. $$
\end{proof}

\begin{rmk}
	Our naming of the polynomial generators $h_2, h_3, \ldots$ of $\E{AF}{}_1$ may seem bizarre, and differs from \cite{LM}, where these generators are named $h_0, h_1, \ldots$.  Our reason for this different indexing is that in our notation, $h_i$ will correspond to the element of the same name in the classical Adams spectral sequence for the sphere.
\end{rmk}

\subsection*{The lexigraphical filtration spectral sequence}

In order to compute $\E{AF}{}_2$, we observe that from (\ref{eq:d0wss}) we have
$$
d_1^{AF}h_i = 
\begin{cases}
v_0 h_{i-1}^2, & i \ge 3, \\
0, & i = 2.
\end{cases}
$$
The Adams filtration spectral sequence, like all the spectral sequences we are employing, is multiplicative, and hence
$$ d_1^{AF}(xh_2^{k_2}h_3^{k_3}\cdots h_l^{k_l}) = \sum_{\substack{3 \le j \le l \\ k_j \: \mr{odd}}} v_0xh_2^{k_2} \cdots h_{j-1}^{k_{j-1}+2} h_j^{k_j-1} \cdots h_l^{k_l}. $$
For the purposes of analyzing the resulting cohomology, it is useful to order the monomials 
$$ h_2^{k_2} h_3^{k_3} h_4^{k_4} \cdots $$
by left lexigraphical ordering on the sequence 
$$
K = (k_2, k_3, k_4, \ldots ).
$$
In this filtration, we have
$$ (k_2, k_3, k_4, \ldots) < (k'_2, k'_3, k'_4, \ldots) $$
if $k_2 < k'_2$, or $k_2 = k'_2$ and $k_3 < k_3'$, or $k_2 = k_2'$ and $k_3=k_3'$ and $k_4 < k'_4$, etc.

We will call the resulting filtration (indexed by the ordinal $\omega^\omega$) \emph{lexigraphical filtration} (LF).  
The filtration is multiplicative, and the differential $d_2^{AF}$ increases lexigraphical filtration.  Amusingly, one way to organize the resulting cohomology is via the $\omega^\omega$-indexed spectral sequence based on this filtration (c.f. \cite{Hu}, \cite{Matschke}):
$$ \E{LF}{}_\alpha \Rightarrow \E{AF}{}_2. $$
The differentials in this spectral sequence (which for expediency of notation we do not index) are
\begin{equation}\label{eq:dLF} d^{LF}(x h_2^{k_2} h_3^{k_3} \cdots h_l^{k_l}) = v_0 x h_2^{k_2} h_3^{k_3} \cdots h_{l-1}^{k_{l-1}+2}h_{l}^{k_l-1}, \quad k_l \: \mr{odd}.
\end{equation}

\begin{prop}[Lellman-Mahowald \cite{LM}, \cite{Mahowaldbo}]\label{prop:E2AF}
The Adams filtration spectral sequence term $\E{AF}{}_2$ has a basis given by elements
$$ x h_2^\epsilon, \qquad \epsilon = 0,1, $$
and
$$ y h_2^{k_2} h_3^{k_3} \cdots h_{l}^{k_l}, \qquad k_l \ge 2, $$
where $x$ runs through a basis of 
$$
\begin{cases} \E{ass}{}_2^{*,*}(\bo), & \epsilon = 0, \\ \E{ass}{}_2^{*,*}(\bsp), & \epsilon = 1, \end{cases}
$$ and $y$ runs through a basis of
$$ \E{ass}{}^{0,*}(b_{I[k_2, \ldots, k_l]}). $$
\end{prop}

\begin{proof}
The proof amounts to analyzing the result of running the differentials of (\ref{eq:dLF}) in order of increasing lexigraphical filtration.  This analysis is simplified by the fact that the differentials in the lexigraphical filtration spectral sequence send monomials to monomials.

Equation (\ref{eq:dLF}) implies that 
$$ d^{LF}(x h_2^\epsilon )= 0. $$
As nothing can hit these classes, these provide the first part of the basis.  Note that for $K = (k_2, \ldots, k_l)$
\begin{align*}
\norm{I[K]} & = k_2+2k_3+ \cdots + 2^{l-2}k_l, \\
\alpha(I[K]) & = k_2 + k_3 + \cdots + k_l
\end{align*}
and therefore for $K' = (k_2, k_3, \cdots, k_l+2, k_l-1)$ we have
\begin{align*}
\norm{I[K']} = \norm{I[K]},  \\
\alpha(I[K']) = \alpha(I[K])+1.
\end{align*}
It follows that
$$ b_{I[K]} = b^{\bra{1}}_{I[K']}. $$
Therefore the differentials
$$ d^{LF}(xh_2^{k_2} \cdots h_{k-1}^{k_{l-1}}h_l) = v_0 h_2^{k_2} \cdots h_{l-1}^{k_{l-1}+2} $$
are all non-trivial.  The only elements not hit by the differentials above are of the form
$$ x h_2^{a_2} \cdots h_l^{a_m} $$
with $a_m \ge 2$ and $AF(x) = 0$.  The remaining possible differentials
$$ d^{LF}(xh_2^{k_2} \cdots h_{k-1}^{k_{l-1}}h_l^{k_l}) = v_0 h_2^{k_2} \cdots h_{l-1}^{k_{l-1}+2}h_{l}^{k_l-1} $$
(with $k_l \ge 3$ odd) are all zero since their targets are already killed by the shorter differentials
$$ d^{LF}(xh_2^{k_2} \cdots h_{k-1}^{k_{l-1}+2}h_l^{k_l-3}h_{l+1}) = v_0 h_2^{k_2} \cdots h_{l-1}^{k_{l-1}+2}h_{l}^{k_l-1}. $$
\end{proof}

We have at this point deduced Mahowald's ``Bounded Torsion Theorem'' \cite{Mahowaldbo}:

\begin{cor}\label{cor:BTT}
For $n \ge 2$ and $a > 0$, we have
$$ \E{AF}{}_2^{n,*,*,a} = 0. $$
\end{cor}

Because what remains in $\E{AF}{}_2$ (with the exception of the classes $xh_2^\epsilon$) is concentrated in Adams filtration 0, there are no further differentials in the Adams filtration spectral sequence, and we deduce the following corollary.

\begin{cor}\label{cor:E1wss}
We have
$$ \E{AF}{}_2 = \E{AF}{}_\infty $$
and $\E{wss}{}_1$ is essentially given by Proposition~\ref{prop:E2AF}.
\end{cor}

\subsection*{The higher differentials in the weight spectral sequence}

We now compute the remaining differentials in the weight spectral sequence.  By \ref{thm:d1}, these are given by
\begin{equation}\label{eq:drwss} [d_{2^r}^{wss} (w^m h_2^{k_2} \cdots h_l^{k_l})] = v_0^{\alpha(m-2^r)-\alpha(m)+1} w^{m-2^r}h_{r+2} h_2^{k_2} \cdots h_l^{k_l}.
\end{equation}
(provided the source and target persist to $\E{wss}{}_{2^r}$).

\begin{prop}[Lellmann-Mahowald \cite{LM}]\label{prop:drwss}
The remaining differentials in the weight spectral sequence are given by
\begin{align*}
[d_{1}^{wss}(v_0^iw^{m})] & = v_0^{i+\nu_2(m)}w^{m-1} h_2, \\
[d_{2^r}^{wss}(w^{2^r a}h_2^{k_2} \cdots h_l^{k_l})] & = w^{2^r(a-1)} h_{r+2} h_2^{k_2} \cdots h_l^{k_l}, \quad k_l \ge 2, \: a \: \mr{odd}. 
\end{align*}
\end{prop}

\begin{proof}
The first formula follows from (\ref{eq:drwss}) and the fact that 
$$ \alpha(m-1) - \alpha(m)+1 = \nu_2(m). $$
The second formula follows from the fact that, for $k_l \ge 2$, the classes
$$ v_0^i w^m h_2^{k_2} \cdots h_l^{k_l} $$
die in the lexigraphical filtration spectral sequence for $i > 0$.  Thus the only possible non-trivial differentials coming from (\ref{eq:drwss}) are
$$ [d_{2^r}^{wss} (w^m h_2^{k_2} \cdots h_l^{k_l})] = w^{m-2^r}h_{r+2} h_2^{k_2} \cdots h_l^{k_l} $$
for $r$ such that the dyadic expansion of $m$ has a $1$ in the $r$th place.  The first of these is $r = \nu_2(m)$.
\end{proof}

In \cite{LM}, Lellmann and Mahowald proceed to deduce a closed form computation for $H^{*,*}(\mc{C})$.  We give our own description of this cohomology, based on the algebraic good complex $\mc{C}^{*,0,*}_{alg}$, in Corollary~\ref{cor:HC}.


\section{The algebraic $\bo$-resolution}\label{sec:alg}

We now construct and analyze an algebraic parallel to the $\bo$-ASS studied so far, the $\bo$-Mahowald spectral sequence ($\bo$-MSS).

\subsection*{Construction of the $\bo$-MSS}

Taking homology of the cofiber sequence
$$ S \rightarrow \bo \rightarrow \br{\bo}, $$
we get a short exact sequence
$$ 0  \rightarrow \FF_2 \rightarrow A \mmod A(1)_*  \rightarrow \br{A \mmod A(1)}_* \rightarrow 0 $$
of $A_*$-comodules, and hence short exact sequences 
$$ 0 \rightarrow \br{A \mmod A(1)}^n_* \rightarrow A \mmod A(1)_* \otimes \br{A \mmod A(1)}^n_*\rightarrow \br{A \mmod A(1)}^{n+1}_*   \rightarrow 0. $$
Piecing together the associated long exact sequences of $\Ext$-groups, and using the change of rings isomorphisms
$$ \Ext^{*,*}_{A_*}(A \mmod A(1)_* \otimes \br{A \mmod A(1)}^n_*) \cong \Ext^{*,*}_{A(1)_*}(\br{A \mmod A(1)}^n_*),  $$
we get an associated ``algebraic $\bo$-resolution''
$$
\xymatrix@C-1em@R-1em{
\Ext_{A_*}^{s,t}(\FF_2) \ar[d] 
& \Ext^{s-1,t}_{A_*}(\br{A \mmod A(1)}_*) \ar[l] \ar[d]
& \Ext^{s-2,t}_{A_*}(\br{A \mmod A(1)}^2_*) \ar[l] \ar[d] 
& \cdots \ar[l]
\\
\Ext^{s,t}_{A(1)_*}(\FF_2)
& \Ext^{s-1,t}_{A(1)_*}(\br{A \mmod A(1)}_*)
& \Ext^{s-2,t}_{A(1)_*}(\br{A \mmod A(1) }^2_*)
}
$$
This gives a spectral sequence (the \emph{$\bo$-Mahowald spectral sequence})
$$ \E{bo}{alg}_1^{n,s,t} = \Ext^{s,t}_{A(1)_*}(\br{A\mmod A(1)}^{n}_*) \Rightarrow \Ext^{s+n,t}_{A_*}(\FF_2) $$
with differentials
$$ d^{alg}_r: \E{bo}{alg}^{n,s,t}_{r} \rightarrow \E{bo}{alg}^{n+r,s-r+1,t}_{{r}}. $$

\subsection*{The $E_1$-term of the $\bo$-MSS}

Let $\ull{B_i}$ denote the $i$th integral Brown-Gitler comodule, the $A_*$-comodule obtained by taking homology of the $i$th integral Brown-Gitler spectrum
$$ \ull{B_i} := H_*(B_i). $$
The motivation behind Theorem~\ref{thm:splitting} is that there is a splitting of $A(1)_*$-comodules:
\begin{equation}\label{eq:algsplitting}
 \br{A \mmod A(1)}^n_* \cong_{A(1)_*} \bigoplus_{\abs{I} = n} \Si^{4\norm{I}} \ull{B_I}
 \end{equation}
where $I = (i_1, \ldots, i_n)$ with each $i_l > 0$ and 
$$ \ull{B_I} := \ull{B_{i_1}} \otimes \cdots \otimes \ull{B_{i_n}}. $$
The $\Ext_{A(1)_*}$-groups of these comodules are given by
$$ \Ext^{*,*}_{A(1)_*}(\ull{B_I}) \cong \E{ass}{}_2(b_I) \oplus V_I $$
where the graded $\FF_2$-vector space $V_I$ has cohomological degree zero:
$$ V_I \subseteq \Ext^{0,*}_{A(1)_*}(\ull{B_I}). $$
Thus the evil subcomplex $(V^{*,*}, d_1^{evil})$ with
$$
V^{n,*} := \bigoplus_{\abs{I} = n} \Si^{4\norm{I}} V_I
$$
is a subcomplex of $(\E{bo}{alg}^{*,0,*}, d_1^{alg})$.
We will define $(\mc{C}^{*,*,*}_{alg}, d_1^{alg,good})$ to be the quotient complex, where
\begin{equation}\label{eq:Calg}
\mc{C}^{n,s,t}_{alg} := \bigoplus_{\abs{I} = n} \E{ass}{}^{s,t}_2(\Si^{4\norm{I}} b_I).
\end{equation}
Analogous to Section~\ref{sec:E1}, we will denote elements of $\mc{C}_{alg}^{*,*,*}$
using notation
$ x(I)$
for $x$ an element of $\E{ass}{}_2^{*,*}(b_I)$
using notation
$$ x = v_0^k w^m h_1^a $$
where $w$ is the formal expression 
$$ w := v_1^2/v_0^2. $$
The following proposition is an immediate consequence of Theorem~\ref{thm:d1}.

\begin{prop}\label{prop:d1alg}
	The differential $d_1^{good, alg}$ in the complex $\mc{C}_{alg}^{*,*,*}$ is given by the formula
	\begin{multline*}
	 d_1^{good,alg}(w^m(i_1, \ldots, i_n)) = \sum_{\substack{m' + m'' = m \\
	 \alpha(m')+\alpha(m'') = \alpha(m)}} w^{m'}(m'', i_1, \ldots, i_n) \\ + 
	\sum_k \sum_{\substack{i_k'+i_k'' = i_k \\ \alpha(i_k') + \alpha(i_k'') = \alpha(i_k)}} w^m(i_1, \ldots, i_k', i_k'', \ldots, i_n) \end{multline*}
	and the fact that the differential is $v_0$ and $h_1$-linear.
\end{prop}

\subsection*{The algebraic weight spectral sequence}

Analogous to the weight spectral sequence of Section~\ref{sec:wss}, we can set up an algebraic weight spectral sequence to compute $H^{*,*,*}(\mc{C}_{alg})$. 
Just as in the topological case, we endow $\mc{C}^{*,*,*}_{alg}$ with a decreasing filtration by weight ({$wt$}), where we define
\[ wt(x(I)) = \norm{I}. \]
Proposition~\ref{prop:d1alg} implies that $d_1^{good, alg}$ does not decrease weight.  There is a resulting spectral sequence
\[ \E{wss}{alg}_0^{n,s,t,w} := [\mc{C}^{n,s,t}_{alg}]_{wt=w} \Rightarrow H^{n,s,t}(\mc{C}_{alg}) \]
with differentials
\[ d_r^{wss, alg}: \E{wss}{alg}_r^{n,s,t,w} \rightarrow \E{wss}{alg}_r^{n+1, s,t, w+r}. \]
From Proposition~\ref{prop:d1alg}, we see that $d_0^{wss, alg}$ is given by
\begin{equation}\label{eq:d0wssalg} 
d_0^{wss, alg}(x(i_1, \ldots, i_n)) = \sum_k \sum_{
\substack{i_k'+i_k'' = i_k \\ \alpha(i_k')+\alpha(i_k'') = \alpha(i_k)}} x(i_1, \ldots, i_k', i_k'', \ldots, i_n). 
\end{equation}

Note that this is precisely the formula for $d_0^{AF}$ of Section~\ref{sec:wss} {(see (\ref{eq:d0AF})).  Therefore, the same proof for Proposition~\ref{prop:E2AF} yields the following:

\begin{prop}\label{prop:E1wssalg}
	An additive basis for $\E{wss}{alg}_1$ is given by elements
	$$ x h_2^{k_2} h_3^{k_3}h_4^{k_4} \ldots $$
	indexed by $K = (k_2, k_3, \ldots)$, detected by $x(I[K])$ where 
	$$ I[K] = (\underbrace{1, \ldots ,1}_{k_2}, \underbrace{2, \ldots ,2}_{k_3}, \underbrace{4, \ldots ,4}_{k_4}, \ldots). $$
	Here, the elements $x$ run through a basis of $\E{ass}{}_2(b_{I[K]})$.
\end{prop}

By Proposition~\ref{prop:d1alg}, the remaining differentials in the algebraic weight spectral sequence are given by
\begin{align*}
d_{2^r}^{wss, alg}(w^{2^r a}h_2^{k_2} \cdots h_l^{k_l}) & = w^{2^r(a-1)} h_{r+2} h_2^{k_2} \cdots h_l^{k_l}, \: a \: \mr{odd}. 
\end{align*}
when the source and target persist to $\E{wss}{alg}_{2^r}$.

\begin{lem}\label{lem:drwssalg}
The non-trivial differentials in the algebraic weight spectral sequence are given by
\begin{align*}
d_{2^{r-2}}^{wss, alg}(w^{2^{r-2} a}h_{r'}^{k_{r'}}  \cdots h_l^{k_l}) & = w^{2^{r-2}(a-1)} h_r h_{r'}^{k_{r'}}  \cdots h_l^{k_l}, \: k_{r'} > 0, \: r \le r', \: a \: \mr{odd},
\end{align*}
where $r\geq 2$.
\end{lem}

\begin{proof}
The differentials
\begin{align*}
d_{2^{r-2}}^{wss, alg}(w^{2^{r-2} a}h_{r'}^{k_{r'}}  \cdots h_l^{k_l}) & = w^{2^{r-2}(a-1)} h_r h_{r'}^{k_{r'}}  \cdots h_l^{k_l}, \: k_{r'} > 0, \: r > r', \: a \: 
\mr{odd}
\end{align*}
never get a chance to run, because the sources of these potential differentials are targets of the shorter differentials:
\[d_{2^{r'-2}}^{wss, alg}(w^{2^{r-2} a+2^{r'-2}}h_{r'}^{k_{r'}-1}  \cdots h_l^{k_l})  = w^{2^{r-2} a}h_{r'}^{k_{r'}}  \cdots h_l^{k_l}. \qedhere\]
\end{proof}

To describe $\E{wss}{alg}_\infty$, we need some notation.  Consider the ``$\bo$-pattern'':
\begin{center}
\includegraphics[width=0.7\linewidth]{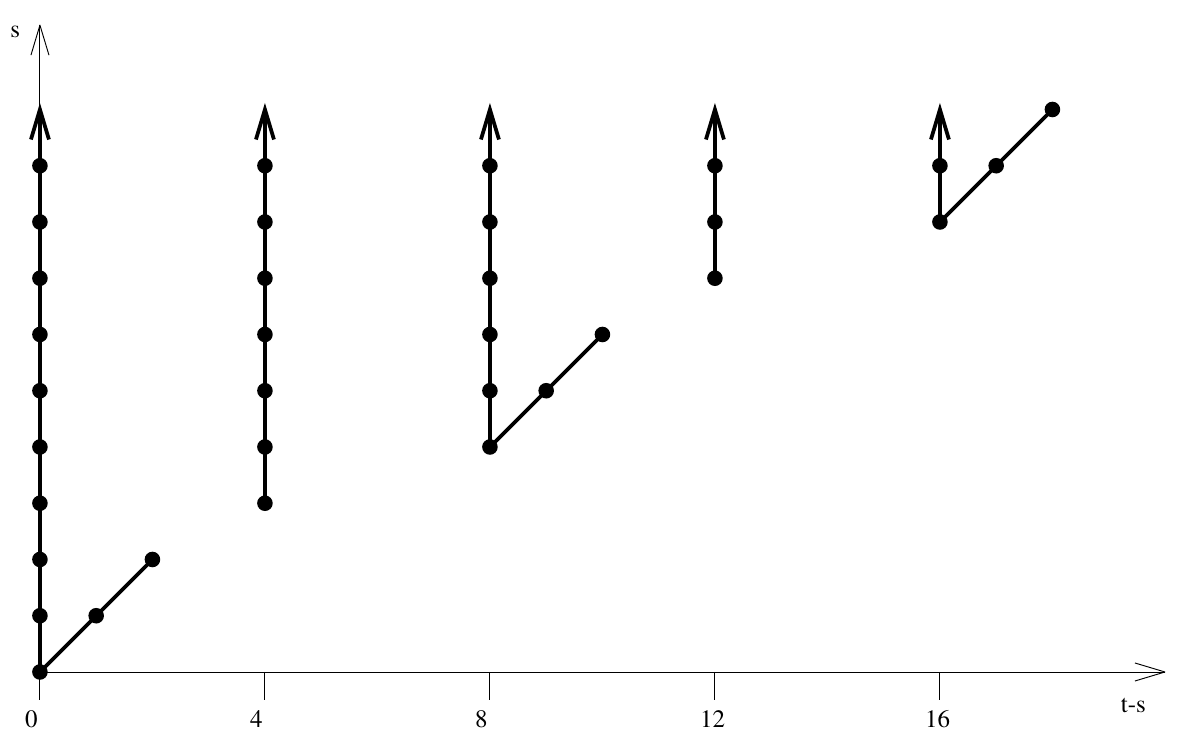}
\end{center}
We shall use $\bo_i[j]$ to denote the sub-pattern where we only include every $(2^i)$th $v_0$ tower, and truncate these $v_0$-towers to have length $j$.  For example, $\bo_1[3]$ is given by:
\begin{center}
\includegraphics[width=0.7\linewidth]{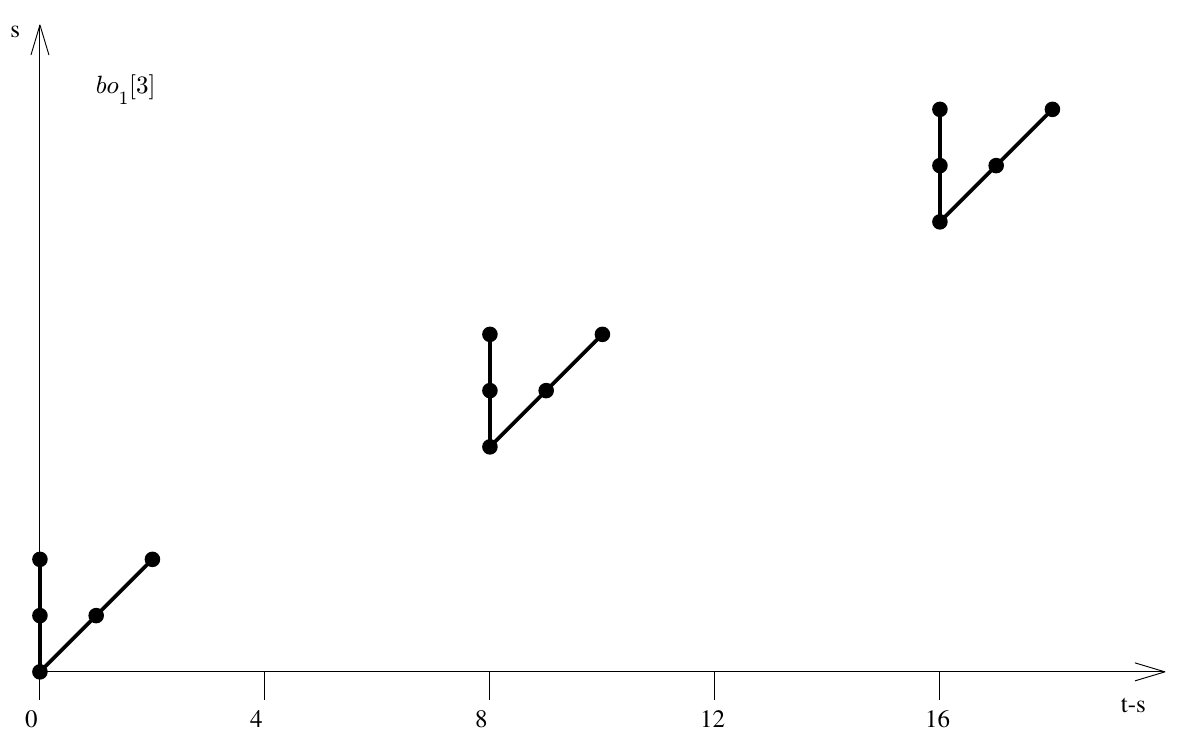}
\end{center}
whereas $\bo_2[7]$ is given by:
\begin{center}
\includegraphics[width=0.7\linewidth]{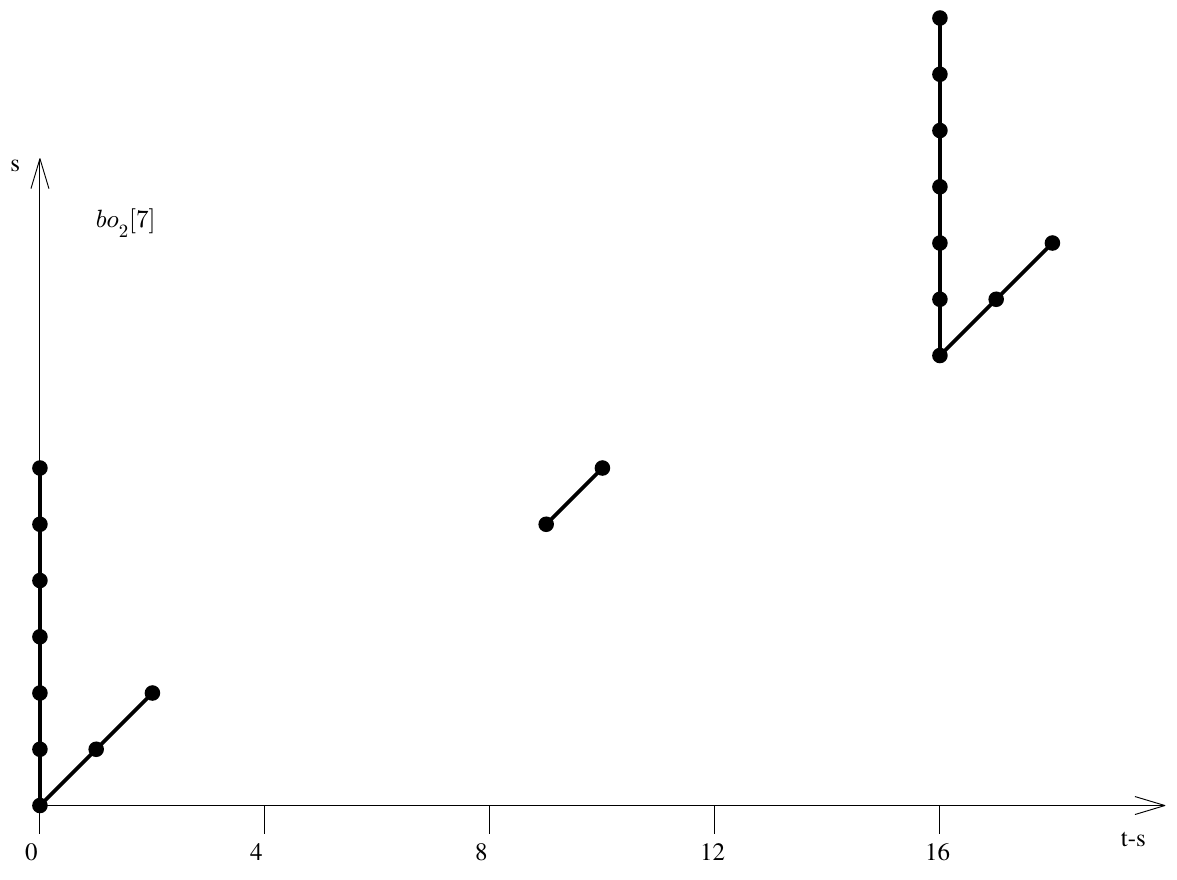}
\end{center}
We also need to consider an analog of these patterns for $\bsp$. Let $\bsp_1[3]$ denote the following pattern:
\begin{center}
\includegraphics[width=0.7\linewidth]{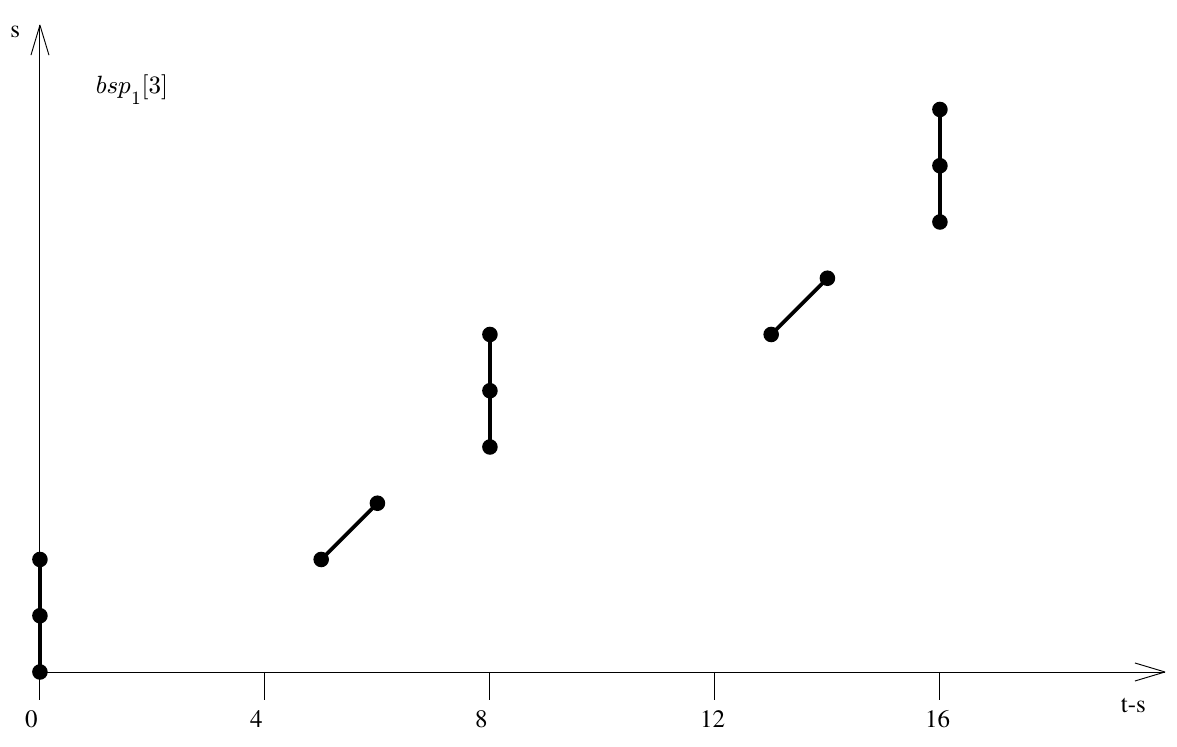}
\end{center}
Finally, we let $\bo_i[j]^{\bra{k}}$ denote the ``$k$th Adams cover'' of the pattern $\bo_i[j]$.  The pattern $\bo_1[3]^{\bra{2}}$ is:
\begin{center}
\includegraphics[width=0.7\linewidth]{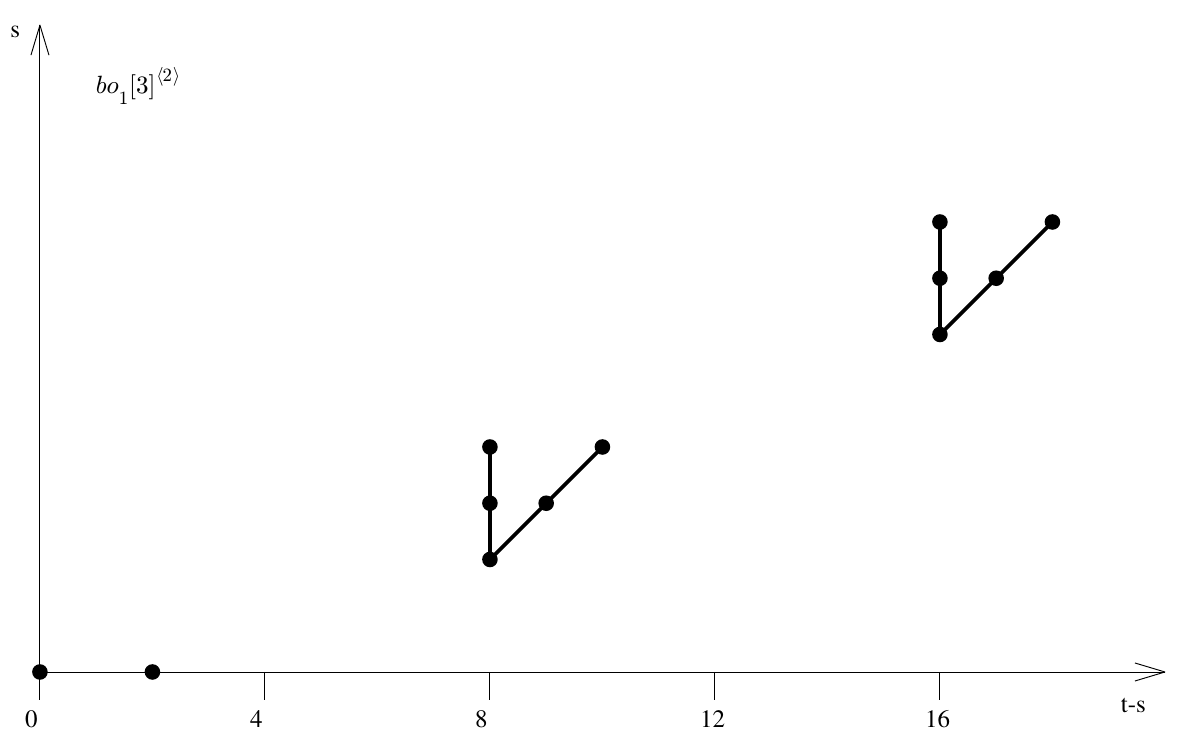}
\end{center}

With regard to these patterns, $H^{*,*,*}(\mc{C}_{alg})$ is described by
the following theorem.

\begin{thm}\label{thm:HCalg}
The groups $H^{*,*,*}(\mc{C}_{alg})$ have a basis given by
\begin{align*}
v_0^i \cdot 1, \: & i \ge 0, \\
x \cdot h_r^{k_r} \ldots h_l^{k_l}, \: & r \ge 2, \: k_r \ge 1, \: \mr{all \: other} \: k_j \ge 0, 
\end{align*}
where
$$ K = (0, \ldots, 0, k_r, \ldots, k_l) $$
and $x$ runs through a basis of
$$
\begin{cases}
\bsp_1[3]^{\bra{2\norm{I[K]}-\abs{I[K]}-1}}, & r = 2, k_2 \: \mr{odd}, \\
\bo_{r-1}[2^r-1]^{\bra{2\norm{I[K]}-\abs{I[K]}}}, & \mr{otherwise}.
\end{cases}
$$
\end{thm}

\begin{rmk}
We note for the reader's convenience that
\begin{align*}
\norm{I[K]} & = k_2+2k_3+4k_4+\cdots, \\
\abs{I[K]} & = k_2 + k_3 + k_4 + \cdots .
\end{align*}
\end{rmk}

\begin{proof}
Proposition~\ref{prop:E1wssalg} gives a basis of $\E{wss}{alg}_1$ as
$ xh_2^{k_2} \cdots h_l^{k_l} $
where $x$ runs through a basis of
$$
\begin{cases}
\E{ass}{}_2\left(\bsp^{\bra{2\norm{I[K]}-\abs{I[K]}-1}}\right), &  k_2 \: \mr{odd}, \\
\E{ass}{}_2\left( \bo^{\bra{2\norm{I[K]}-\abs{I[K]}}} \right), & \mr{otherwise}.
\end{cases}
$$
Lemma~\ref{lem:drwssalg} implies that if 
$$ K = (0, \ldots, 0, k_r, \ldots, k_l) $$
with $k_r \ge 1$, then only every $2^{r-1}$ $v_0$-tower
in these patterns persist to $\E{wss}{alg}_\infty$.  Moreover, these $v_0$-towers are truncated by differentials emanating from the classes
$$ y h_r^{k_r-1}\cdots h_l^{k_l}. $$
Assuming $k_2$ is even, ($k_2$ odd is a similar separate case) and letting 
$$ K' = (0, \ldots, 0, k_r-1, \ldots, k_l) $$
we compute the height of the $v_0$-towers to be
\[ (2\norm{I[K]}-\abs{I[K]}) - (2\norm{I[K']}-\abs{I[K']}) = 2^r-1. \qedhere\]
\end{proof}

By Proposition~\ref{prop:drwss}, the differentials in the algebraic weight spectral sequence (in $\bo$-filtration $\ge 2$) are simply a restriction of the differentials in the topological weight spectral sequence. We can thus read off $H^{n,*}(\mc{C})$ from $H^{n,0,*}(\mc{C}_{alg})$ for $n \ge 2$.  The resulting description of these cohomology groups is given below.

\begin{cor}\label{cor:HC}
The cohomolology of the topological good complex $\mc{C}^{*,*}$ is given in degrees $0$ and $1$ by:
\begin{align*}
H^{0,*}(\mc{C})  = & \ZZ\{1\} \oplus \ZZ/2\{v_1^{4i}h_1^j \: : \: i \ge 0, \: j = 1,2 \}, \\
H^{1,*}(\mc{C})  = & \ZZ/8 \{ v_1^{4i}h_2 \: : \: i \ge 0 \} \\
& \oplus 
\ZZ/2^{\nu_2(i+1)+4}\{ v_1^{4i} \cdot v_0wh_2 \: : \: i \ge 0 \} \\
& \oplus
\ZZ/2\{v_1^{4i}h_1^j \cdot v_0wh_2 \: : \: i \ge 0, \: j = 1,2 \}.
\end{align*}
In degrees $n \ge 2$, these cohomology groups are given by the subspace
$$ H^{n,*}(\mc{C}) \subseteq H^{n,0,*}(\mc{C}_{alg}) $$
generated by elements of the form
$$ x \cdot h_2^{k_2} \cdots h_l^{k_l}, \quad k_l \ge 2 $$
for $x$ as in Theorem~\ref{thm:HCalg}.
\end{cor}

In particular, we can derive a vanishing line for $H^{*,*}(\mc{C})$.

\begin{cor}\label{cor:HCvanishing}
We have
$$ H^{n,t}(\mc{C}) = 0 $$
for $n \ge 1/5(t-n) + 2$.
\end{cor}

\begin{proof}
By Corollary~\ref{cor:HC}, the lowest $t-n$ where a class
$$ x h_r^{k_r} \cdots h_l^{k_l} $$
can contribute to $H^{n,t}(\mc{C})$
is
\begin{align*}
 t-n & = \underbrace{k_r(2^r-1)+\cdots +k_l(2^l-1)}_{\abs{h_2^{k_2}\cdots h_l^{k_l}}} + \underbrace{2}_{\abs{v_1}}(\underbrace{2\norm{I[K]} - \abs{I[K]} - \epsilon}_{\text{Adams cover}} - \underbrace{(2^{r}-2)}_{\text{$v_0$-truncation}}) \\ 
 & = (2^{r+1}k_2 + \cdots + 2^{l+1}k_l) - 3(k_2 + \cdots + k_l) - 2^{r+1}+4-2\epsilon \\
 & \ge (2^{r+1}-3)(k_2 + \cdots + k_l)  - 2^{r+1}+2. 
\end{align*}
where
$$
\epsilon = \begin{cases}
0,  & k_2 \: \mr{even}, \\
1, & k_2 \: \mr{odd}.
\end{cases}
$$
Since
$$ n = k_2 + \cdots + k_l $$
and $r \ge 2$, we deduce that such a class must satisfy
\begin{align*}
n & \le \frac{1}{2^{r+1}-3}(t-n) + \frac{2^{r+1}-2}{2^{r+1}-3} \\
& < \frac{1}{5}(t-n) + 2.
\end{align*}
\end{proof}


\section{$v_1$-periodicity}\label{sec:v1}

In Section~\ref{sec:alg}, we established that there is a long exact sequence
\begin{equation}
 \cdots \rightarrow H^{n,k}(V) \rightarrow \E{bo}{alg}_2^{n,0,k} \xrightarrow{g_{alg}} H^{n,0,k}(\mc{C}_{alg}) \xrightarrow{\partial_{alg}} H^{n+1,k}(V) \rightarrow \cdots
\end{equation}
which assembles $\E{bo}{alg}_2$ out of good and evil classes. 
Furthermore, we computed the good component $H^{*,*,*}(\mc{C})$ in its entirety (Theorem~\ref{thm:HCalg}).  Since the spectral sequence $\{\E{bo}{alg}_r\}$ converges to $\Ext_{A_*}(\FF_2)$, it stands to reason that if we know $\Ext_{A_*}(\FF_2)$ through a range, we should be able to deduce the evil component $H^{*,*}(V)$.  As summarized in the introduction, this is predicated on actually having a means of determining which elements of $\Ext$ are detected by good and evil classes.  The Dichotomy Principle (Theorem~\ref{thm:dichotomy}) will relate this to $v_1$-periodicity.  We review the definition and properties of $v_1$-periodic Ext groups in this section.

\subsection*{$v_1$-periodic Ext groups}

Our approach to $v_1$-periodic Ext follows that of  \cite{MahowaldShick}, \cite{DavisMahowaldv1}, which is based on the Adams periodicity theorem \cite{AdamsP}.  Let $M$ be an $A(n)_*$-comodule.  For $n \ge 2$, Adams produced elements
$$ v_1^{2^n} \in \Ext^{2^n, 3\cdot 2^n}_{A(n)_*}(\FF_2) $$
which are appropriately compatible under the maps
$$ \Ext^{*,*}_{A(n+1)_*}(\FF_2) \rightarrow \Ext^{*,*}_{A(n)_*}(\FF_2). $$
For any $A(n)_*$-comodule $M$, we can form the localized Ext groups
$$ v_1^{-1}\Ext^{s,t}_{A(n)_*}(M). $$
We then define
\begin{equation}\label{eq:v1Ext} v_1^{-1}\Ext^{s,t}_{A_*}(M) := \varprojlim_n v_1^{-1}\Ext^{s,t}_{A(n)_*}(M).
\end{equation}
For certain $M$, the inverse system in (\ref{eq:v1Ext}) stabilizes for large $n$ for fixed bidegrees $(s,t)$. When this is the case, the localized Ext groups are manageable.

One class of comodules for which this is true are those whose duals are $A(0)$-free.  Indeed, as pointed out by \cite{DavisMahowaldv1}, Adams proves:

\begin{thm}[Adams \cite{AdamsP}]\label{thm:AdamsP}
Suppose that $M$ is a connective $A_*$-comodule whose dual is free over $A(0)$, and suppose $n \ge 2$. 
\begin{enumerate}
\item For $t-s < 3s$, the map
$$ v_1^{2^n}: \Ext^{s,t}_{A(n)_*}(M) \rightarrow \Ext^{s+2^n, t+3\cdot 2^n}_{A(n)_*}(M) $$
is an isomorphism.

\item For $s > 0$ and $t-s < 2s + 2^{n+1} - 5$, the map
$$ \Ext_{A_*}^{s,t}(M) \rightarrow \Ext_{A(n)_*}^{s,t}(M) $$
is an isomorphism.
\end{enumerate}
\end{thm}

In other words, $\Ext_{A(n)}$ is $v_1^{2^n}$-periodic above the line of slope $1/3$ through the origin. Further, there is a sequence of lines of slope $1/2$ above which $\Ext_{A_*}$ agrees with $\Ext_{A(n)_*}$. We deduce that above the line of slope $1/3$, every element of $\Ext_{A_*}$ is $v_1^{2^n}$-periodic for some $n$ which depends on which ``band'' of slope $1/2$ it lies in. The typical picture one draws of this (and indeed a similar picture appears in \cite{DavisMahowaldv1}) is:

\begin{center}
\includegraphics[width=0.7\linewidth]{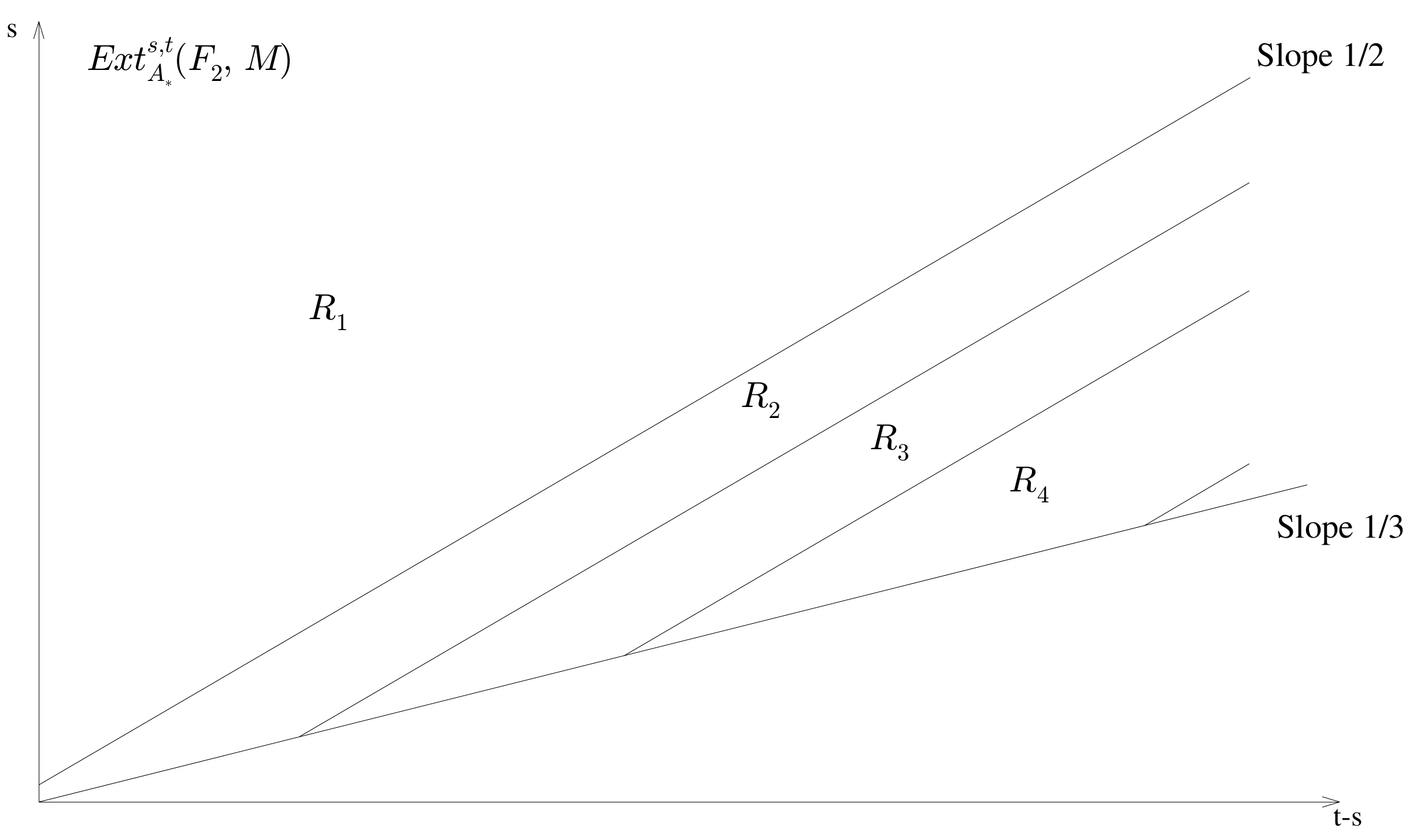}
\end{center}
In the figure above, the Ext groups in the region $R_1$ are all trivial.  In the region $R_n$, for $n \ge 2$ the Ext groups are $v_1^{2^n}$-periodic, and isomorphic to the corresponding $\Ext_{A(n)_*}$-groups.  The corresponding $v_1$-periodic Ext groups take the form:
\begin{center}
\includegraphics[width=0.7\linewidth]{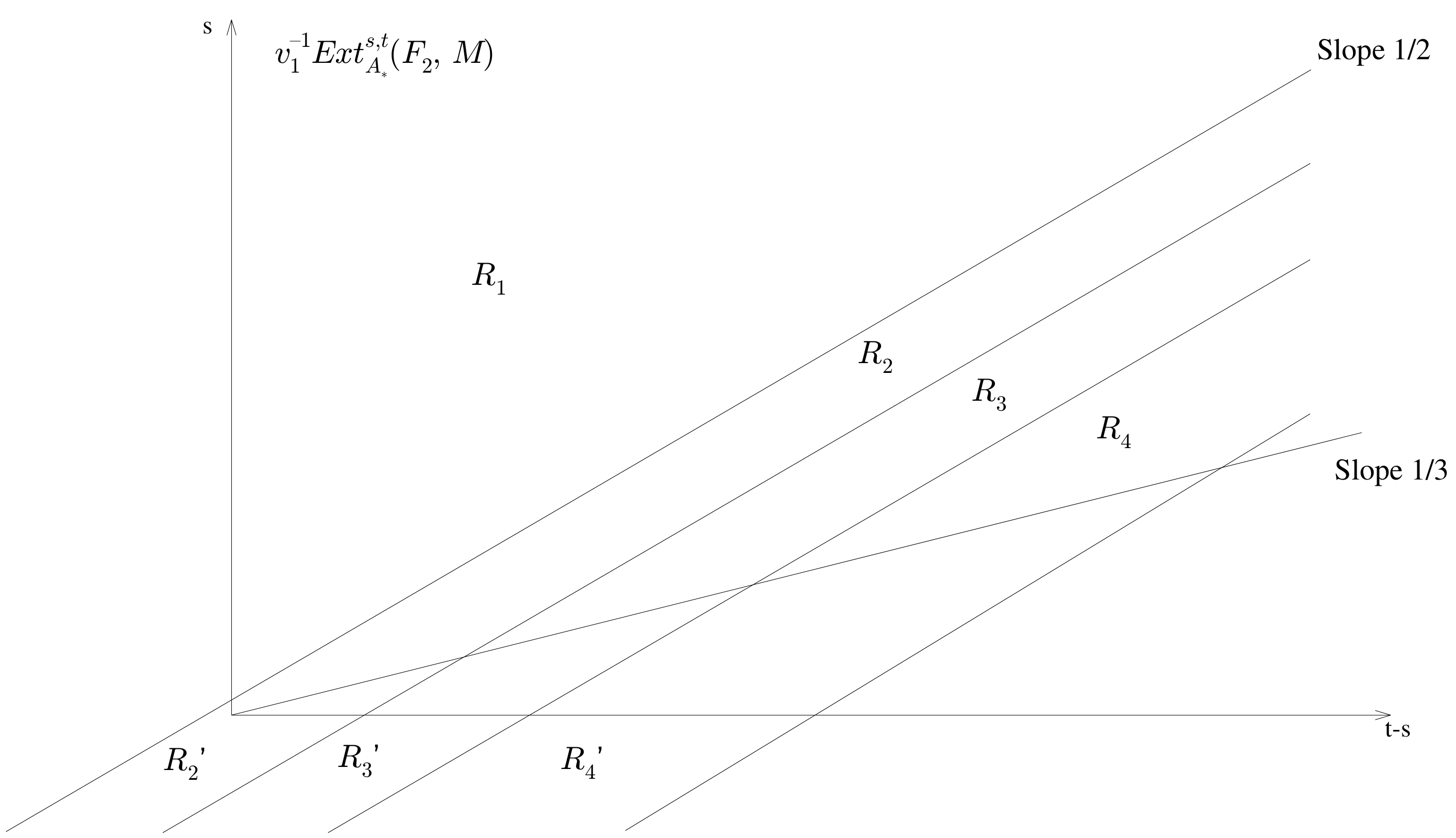}
\end{center}
Here the periodic Ext groups in the region $R_n'$ are obtained by interpolating back the values of the Ext groups of the region $R_n$. In particular, the values of $v_1^{-1}\Ext_{A(m)_*}$ stabilize in the region $R_n \cup R_n'$ for $m \ge n$.

The comodule $\FF_2$ is not free over $A(0)$, but $\br{A\mmod A(0)}_*$ is. The following remark establishes an explicit relationship between their Ext groups.
\begin{rmk}\label{rmk:AmodA0SES}
The short exact sequence
\begin{equation}\label{eq:AmodA0SES}
 0 \rightarrow \FF_2 \xrightarrow{i} A\mmod A(0)_* \xrightarrow{p} \AA0 \rightarrow 0
\end{equation}
induces a long exact sequence of Ext groups.  Since by change of rings we have
\[ \Ext^{*,*}_{A_*}(A \mmod A(0)_*) \cong \Ext^{*,*}_{A(0)_*}(\FF_2) = \FF_2[v_0], \]
the connecting homomorphism
\[ \Ext^{s-1,t}_{A_*}(\AA0) \xrightarrow{\bar{\partial}} \Ext_{A_*}^{s,t}(\FF_2) \]
is an isomorphism for $t \ne s$.  

The short exact sequence (\ref{eq:AmodA0SES}) is the homology of the cofiber sequence
$$ S \rightarrow H\ZZ \rightarrow \br{H\ZZ}. $$ 
\end{rmk}

\begin{prop}\label{prop:v1ext}
$\quad$
\begin{enumerate} 
\item For $ 0 < t -s < 2s+2^{n+1}-6 $,
there are well defined homomorphisms (compatible as $n$ varies)
\[
v_1^{2^n}: \Ext^{s,t}_{A_*}(\FF_2) \rightarrow \Ext^{s+2^n,t+3\cdot 2^n}_{A_*}(\FF_2). 
\]
For $t -s < 3s-4$
these homomorphisms are isomorphisms.

\item The maps
$$ \Ext^{s,t}_{A_*}(\FF_2) \rightarrow \Ext^{s,t}_{A(n)_*}(\FF_2) $$
are isomorphisms for $t-s < -s+2^{n+1}$.
\end{enumerate}
\end{prop}  

\begin{proof}
Statement (1) follows from the isomorphism of Remark~\ref{rmk:AmodA0SES} by applying Theorem~\ref{thm:AdamsP} to the comodule $\br{A\mmod A(0)}_*$, which is free over $A(0)$. Statement (2) follows from the fact that in degrees less than $2^{n+1}$, the map
$$ A_* \rightarrow A(n)_* $$
is an isomorphism.
\end{proof}

\begin{rmk}\label{rmk:onefifth}
The line $t -s < 3s-2$ of Proposition~\ref{prop:v1ext} can be upgraded to a line of slope $1/5$ (see \cite[Thm~3.4.6]{Ravenel}).
\end{rmk}

Proposition~\ref{prop:v1ext} allows to define $v_1^{-1}\Ext_{A_*}(\FF_2)$ by interpolating $v_1$-periodic families backwards across the slope $1/3$ line. 
\footnote{There is one exception: the $v_0$-tower in $t-s = 0$ is $v_1$-periodic, with ``infinite period''.}

\begin{cor}\label{cor:stabilization}
For fixed $s,t$, the map
\[ v_1^{-1}\Ext^{s,t}_{A_*}(\FF_2) \rightarrow \Ext^{s,t}_{A(n)_*}(\FF_2) \]
is an isomorphism for $n \gg 0$.
\end{cor}

\subsection*{The $v_1$-periodic $\bo$-MSS}

Davis and Mahowald \cite{DavisMahowaldv1} considered a $v_1$-periodic version of the $\bo$-MSS:
\[ v_1^{-1} \! \E{bo}{alg}_1^{n,s,t}(M) = v_1^{-1} \Ext^{s,t}_{A(1)_*}(\br{A\mmod A(1)}^{n}_* \otimes M) \Rightarrow v_1^{-1}\Ext^{s+n,t}_{A_*}(M). \]
They use the stabilization of $v_1^{-1}\Ext_{A(n)_*}$ in fixed bidegrees to prove the convergence of this spectral sequence for comodules $M$ which are free over $A(0)$.  Corollary~\ref{cor:stabilization} implies that the Davis-Mahowald convergence argument also applies to the case where $M = \FF_2$.

\begin{lem}\label{lem:v1E2boalg}
The maps
\[
v_1^{-1} \! \E{bo}{alg}^{n,s,t}_2 \xrightarrow{v_1^{-1}g_{alg}} H^{n,s,t}(v_1^{-1} \mc{C}_{alg})
\]
are isomorphisms.
\end{lem}

\begin{proof}
The evil subcomplex
\[ V^{n,*} \subseteq \Ext^{0,*}_{A(1)_*}(\br{A\mmod A(1)}^{n}_*) \]
is $v_1$-torsion, so we have an isomorphism
\[ v_1^{-1} \! \E{bo}{alg}_1^{n,s,t} = v_1^{-1} \Ext^{s,t}_{A(1)_*}(\br{A\mmod A(1)}^{n}_*) \xrightarrow{\cong} v_1^{-1}\mc{C}^{n,s,t}_{alg} \]
and hence an isomorphism
\[ v_1^{-1} \! \E{bo}{alg}^{n,s,t}_2 \xrightarrow{\cong}  H^{n,s,t}(v_1^{-1}\mc{C}_{alg}). \qedhere \]
\end{proof}

\section{Comparison with $\br{A\mmod A(0)}$}\label{sec:AA0}

We would like to leverage the results of Section~\ref{sec:v1} to relate $v_1$-periodicity in $\Ext_{A_*}(\FF_2)$ to being ``good'' in a precise manner.  However, as we saw in Section~\ref{sec:v1}, the groups 
$$ \Ext^{s,t}_{A_*}(\AA0) $$
have much better behaved $v_1$-periodic phenomena. For instance, there is a sequence of lines $L_n$ of slope 1/2 above which the Ext groups are entirely $v^{2^n}_1$-periodic and simultaneously isomorphic to the corresponding $\Ext_{A(n)_*}$ groups, where $v_1^{2^n}$ is well defined.

While $\AA0$ and $\FF_2$ have essentially the same Ext groups (see Remark~\ref{rmk:AmodA0SES}), their respective $\bo$-MSS's differ.  In this section, we give a complete dictionary between these two $\bo$-MSS's.  This will allow us to transport results on $v_1$-periodicity from $\AA0$ to $\FF_2$.  There is an added bonus: since the two $\bo$-MSS's are different, we will be able to deduce hidden extensions or differentials in one from non-hidden extensions or differentials \emph{of different length} in the other.

Given any $A_*$-comodule $M$, we can consider the associated $\bo$-MSS
$$ \E{bo}{alg}_1^{n,s,t}(M) = \Ext^{s,t}_{A(1)_*}(\br{A\mmod A(1)}^{n}_* \otimes M) \Rightarrow \Ext^{s+n,t}_{A_*}(M). $$
Suppose that the map 
$$ \E{bo}{alg}_1^{n,s,t}(M) \rightarrow v_1^{-1} \E{bo}{alg}_1^{n,s,t}(M) $$
is injective for $s > 0$.  Then we will define $V^{n,t}(M)$ to be the
kernel of the above map for $s = 0$.  Just as in the case of $M = \FF_2$, the evil subgroup $V^{*,*}(M)$ is a subcomplex with respect to $d_1^{bo,alg}$, and we define the good complex 
$$ (\mc{C}_{alg}^{*,*,*}(M), d_1^{good, alg}) $$ 
to be the quotient complex.  Then there is a long exact sequence
\begin{multline}\label{eq:LESM}
 \cdots \rightarrow H^{n,k}(V(M)) \rightarrow \E{bo}{alg}_2^{n,0,k}(M) \xrightarrow{g^M_{alg}} H^{n,0,k}(\mc{C}_{alg}(M)) \\ \xrightarrow{\partial^M_{alg}} H^{n+1,k}(V(M)) \rightarrow \cdots
\end{multline}

The following lemma is an algebraic analog of the ``generalized connecting homomorphism theorem'' \cite[Thm~2.3.4]{Ravenel}, and its proof is identical to the topological case.

\begin{lem}[Connecting Homomorphism Lemma]\label{lem:MSSSES}
Associated to the short exact sequence
$$ 0 \rightarrow \FF_2 \xrightarrow{i} {A \mmod A(0)_*} \xrightarrow{p} \AA0 \rightarrow 0 $$
we have a sequence of spectral sequences
$$ 
\xymatrix@C-1.5em{
\E{bo}{alg}_1^{n,s,t}(\FF_2) \ar[r]^-{i_*} \ar@{=>}[d] &
\E{bo}{alg}^{n,s,t}_1({A\mmod A(0)_*}) \ar[r]^-{p_*} \ar@{=>}[d] &
\E{bo}{alg}^{n,s,t}_1(\br{A\mmod A(0)}_*) \ar[r]^-{\partial} \ar@{=>}[d] &
\E{bo}{alg}^{n,s+1,t}_1(\FF_2)  \ar@{=>}[d]
\\
\Ext^{n+s, t}_{A_*}(\FF_2) \ar[r]_-{i_*} &
\Ext^{n+s, t}_{A_*}(A\mmod A(0)_*) \ar[r]_-{p_*} &
\Ext^{n+s, t}_{A_*}(\AA0) \ar[r]_-{\bar{\partial}} &
\Ext^{n+s+1, t}_{A_*}(\FF_2) 
}
$$
The top and bottom rows are long exact sequences.  The map $\partial$ induces a map of spectral sequences, converging to $\bar{\partial}$.
\end{lem}

\subsection*{The $\pmb\bo$-MSS for $\pmb{A \mmod A(0)}$}

We now study the spectral sequence
$$ \E{bo}{alg}^{*,*,*}_1(A\mmod A(0)_*) \Rightarrow \Ext^{*,*}_{A_*}(A \mmod A(0)_*) \cong \FF_2[v_0]. $$
The next lemma computes the ``good'' part of the $E_1$-term.

\begin{lem}\label{lem:AmodA0good}
We have
\[\mc{C}_{alg}^{n,*,*}(A \mmod A(0)_*) \cong \FF_2[v_0, w] \otimes \FF_2[\zeta_1^{4}]^{\otimes n}  \]
where $w = v_1^2/v_0^2$.
\end{lem}

\begin{proof}
By change of rings, we have
\begin{align*}
 \E{bo}{alg}_1^{n,*,*}(A \mmod A(0)_*) & \cong 
 \Ext^{*,*}_{A(1)_*} ((A \mmod A(1)_*)^{\otimes n} \otimes A\mmod A(0)_*) \\
 & \cong \Ext^{*,*}_{A_*} (A \mmod A(1)_* \otimes (A \mmod A(1)_*)^{\otimes n} \otimes A\mmod A(0)_*) \\
 & \cong \Ext^{*,*}_{A(0)_*} (A \mmod A(1)_* \otimes (A \mmod A(1)_*)^{\otimes n})
 \end{align*}
 The latter is computed using Margolis homology. Using the fact that in 
 \[ A\mmod A(1)_* \cong \FF_2[\zeta_1^{4}, \zeta_2^{2}, \zeta_3, \ldots  ] \]
 we have
 \[ Q_0^* \zeta_i = \zeta^2_{i-1} \] 
 we deduce that the Margolis homology is given by
\[ H_*((A \mmod A(1)_*)^{\otimes n+1}; Q_0^*) = \FF_2[\zeta_1^{4}]^{\otimes n+1}. \]
Identifying the first factor of $\FF_2[\zeta^4_1]$ with $\FF_2[w]$,
 we have
\[ \E{bo}{alg}_1^{n, *, *} = \FF_2[v_0, w] \otimes \FF_2[\zeta_1^{4}]^{\otimes n} \oplus V^{n,*}(A \mmod A(0)_*). \qedhere \]
\end{proof}

Just as in Section~\ref{sec:alg}, we will use the notation 
\[ v_0^i w^m (i_1, \ldots, i_n) := v_0^i w^m \zeta_1^{4i_1}\otimes  \cdots \otimes \zeta_1^{4i_n} \]
to denote a generic element of $\E{bo}{alg}_1(A \mmod A(0)_*)$.
From the map of good complexes
\[ i_* : \mc{C}_{alg}^{*,*,*}(\FF_2) \rightarrow \mc{C}_{alg}^{*,*,*}(A \mmod A(0)_*) \]
(and the fact that $d_1^{good, alg}$ is $v_0$-linear), we find the formula for $d_1^{good, alg}$ is identical to that of Proposition~\ref{prop:d1alg}.
Let
\[ \E{wss}{alg}_0^{n, s, t, w}(A \mmod A(0)_*) \Rightarrow H^{n, s, t}(\mc{C}_{alg}(A \mmod A(0)_*)) \]
denote the associated algebraic weight spectral sequence.  Just as in Proposition~\ref{prop:E1wssalg}, we have 
\[ \E{wss}{alg}^{*,*,*,*}_1(A \mmod A(0)_*) = \FF_2[v_0, w, h_2, h_3, \cdots ]. \]
Just as in Lemma~\ref{lem:drwssalg}, the non-trivial differentials in the algebraic weight spectral sequence are given by
\begin{align*}
d_{2^{r-2}}^{wss, alg}(w^{2^{r-2} a}h_{r'}^{k_{r'}}  \cdots h_l^{k_l}) & = w^{2^{r-2}(a-1)} h_r h_{r'}^{k_{r'}}  \cdots h_l^{k_l}, \: k_{r'} > 0, \: r \le r', \: a \: \mr{odd}. 
\end{align*}
We deduce
\[ H^{*,*,*}(\mc{C}_{alg}(A \mmod A(0)_*)) = \FF_2[v_0]. \]

\begin{prop}\label{prop:boMSSAmodA0collapse}
The $\bo$-MSS for $A \mmod A(0)_*$ collapses at $E_2$.
\end{prop}

\begin{proof}
With the exception of $\FF_2[v_0]$, all of the good classes are targets and sources of $d_1$-differentials.  Since
$d_r^{bo, alg}$ changes $s$-degree for $r > 1$, and the evil classes are concentrated in $s = 0$, we deduce that
\[ H^{*,*}(V(A \mmod A(0)_*)) = 0. \]
The result follows.
\end{proof}

\subsection*{The structure of $\pmb{\E{bo}{alg}_1(\AA0)}$}

In this subsection, we will describe the $E_1$-term $\E{bo}{alg}_1(\AA0)$ in terms of $\E{bo}{alg}_1(\FF_2)$.  We shall find that while these $E_1$-terms are different, their relationship admits a complete description.

The primary tool in this analysis is the long exact sequence
\begin{multline}\label{eq:boMSSLESE1}
\cdots \xrightarrow{} \E{bo}{alg}^{n, s, t}_1(\FF_2) 
\xrightarrow{i} \E{bo}{alg}^{n, s, t}_1(A \mmod A(0)_*) \xrightarrow{p}
\E{bo}{alg}^{n, s, t}_1(\AA0) \\
\xrightarrow{\partial} \E{bo}{alg}^{n, s+1, t}_1(\FF_2) \xrightarrow{} \cdots 
\end{multline}
These long exact sequences decompose into a direct sum of long exact sequences
\begin{multline*}
\cdots \rightarrow \Ext^{s,t}_{A(1)_*}(\ull{B_I}) \xrightarrow{i} \Ext^{s,t}_{A(1)_*}(\ull{B_I} \otimes A \mmod A(0)_*) \xrightarrow{p} \Ext^{s,t}_{A(1)_*}(\ull{B_I} \otimes \AA0) \\ \xrightarrow{\partial} \Ext^{s+1, t}_{A(1)_*}(\ull{B_I}) \rightarrow \cdots
\end{multline*}
using the decomposition
\[ \E{bo}{alg}^{n, s, t}_1(M) = \bigoplus_{\abs{I} = n}  \Ext^{s, t}_{A(1)_*}(\Sigma^{4\norm{I}} \ull{B_I} \otimes M) \]
induced by the splitting (\ref{eq:algsplitting}).

We first examine the behavior of the good classes.  A useful schematic is depicted below.
\begin{center}
\includegraphics[width=\linewidth]{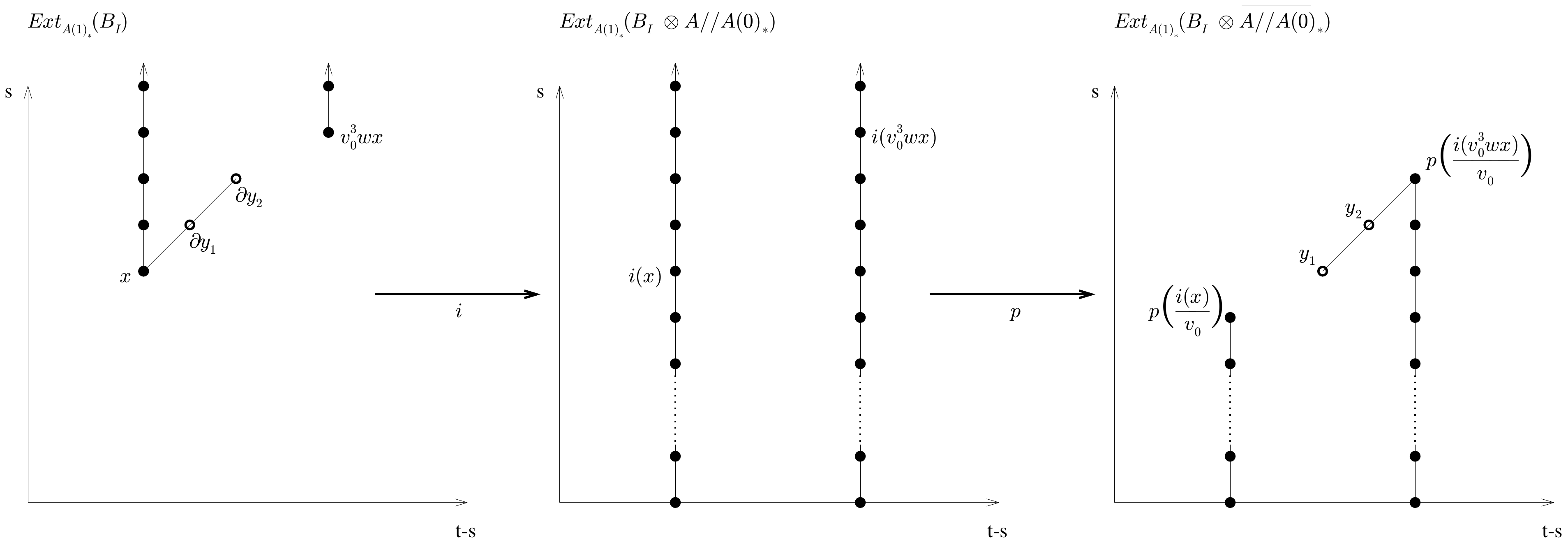}
\end{center}
In the figure above, there are two kinds of good classes in $\Ext_{A(1)_*}(\ull{B_I})$: those which are parts of $v_0$-towers, and those which are $v_0$-torsion.  We shall refer to the former as \emph{$v_0$-good} (marked with solid dots in the above figure) and the latter as \emph{$h_1$-good} (so that there is a basis of $\E{bo}{alg}_1(\FF_2)$ consisting of $v_0$-good, $h_1$-good, and evil classes).  The good part of $\E{bo}{alg}_1(A \mmod A(0)_*)$ was computed in the last subsection, as depicted in the middle chart above.  The righthand chart is then deduced by the long exact sequence.  We see that in $\E{bo}{alg}_1(\br{A\mmod A(0)}_*)$, the $v_0$-towers get turned upside down (we will call these classes $v_0$-good as well), while the $h_1$-good classes get transported by the boundary homomorphism  $\partial$ (which we also call $h_1$-good).

The extension 
\begin{equation}\label{eq:extension} 
h_1 y_{2} = p \left( \frac{i(v_0^3wx)}{v_0} \right)
\end{equation}
in the figure above is of special importance.  One way of deducing it is to observe that associated to the cofiber sequence
\begin{equation}\label{eq:kocofiber}
\bo \rightarrow \bo \wedge H\ZZ \rightarrow \bo \wedge \br{H\ZZ}
\end{equation}
there is a sequence of $v_1$-localized Adams spectral sequences
\[\xymatrix@C-1em{
v_1^{-1}\Ext^{s,t}_{A(1)}(\ull{B_I}) \ar[r]^-{i} \ar@{=>}[d] & 
v_1^{-1}\Ext^{s,t}_{A(1)}(\ull{B_I} \otimes A \mmod A(0)_*) \ar[r]^-p \ar@{=>}[d] &
v_1^{-1}\Ext^{s,t}_{A(1)}(\ull{B_I} \otimes \AA0) \ar@{=>}[d]
\\
\pi_{t-s} (\KO \wedge B_I) \ar[r] & 
\pi_{t-s} ((\KO \wedge B_I)_\QQ) \ar[r] &  
\pi_{t-s} (\Sigma M_1(\KO \wedge B_I)) 
}
\]
(where $M_1(-)$ is the first monochromatic layer). The localized spectral sequences converge as indicated because of the following lemma.

\begin{lem}
Inverting the Bott element in the cofiber sequence (\ref{eq:kocofiber}) yields the cofiber sequence
\[ \KO \rightarrow \KO_\QQ \rightarrow \Sigma M_1\KO \]
\end{lem}

\begin{proof}
The computations earlier in this section specialize to show that the Adams spectral sequence
$$ \Ext_{A(1)_*}(A\mmod A(0)_*) \Rightarrow \bo_*H\ZZ $$
collapses to give
$$ \bo_*H\ZZ \cong \ZZ[w] \oplus \text{$v_1$-torsion} $$
with $w = v_1^2/4$.  Inverting the Bott element $v_1^4$, we get
$$ \KO_* H\ZZ = \QQ[v_1^{\pm 2}]. $$
The map
$$ \KO_* \rightarrow \KO_*H\ZZ $$
is easily computed on the level of the $v_1$-localized Adams $E_2$-terms, and is seen to be a rational isomorphism.  The result follows.
\end{proof}

The extension (\ref{eq:extension}) follows from:
\begin{itemize}
\item $\KO \wedge B_I$ is a suspension of $\KO$,
\item the Brown-Comenetz dual $IM_1(\KO)$ is the Gross-Hopkins dual $I_1\KO$,
\item the Gross-Hopkins dual of $\KO$ is well known to be a suspension of $\KO$ (see, for example, \cite[Cor.9.1]{HeardStojanoska}). 
\end{itemize}

With the language introduced in the discussion above, we have (for $M = \FF_2$, $A \mmod A(0)_*$, or $\AA0$) decompositions
\[ \E{bo}{alg}^{n,*,*}_1(M) \cong \mc{C}^{n,*,*}_{v_0}(M) \oplus \mc{C}^{n, *, *}_{h_1}(M) \oplus V^{n,*}(M) \]
into $v_0$-good, $h_1$-good, and evil components.  Note that for each of these $M$ we have
\begin{align*}
\mc{C}^{n,s,t}_{v_0}(M) & = 
\begin{cases}
\mc{C}^{n,s,t}_{alg}(M), & t-s \equiv 0 \mod 4, \\
0, & \mr{otherwise},
\end{cases}
\\
\mc{C}^{n,s,t}_{h_1}(M) & = 
\begin{cases}
\mc{C}^{n,s,t}_{alg}(M), & t-s \not\equiv 0 \mod 4, \\
0, & \mr{otherwise},
\end{cases}
\end{align*}
(and for $M = A \mmod A(0)_*$ we have $\mc{C}^{*,*,*}_{h_1}(M) = 0$).
The complete structure of these components is summarized in the following proposition.

\begin{prop}\label{prop:MSSSESE1}
There are short exact sequences
\begin{gather*}
 0 \rightarrow V^{n,*}(\FF_2) \oplus \mc{C}^{n,0,*}_{h_1}(\FF_2) \xrightarrow{i} V^{n,*}(A \mmod A(0)_*) \xrightarrow{p} V^{n,*}(\AA0) \rightarrow 0 \\
 0 \rightarrow \mc{C}^{n,*, *}_{v_0}(\FF_2) \xrightarrow{i} \mc{C}^{n, *,*}_{v_0}(A \mmod A(0)_*) \xrightarrow{p} \mc{C}^{n,*,*}_{v_0}(\AA0) \rightarrow 0
\end{gather*}
and isomorphisms
\[ \mc{C}_{h_1}^{n, s, *}(\AA0) \xrightarrow[\cong]{\partial} \mc{C}^{n, s+1, *}_{h_1}(\FF_2). \]
Thus, for every $v_0$-good tower
\[ \{ v_0^jx \}_{j \ge 0} \subset \mc{C}^{n, *, *}_{v_0}(\FF_2) \]
generated by $x \in \mc{C}^{n, s, t}_{v_0}$
there is a corresponding truncated $v_0$-good tower
\[ \left\{ p\left(\frac{i(x)}{v_0^j}\right) \right\}_{1 \le j \le s} \subset \mc{C}^{n, *, *}_{v_0}(\br{A\mmod A(0)}_*). \]
Furthermore, in $\mc{C}_{alg}^{n,*,*}(\AA0)$ we have
\[ h_1 \partial^{-1}(h_1^2 x) = p\left( \frac{i(v_0^3wx)}{v_0} \right). \]
\end{prop}

\begin{proof}
This proposition mostly follows from the preceding discussion, using the long exact sequence (\ref{eq:boMSSLESE1}) and Lemma~\ref{lem:AmodA0good}.
In particular, the second short exact sequence follows from the fact that the map $i$ induces an inclusion
$$ i : \mc{C}^{n,*,*}_{v_0}(\FF_2) \hookrightarrow \mc{C}^{n, *,*}_{v_0}
(A \mmod A(0)_*). $$
The map $i$ vanishes on $\mc{C}^{n,s,t}_{h_1}(\FF_2)$ for $s > 0$ for dimensional reasons, and thus the latter must be in the image of $\partial$.  However, $i$ can map classes in $\mc{C}^{n, 0, t}_{h_1}(\FF_2)$ to evil classes in $V^{n,t}(A \mmod A(0)_*)$, and in fact must do so in an injective fashion, because there are no non-trivial classes in $\E{bo}{alg}^{n, -1, t}(\AA0)$ for $\partial$ to map into $\ker i$.  The $h_1$-relation follows from the Gross-Hopkins duality argument discussed above.
\end{proof}

To state the structure of $\mc{C}_{alg}^{*,*,*}(\AA0)$ more clearly, we remark that the Ext groups
\begin{align*}
\Ext^{s,t}_{A(1)_*}(\AA0) & = \E{ass}{}_2^{s,t}(\bo \wedge \br{H\ZZ}) \\
\Ext^{s,t}_{A(1)_*}(\ull{B_1} \otimes \AA0) & = \E{ass}{}_2^{s,t}(\bsp \wedge \br{H\ZZ})
\end{align*}
take the following form ($v_1$-torsion classes ommitted):
\begin{center}
\includegraphics[width=1\linewidth]{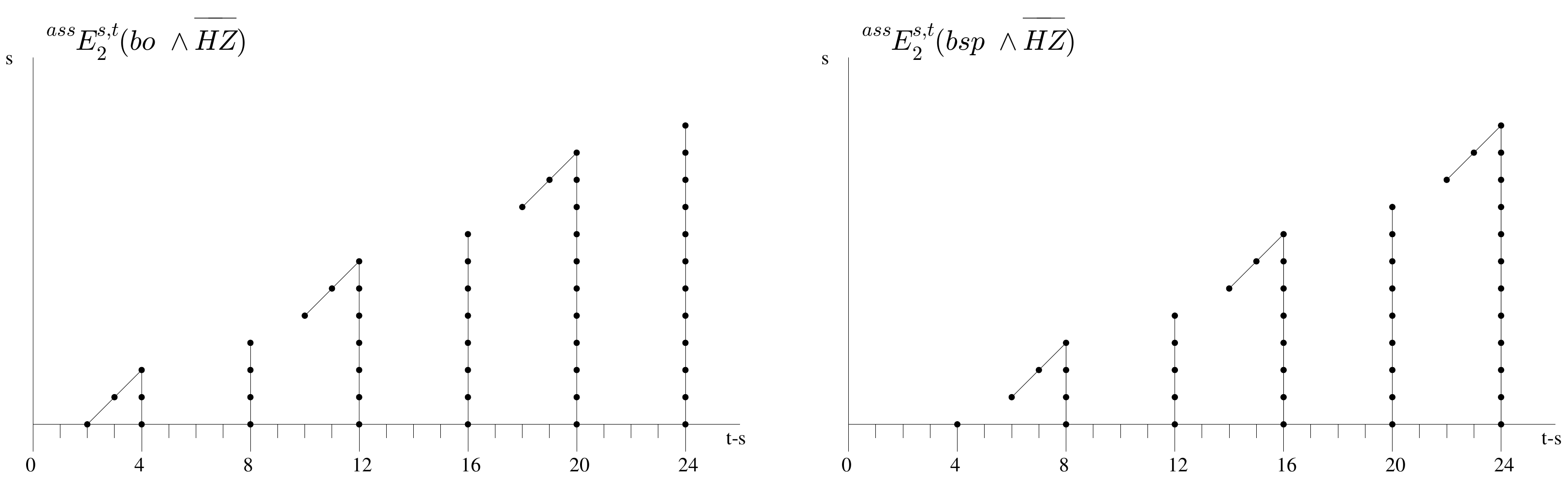}
\end{center}

\begin{cor}
There are decompositions
\[ \mc{C}^{n,s,t}_{alg}(\AA0) \cong \bigoplus_{\abs{I} = n} \E{ass}{}^{s,t}_2(\Sigma^{4\norm{I}}b_I \wedge \br{H\ZZ}) \]
where
\[ \E{ass}{}^{s,t}_2(b_I \wedge \br{H\ZZ}) = 
\begin{cases}
\E{ass}{}_2^{s,t}((\bo \wedge \br{H\ZZ})^{\bra{2\norm{I}-\alpha(I)}}), & \norm{I} \: \mr{even}, \\
\E{ass}{}_2^{s,t}((\bsp \wedge \br{H\ZZ})^{\bra{2\norm{I}-\alpha(I)-1}}), & \norm{I} \: \mr{odd}.
\end{cases} \]
\end{cor}

\subsection*{The structure of $\pmb{\E{bo}{alg}_2(\AA0)}$}  We now turn to understanding the $E_2$-page of the $\bo$-MSS for $\br{A\mmod A(0)}_*$.  We shall let
\[ w^m(i_1, \ldots, i_n) \in \E{bo}{alg}_1^{n, 0, 4\norm{I}+4m}(\AA0) \]
denote the image under $p$ of the element of the same name in $\E{bo}{alg}_1(A \mmod A(0)_*)$.  Since $p$ induces a map of spectral sequences, we get the same formula for $d_1^{good, alg}$ on the classes above as in $\E{bo}{alg}_1(A\mmod A(0)_*)$ (the formula from \ref{prop:d1alg}).
The formula for $d_1^{good,alg}$ on all of $\mc{C}_{alg}^{*,*,*}(\br{A\mmod A(0)}_*)$ follows from the fact that it is $v_0$ and $h_1$-linear.  We use a weight spectral sequence:
\[ \E{wss}{alg}^{n, s, t, w}(\AA0) \Rightarrow H^{n, s, t}(\mc{C}_{alg}(\br{A \mmod A(0)_*})). \]
The computation of $\E{wss}{alg}_1(\AA0)$ is just as in Proposition~\ref{prop:E1wssalg}:

\begin{prop}\label{prop:E1wssAmodA0}
	An additive basis for $\E{wss}{alg}_1(\br{A \mmod A(0)_*})$ is given by elements
	$$ x h_2^{k_2} h_3^{k_3}h_4^{k_4} \ldots $$
	indexed by $K = (k_2, k_3, \ldots)$, detected by $x(I[K])$ where 
	$$ I[K] = (\underbrace{1, \ldots ,1}_{k_2}, \underbrace{2, \ldots ,2}_{k_3}, \underbrace{4, \ldots ,4}_{k_4}, \ldots). $$
	Here, the elements $x$ run through a basis of $\E{ass}{}_2(b_{I[K]} \wedge \br{H\ZZ})$.
\end{prop}

The remaining differentials in the weight spectral sequence are induced by the map $p$ from the weight spectral sequence for $A \mmod A(0)_*$ (the $h_1$-good classes in $\E{wss}{alg}_1(\AA0)$ are all permanent cycles):
\begin{multline}\label{eq:drwssAmodA0}
d_{2^{r-2}}^{wss, alg}(w^{2^{r-2} a}h_{r'}^{k_{r'}}  \cdots h_l^{k_l}) = w^{2^{r-2}(a-1)} h_r h_{r'}^{k_{r'}}  \cdots h_l^{k_l}, \\
k_{r'} > 0, \: r \le r', \: a \: \mr{odd}. 
\end{multline}
The figure below illustrates the relationship between the weight spectral sequences (and consequently the corresponding $\E{bo}{alg}_2$-terms) for $\FF_2$, $A \mmod A(0)_*$, and $\AA0$.

\begin{figure}[h]
\centering
\includegraphics[width=1\linewidth]{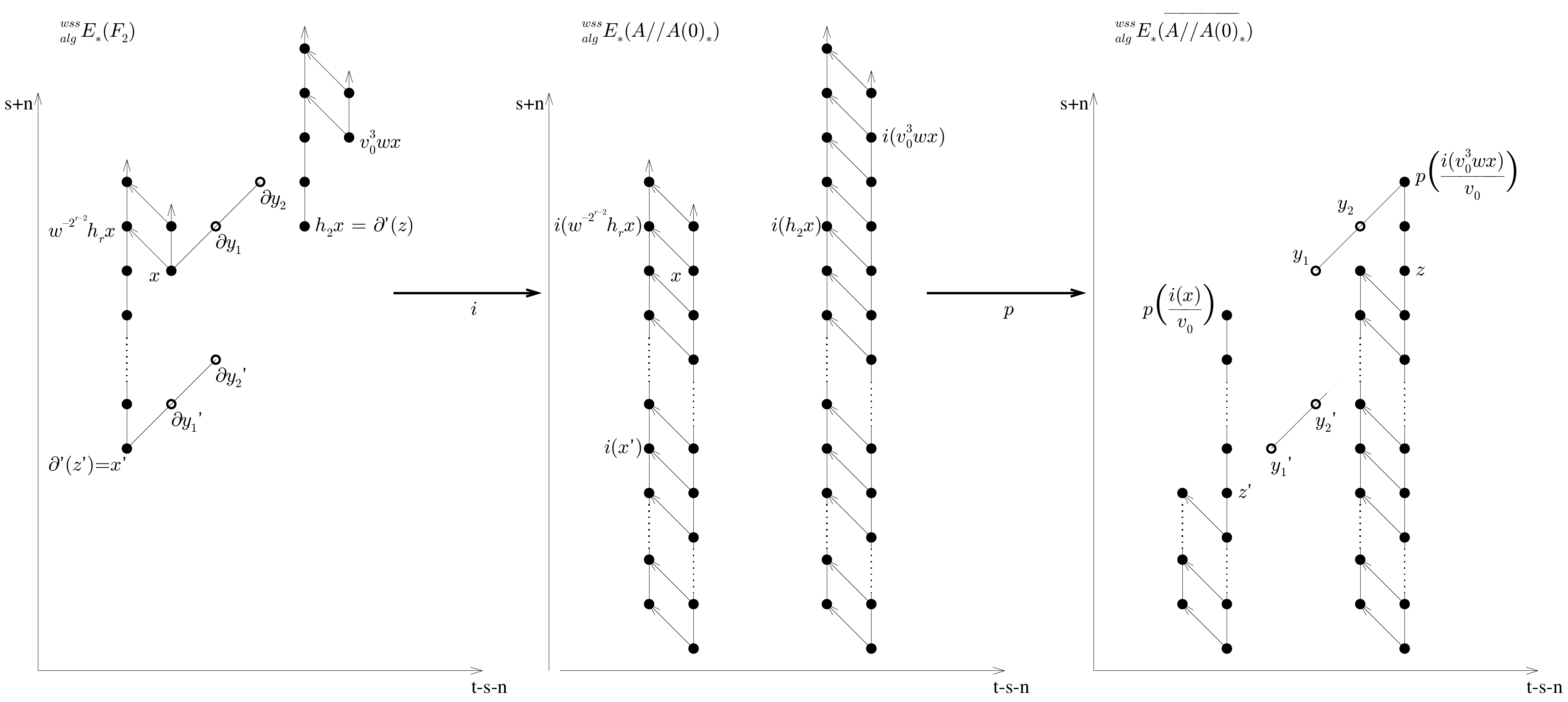}
\caption{}
\label{fig:E2AmodA0}
\end{figure}

We see that there is a bijective correspondence between the truncated $v_0$-towers that comprise the $v_0$-good classes in $\E{bo}{alg}_2(\FF_2)$ and $\E{bo}{alg}_2(\AA0)$, with the notable feature that their respective $\bo$-filtrations differ by $1$ while their $s$-degrees are identical.  Similarly, there is a bijective correspondence between the $h_1$-good classes in each of these $E_2$-terms, but with identical $\bo$-filtrations and $s$-degrees which differ by $1$ (\emph{except for $h_1$-good classes in $\E{bo}{alg}_2^{*,0,*}(\FF_2)$, for which there is no corresponding class in $\E{bo}{alg}_2(\br{A\mmod A(0)}_*)$}).    In the notation of Figure~\ref{fig:E2AmodA0} above, we have
$$ h_1 \cdot x' = \partial y_1' \quad \mr{in} \: \E{bo}{alg}_2(\FF_2) $$
while 
$$ h_1 \cdot z' = 0 \quad \mr{in} \: \E{bo}{alg}_2(\br{A \mmod A(0)_*}) $$
(with $y'$ lying in higher $\bo$ filtration), and also we have
$$ h_1 \cdot y_2 = p\left( \frac{i(v_0^3wx)}{v_0} \right) \quad \mr{in} \: \E{bo}{alg}_2(\br{A \mmod A(0)}) $$
while 
$$ h_1 \cdot \partial y_2 = 0 \quad \mr{in} \: \E{bo}{alg}_2(\FF_2) $$
(with 
$v_0^2 h_2 x$ lying in higher $\bo$ filtration).  The map $\partial'(-)$ used in Figure~\ref{fig:E2AmodA0} is one of the two connecting homomorphisms associated to long exact sequences associated to the short exact sequences of Proposition~\ref{prop:MSSSESE1}:
\begin{align*}
\partial': & H^{n,s,t}(\mc{C}_{v_0}(\AA0)) \rightarrow H^{n+1,k}(\mc{C}_{v_0}(\FF_2)), \\
\partial': & H^{n,t}(V(\AA0)) \rightarrow H^{n+1,t}(V(\FF_2) \oplus \mc{C}^{*,0,*}_{h_1}(\FF_2)). 
\end{align*}
A complete description of $\E{bo}{alg}_2(\br{A\mmod A(0)}_*)$ in terms of $\E{bo}{alg}_2(\FF_2)$ is provided by the following proposition.

\begin{prop}\label{prop:E2boMSSAA0}
The maps $\partial$, $\partial'$ induce isomorphisms
\begin{align*}
H^{n,s,t}(\mc{C}_{h_1}(\AA0)) &  \xrightarrow[\cong]{\partial} H^{n, s+1, t}(\mc{C}_{h_1}(\FF_2)), 
\\
H^{n,s,t}(\mc{C}_{v_0}(\AA0)) & \xrightarrow[\cong]{\partial'} H^{n+1,s, t}(\mc{C}_{v_0}(\FF_2)),
\quad s \ne t,
\\
H^{n,t}(V^{*,*}(\AA0)) & \xrightarrow[\cong]{\partial'} H^{n+1, t}(V^{*,*}(\FF_2)\oplus \mc{C}^{*,0,*}_{h_1}(\FF_2)).
\end{align*}
Here, the groups $H^{n, t}(V(\FF_2)\oplus \mc{C}^{*,0,*}_{h_1}(\FF_2))$ are determined by
$$
H^{n, t}(V(\FF_2)\oplus \mc{C}^{*,0,*}_{h_1}(\FF_2)) =
\begin{cases}
\E{bo}{alg}^{n,0,t}_2(\FF_2), & t \not\equiv 0 \mod 4, \\
H^{n,t}(V(\FF_2)), & \mr{otherwise}.
\end{cases}
$$
The connecting homomorphism $\partial_{alg}^{\AA0}$
in the long exact sequence (\ref{eq:LESM}) is determined for $t \equiv 0 \mod 4$ by the commutative diagram
$$
\xymatrix{
H^{n,0,t}(\mc{C}_{v_0}(\AA0)) \ar[r]^{\partial_{alg}^{\AA0}} \ar[d]_{\partial'}^{\cong} &
H^{n+1,t}(V(\AA0)) \ar[d]^{\partial'}_{\cong}
\\
H^{n+1,0,t}(\mc{C}_{{v_0}}(\FF_2)) \ar[r]_{\partial_{alg}^{\FF_2}} &
H^{n+2,t}(V(\FF_2))
}
$$
and for $t \not\equiv 0 \mod 4$ by the commutative diagram
$$
\xymatrix{
H^{n,0,t}(\mc{C}_{h_1}(\AA0)) \ar[r]^{\partial_{alg}^{\AA0}} \ar[d]_{\partial}^{\cong} &
H^{n+1,t}(V(\AA0)) \ar[d]^{\partial'}_{\cong}
\\
H^{n,1,t}(\mc{C}_{h_1}(\FF_2)) \ar@{=}[d] &
H^{n+2,t}(V(\FF_2) \oplus \mc{C}^{*,0,*}_{h_1}(\FF_2)) \ar@{=}[d]
\\
\E{bo}{alg}_2^{n,1,t}(\FF_2) \ar[r]_{d_2^{bo,alg}} &
\E{bo}{alg}_2^{n+2,0,t}(\FF_2)
}
$$
\end{prop}

\begin{proof}
The isomorphisms in the first part of the proposition follow from Proposition~\ref{prop:MSSSESE1} and the fact that we have (Proposition~\ref{prop:boMSSAmodA0collapse})
\begin{align*}
H^{n,s,t}(\mc{C}_{alg}(A \mmod A(0)_*)) & = 0, \quad \text{unless $n = 0$ and $s = t$}, \\
H^{n,t}(V(A \mmod A(0)_*)) & = 0.
\end{align*}
The identification of $H^{*,*}(V(\FF_2)\oplus \mc{C}_{h_1}^{*,0,*}(\FF_2))$ simply follows from the fact that for $t \not\equiv 0 \mod 4$ we have
\[ \E{bo}{alg}_1^{n,{0}, t}(\FF_2) = V^{n,t}(\FF_2)\oplus \mc{C}_{h_1}^{n,0,t}(\FF_2). \]
The diagram computing $\partial^{\AA0}_{alg}$ for $t \equiv 0 \mod 4$ follows from the fact that for such $t$ there is a $3\times 3$ diagram of short exact sequences of chain complexes
$$
\xymatrix{
V^{\bullet,t}(\FF_2) \ar[r] \ar[d] &
V^{\bullet,t}(A \mmod A(0)_*) \ar[r] \ar[d] &
V^{\bullet,t}(\AA0) \ar[d] 
\\
\E{bo}{alg}_1^{\bullet,0,t}(\FF_2) \ar[r] \ar[d] &
\E{bo}{alg}_1^{\bullet,0,t}(A \mmod A(0)_*) \ar[r] \ar[d] &
\E{bo}{alg}_1^{\bullet,0,t}(\AA0) \ar[d] 
\\
\mc{C}_{v_0}^{\bullet,0,t}(\FF_2) \ar[r] &
\mc{C}_{v_0}^{\bullet,0,t}(A \mmod A(0)_*) \ar[r] &
\mc{C}_{v_0}^{\bullet,0,t}(\AA0) 
}
$$
(note there is no sign in the commutativity of the resulting diagram of connecting homomorphisms because we are working in characteristic 2).
The proof of the commutativity of the diagram computing $\partial_{alg}^{\AA0}$ for $t \not\equiv 0 \mod 4$ will be deferred to the next subsection (Case~(1) of Theorem~\ref{thm:AA0}, $r = 1$).
\end{proof}

\subsection*{Higher differentials and hidden extensions in the $\bo$-MSS for $\pmb{\AA0}$.}

We have so far established a dictionary between classes in the $\bo$-MSS's for $\AA0$ and $\FF_2$:
\begin{align*}
x \: \text{is $h_1$-good} \quad \Leftrightarrow \quad & \partial(x) \: \text{is $h_1$-good} \\
\\
x \: \text{is not $h_1$-good} \quad \Leftrightarrow \quad & \partial'(x) \: \text{is not $h_1$-good} \\
& \text{or $\partial'(x)$ is $h_1$-good and lies in $s = 0$.}
\end{align*}

We will now extend this dictionary to all of the higher differentials. For the purpose of the statement of the next theorem, a non-trivial class $x \in \E{bo}{alg}_r(\AA0)$ will be regarded as $h_1$-good if and only if $\partial(x) \ne 0 \in \E{bo}{alg}_r(\FF_2)$.

\begin{thm}\label{thm:AA0}
Suppose that $z$ and $z'$ are classes in the $\bo$-MSS for $\AA0$.   \begin{enumerate}

\item
If $z$ is $h_1$-good and $z'$ not $h_1$-good, then 
$$ d^{bo, alg}_r (z) = z' \quad \Leftrightarrow \quad d^{bo,alg}_{r+1}(\partial(z)) = \partial'(z'). $$

\item
If both $z$ and $z'$ are $h_1$-good, we have
$$ d_r^{bo,alg}(z) = z'  \quad \Leftrightarrow \quad d^{bo,alg}_r(\partial(z)) = \partial(z'). $$ 

\item 
If $z$ is not $h_1$-good and $z'$ is $h_1$-good, we have
$$ d_{r}^{bo,alg}(z) = z' \quad \Leftrightarrow \quad d_{r-1}^{bo,alg}(\partial'(z)) = \partial(z'). $$

\item
If neither $z$ nor $z'$ are $h_1$-good, we have
$$ d_r^{bo,alg}(z) = z'  \quad \Leftrightarrow \quad d^{bo,alg}_r(\partial'(z)) = \partial'(z'). $$ 
\end{enumerate}
\end{thm}

\begin{proof}
This theorem is proven using a combination of the Connecting Homomorphism Lemma (CHL) (Lemma~\ref{lem:MSSSES}) and an adaptation of the Geometric Boundary Theorem (GBT) \cite[Lem. A.4.1]{goodEHP} to our algebraic setting --- the $\bo$-MSS's associated to the short exact sequence of $A_*$-modules:
$$ 0 \rightarrow \FF_2 \xrightarrow{i} A \mmod A(0)_* \xrightarrow{p} \AA0 \rightarrow 0. $$ 
In general, the GBT has a great deal of potential ambiguity.  However, in our case much of it goes away as the $\bo$-MSS for $A \mmod A(0)_*$ collapses at $E_2$, where it is virtually acyclic (Proposition~\ref{prop:boMSSAmodA0collapse}).

Suppose first that $z$ is $h_1$-good, so $\partial(z) \ne 0$.  Suppose
there is a non-trivial differential
\begin{equation}\label{GBT1}
d_r(z) = z' \ne 0
\end{equation}
and $\partial(z') = 0$ (i.e. $z'$ is either $v_0$-good or evil).  Then we apply Case~(2) of the GBT to (\ref{GBT1}) deduce that there is a differential
$$ d_{r+1}(\partial(z)) = \partial'(z') \ne 0. $$
This establishes:
\begin{quote}
(1) $\quad$ If $z$ is $h_1$-good and $z'$ not $h_1$-good, then 
$$ d_r (z) = z' \ne 0 \quad \Leftrightarrow \quad d_{r+1}(\partial(z)) = \partial'(z') \ne 0. $$
\end{quote}
Suppose however that the class $z'$ of (\ref{GBT1}) is not $h_1$-good.  Then it follows from the CHL that we have 
$$ d_r(\partial(z)) = \partial(z'). $$
This implies
\begin{quote}
(2') $\quad$ If both $z$ and $z'$ are $h_1$-good, we have
$$ d_r(z) = z' \ne 0  \quad \Rightarrow \quad d_r(\partial(z)) = \partial(z'). $$ 
\end{quote}
Furthermore, we deduce
\begin{quote}
(2.5) $\quad$ If $z$ is an $h_1$-good permanent cycle, then $\partial(z)$ is a permanent cycle.
\end{quote}  

Suppose now that $z$ is a non-trivial class which is not $h_1$-good, so $\partial(z) = 0$.  By (\ref{eq:LESM}), we deduce that there is a $y \in \E{bo}{alg}_1(A \mmod A(0)_*)$ so that $p(y) = z$.  Then there is a non-trivial differential (Proposition~\ref{prop:boMSSAmodA0collapse}) 
$$ d_1(y) = y' \ne 0. $$
We apply the GBT to this differential.  Note that by the definition of the connecting homomorphism $\partial'$, there is a representative for $\partial'(z)$ in $\E{bo}{alg}_{1}(\FF_2)$ (which we abusively also call $\partial'(z)$) so that
$$ i(\partial'(z)) = y'. $$
Case (2) of the GBT then implies that if there is a non-trivial differential
\begin{equation}\label{eq:GBT2}
d_{r}(z) = z' \ne 0
\end{equation}
with $z'$ $h_1$-good,
then there is a differential
$$ d_{r-1}(\partial'(z)) = \partial(z'). $$
Thus we have shown
\begin{quote}
(3') $\quad$ If $z$ is not $h_1$-good and $z'$ is $h_1$-good, we have
$$ d_{r}(z) = z' \ne 0 \quad \Rightarrow \quad d_{r-1}(\partial'(z)) = \partial(z'). $$
\end{quote}
Suppose however that $z'$ of (\ref{eq:GBT2}) is not $h_1$-good.  Then we are in Case~(3) of the GBT, and we have
$$ d_r(\partial'(z)) = \partial'(z'). $$  
We have shown
\begin{quote}
(4') $\quad$ If $z$ and $z'$ are not $h_1$-good, we have
$$ d_{r}(z) = z' \ne 0 \quad \Rightarrow \quad d_{r}(\partial'(z)) = \partial'(z'). $$
\end{quote}
Finally, suppose that $z$ is a permanent cycle.  Then Cases~(4)-(5) of the GBT imply that $\partial'(z)$
is a permanent cycle.  Thus we have shown
\begin{quote}
(4.5) $\quad$ If $z$ is a permanent cycle which is not $h_1$-good, then $\partial'(z)$ is a permanent cycle.
\end{quote}

Suppose inductively that we have established (2), (3), (4) for all $r < r_0$. Since we have also established (1), we have then established our dictionary between $d_r$ differentials in the $\bo$-MSS for $\AA0$ for $r < r_0$ and corresponding differentials in the $\bo$-MSS for $\FF_2$.  The differentials on the right-hand side of (2'), (3'), and (4') could in principle be trivial, but only if their targets were hit by non-trivial shorter differentials in the $\bo$-MSS for $\FF_2$.  This would violate our inductive hypothesis.  We conclude inductively that the differentials on the right-hand side of (2'), (3'), and (4') are actually non-trivial.  This, combined with (2.5) and (4.5), upgrades these statements to (2), (3), and (4) of the theorem.
\end{proof}


\section{The agathokakological spectral sequence}\label{sec:rules}

At the end of the day we would like to deduce the evil classes $H^{*,*}(V)$ (in a range) directly from the known quantities $H^{*,*,*}(\mc{C}_{alg})$ and $\Ext^{*,*}_{A_*}(\FF_2)$.  This is somewhat confusing, as the relationship of these three quantities occurs through the combination of a spectral sequence (the $\bo$-MSS) and a long exact sequence (\ref{eq:algLES}). To mitigate this complication, we introduce a variant of the $\bo$-MSS, called the \emph{agathokakological spectral sequence} (AKSS), which combines the three quantities directly:
$$ \E{akss}{alg}_{1+\epsilon}^{*,*,*} = H^{*,*}(V) \oplus H^{*,*,*}(\mc{C}_{alg})\Rightarrow \Ext_{A_*}^{*,*}(\FF_2). $$
(The indexing of this spectral sequence is unorthodox, and will be explained.)

\subsection*{The construction of the spectral sequence}

For convenience, fix a splitting of the short exact sequence
\begin{equation}\label{eq:SES} 
0 \rightarrow V^{n,*} \rightarrow \E{bo}{alg}^{n,*,*}_1 \xrightarrow{g_{alg}} \mc{C}^{n,*,*}_{alg} \rightarrow 0.
\end{equation}
Recall that with respect to this splitting, we can decompose elements $x \in \E{bo}{alg}_1$ as
$$ x = x_{evil} + x_{good} $$
with $x_{evil} \in V$ and $x_{good} \in \mc{C}_{alg}$.  While (\ref{eq:SES}) does \emph{not} split as a short exact sequence of chain complexes, it does
introduce a micrograding on the filtration $n$-layer of the $\bo$-MSS.
We will regard the evil subcomplex $V^{n,*}$ as being in filtration $n + \epsilon$, where $\epsilon$ is regarded as being infinitesimal.  The new grading is on the ordered set
$$ \mb{N}^\epsilon := \{ n + \alpha \epsilon \: : \: n \in \mb{N}, \: \alpha \in \{0,1\} \}. $$
A total ordering is defined by
$$ n < n + \epsilon < n+1. $$
The result is a spectral sequence indexed on $\mb{N}^\epsilon$ (see \cite{Matschke}) which we call the \emph{agathokakological spectral sequence} (AKSS):
$$ \{ \E{akss}{alg}^{n+\alpha\epsilon,s,t}_{r + \beta \epsilon} \} \Rightarrow \Ext^{n+s, t}_{A_*}(\FF_2) $$
with
$$\E{akss}{alg}^{n+\alpha \epsilon, s, t}_1 = 
\begin{cases}
\mc{C}^{n, s, t}_{alg}, & \alpha = 0, \\
V^{n, t}, & \alpha = 1, \: s = 0, \\
0, & \text{otherwise}.
\end{cases}
$$
The pages are indexed on the totally ordered set
$$ \mb{N}^{\pm \epsilon} := \{ n + \beta \epsilon \: : \: n \in \mb{N}, \: \beta \in \{-1,0,1\} \} $$
with
$$ n - \epsilon < n < n + \epsilon < n+1. $$
The differentials take the form
\begin{align*}
d^{akss}_{r-\epsilon}:  & \E{akss}{alg}_{r - \epsilon}^{n+ \epsilon, s, t} \rightarrow \E{akss}{alg}_{r- \epsilon}^{n+r, s-r+1, t}, \\
d^{akss}_{r}:  & \E{akss}{alg}_{r}^{n+ \alpha\epsilon, s, t} \rightarrow \E{akss}{alg}_{r}^{n+r +\alpha\epsilon, s-r+1, t}, \\
d^{akss}_{r+\epsilon}:  & \E{akss}{alg}_{r+\epsilon}^{n, s, t} \rightarrow \E{akss}{alg}_{r+ \epsilon}^{n+r +\epsilon, s-r+1, t}.
\end{align*}
Given an element $x \in \E{akss}{alg}_{r+\beta\epsilon}^{n+\alpha\epsilon, s, t}$, we will define
\begin{align*}
n(x) & := n, \\
s(x) & := s, \\
t(x) & := t.
\end{align*}
If $\alpha = 0$, we will refer to the element $x$ as \emph{good}, and if $\alpha = 1$, we refer to $x$ as \emph{evil}.  
The $d_1$-differential 
$$ d^{akss}_{1}: \E{akss}{alg}^{n+\alpha \epsilon, s, t}_1 \rightarrow \E{akss}{alg}^{n+ 1+\alpha\epsilon, s,t}_1
$$
is given by
$$d_1 = 
\begin{cases}
d_1^{good}, & \alpha = 0, \\
d_1^{evil}, & \alpha = 1, s = 0, \\
0, & \text{otherwise}. 
\end{cases}
$$
We therefore have
$$\E{akss}{alg}^{n+\alpha \epsilon, s, t}_{1+\epsilon} = 
\begin{cases}
H^{n, s, t}(\mc{C}_{alg}), & \alpha = 0, \\
H^{n, t}(V), & \alpha = 1, \: s = 0, \\
0, & \text{otherwise}.
\end{cases}
$$
The only nonzero $d_{1+\epsilon}$-differentials are of the form
$$ H^{n, 0, t}(\mc{C}_{alg}) = \E{akss}{alg}^{n, {0}, t}_{1+\epsilon} \xrightarrow{d_{1+\epsilon}} \E{akss}{alg}^{n+1 + \epsilon, 0, t}_{1+\epsilon} = H^{n+1, t}(V), $$
for which we have 
$$ d_{1+\epsilon} = \partial_{alg} $$
where $\partial_{alg}$ is the connecting homomorphism of (\ref{eq:algLES}).
We therefore have
$$ \E{akss}{alg}^{n+\alpha\epsilon,s,t}_2 = 
\begin{cases}
\mr{Im} \left( \E{bo}{alg}^{n, s, t}_{{2}} \xrightarrow{g_{alg}} H^{n, s, t}(\mc{C}_{alg}) \right), & \alpha = 0, \\
\mr{Ker} \left( \E{bo}{alg}^{n, {0}, t}_{{2}} \xrightarrow{g_{alg}} H^{n, {0}, t}(\mc{C}_{alg}) \right), & \alpha = 1, \: s = 0, \\
0, & \text{otherwise}.
\end{cases} $$
Thankfully, as the differentials $d^{bo,alg}_r$ in the $\bo$-MSS decrease $s$ by $r-1$, and the evil classes are concentrated in $s = 0$, the only other potentially non-trivial differentials ($r \ge 2$) in the AKSS are of the form:
\begin{align*}
d^{akss}_r: & \E{akss}{alg}_{r}^{n, s, t} \rightarrow \E{akss}{alg}_{r}^{n+r, s-r+1, t}, \\
d^{akss}_{s + 1 + \epsilon}: & \E{akss}{alg}_{r+\epsilon}^{n, s, t} \rightarrow \E{akss}{alg}_{r+\epsilon}^{n+s+1+\epsilon, 0, t} \\
\end{align*}
in which case they are determined by 
\begin{align*}
d^{akss}_r(x_{good}) & = \left[d^{bo, alg}_r(x_{good}) \right]_{good}, \\
d^{akss}_{s+1+\epsilon}(x_{good}) & = \left[d^{bo, alg}_r(x_{good}) \right]_{evil}.
\end{align*}

The behavior of the differentials in the AKSS is summarized by the following lemma.
\begin{lem}\label{lem:akdr}
In the agathokakological spectral sequence, given a non-trivial differential
$$ d_{r + \beta\epsilon}(x) = y $$
for $r \ge 1$ and $\beta \in \{ -1,0,1\}$, there are three possibilities:
\begin{enumerate}
\item $x$ and $y$ are both evil, $r = 1$, $\beta = 0$,
\item $x$ and $y$ are both good, $r \ge 1$, $\beta = 0$,
\item $x$ is good and $y$ is evil, $r = s(x)+1$, $\beta = 1$. 
\end{enumerate}
In other words, {\bf evil can never triumph over good}.
\end{lem}

The reader might wonder why the authors elected to index the AKSS on $\mb{N}^\epsilon$ rather than, say, half integers.  The reason is that since the map $g_{alg}$ of (\ref{eq:SES}) is a map of algebras, $V^{*,*}$ gets the structure of an ideal, and the multiplicative structure of the $\bo$-MSS descends to a multiplicative structure on the AKSS.  

Specifically, define a monoid structure on $\mb{N}^\epsilon$
determined by the rule $\epsilon+\epsilon = \epsilon$, so we have
$$ (n + \alpha \epsilon) + (n'+ \alpha' \epsilon) = 
\begin{cases}
n + n', & \alpha = \alpha' = 0, \\
(n+n') + \epsilon & \text{otherwise}.
\end{cases}
$$
Then there is a product map
$$ \E{akss}{alg}^{n+\alpha \epsilon, s, t}_{r+\beta\epsilon} \otimes 
\E{akss}{alg}^{n'+\alpha' \epsilon, s', t'}_{r+\beta\epsilon}
\rightarrow \E{akss}{alg}^{(n+\alpha \epsilon)+(n'+\alpha' \epsilon), s+s', t+t'}_{r+\beta\epsilon}. $$
The differential $d_{r+\beta\epsilon}$ is a derivation with respect to this product, provided one interprets that to mean
$$ d_{r-\epsilon}(xy) = 0 $$
if $x$ and $y$ are both evil.

With respect to this multiplicative structure, the product of good classes is good, and the product of evil classes is evil (which is why we require $\epsilon + \epsilon = \epsilon$).  For dimensional reasons, only certain kinds of hidden extensions can occur:

\begin{lem}\label{lem:akhe}
Suppose that there is a hidden extension in the AKSS given by
\[ \td{x}\td{y} = \td{z} \ne 0 \]
in $\Ext^{*,*}_{A_*}(\FF_2)$, where 
$x$, $y$, and $z$ detect $\td{x}$, $\td{y}$, and $\td{z}$ in the AKSS, respectively, with $xy = 0$ in $\E{akss}{alg}_\infty$.  Then either
\[ s(z) < s(x) + s(y) \]
or
\[ s(z) = s(x) = s(y) = 0 \: \text{and $x$ and $y$ are good, and $z$ is evil.}  \]
\end{lem}

\begin{rmk}
When $s(x) \leq 1$, this imposes strict restrictions on $y$. If $s(x)=0$ (for example if $x=h_i$ when $i\geq 2$), $y$ cannot be evil, for this would force $s(z) < 0$. 
Similarly, if $s(x)=1$ (for example, if $x = v_0$ or $h_1$), if $y$ is evil, then $s(z) = 0$. 
\end{rmk}

\subsection*{Comparison with $\AA0$}

The entire construction of the AKSS goes through without modification when the $A_*$-comodule $\FF_2$ is replaced by $\AA0$:
\[ \{ \E{akss}{alg}_{r+\beta\epsilon}^{n+\alpha \epsilon,s,t}(\AA0) \} \Rightarrow \Ext_{A_*}^{*,*}(\AA0) \]
with 
\[\E{akss}{alg}^{n+\alpha \epsilon, s, t}_{1+\epsilon}(\AA0) = 
\begin{cases}
H^{n, s, t}(\mc{C}_{alg}(\AA0)), & \alpha = 0, \\
H^{n, t}(V(\AA0)), & \alpha = 1, \: s = 0, \\
0, & \text{otherwise}.
\end{cases}
\]
The analysis of Section~\ref{sec:AA0}, in particular Proposition~\ref{prop:E2boMSSAA0}, comparing the $\bo$-MSS's of $\FF_2$ and $\AA0$ refines to give a comparison between the respective AKSS's.

\begin{thm}
We have
\[\E{akss}{alg}^{n+\alpha \epsilon, s, t}_{1+\epsilon}(\AA0) \cong 
\begin{cases}
\E{akss}{alg}^{n+1, s, t}_{1+\epsilon}(\FF_2), & \alpha = 0, t-s \equiv 0 \mod 4, \\
\E{akss}{alg}^{n, s+1, t}_{1+\epsilon}(\FF_2), & \alpha = 0, t-s \not\equiv 0 \mod 4, \\
\E{akss}{alg}^{n+1+\epsilon,0,t}_{1+\epsilon}(\FF_2), & \alpha = 1, \: s = 0, t \equiv 0 \mod 4, \\
\E{akss}{alg}^{n+1+\epsilon,0,t}_2(\FF_2)  \\
\oplus \E{akss}{alg}^{n+1,0,t}_2(\FF_2), & \alpha = 1, \: s = 0, t \not\equiv 0 \mod 4, \\
0, & \text{otherwise}.
\end{cases}
\]
Moreover, under this isomorphism, all differentials commute (but potentially changing lengths as the indexing changes dictate).
\end{thm}

As this theorem is essentially just a translation of the results of Section~\ref{sec:AA0} into agathokakological indexing, we will not say anything more about the proof.  However, the second to last case in the statement of the theorem ($\alpha = 1, s = 0, t \not\equiv 0 \mod 4$) does merit clarification.  The $h_1$-good classes in $\E{akss}{alg}^{n,s,t}_{alg}(\FF_2)$ become $h_1$-good classes in $\E{akss}{alg}^{n, s-1, t}(\AA0)$ for $s > 0$, but for $s = 0$ they become evil classes (in $\E{akss}{alg}^{n-1+\epsilon,0,t}_{alg}(\AA0)$).
A $d_{1+\epsilon}^{akss}$ differential from an $h_1$-good class to an evil class in the AKSS for $\FF_2$ becomes a $d_1^{akss}$ differential between the corresponding evil classes in the AKSS for $\AA0$.

\subsection*{The principle of dichotomy}

Given a non-trivial class $x \in \Ext^{*,*}_{A_*}(M)$, for $M = \FF_2$ or $\AA0$,  we will say that $x$ is \emph{good} if it is detected by a good class in the AKSS, and we will say $x$ is \emph{evil} if it is detected by an evil class.  We will say $x$ is \emph{$v_1$-periodic} if its image in $v_1^{-1}\Ext$ is non-trivial, and otherwise we will say that $x$ is \emph{$v_1$-torsion}.
We will now answer the following fundamental question:
\begin{quote}
{\it How do you determine if a given Ext class is good or evil?}
\end{quote}

Recall from Theorem~\ref{thm:AdamsP} (in the case of $\AA0$) and Proposition~\ref{prop:v1ext} (in the case of $\FF_2$), for every class $x \in \Ext^{s,t}_{A_*}(\FF_2)$ with $t-s > 0$, there is an $N = N(x)$ such that
\begin{equation}\label{eq:xxxxNxxx} v_1^{2^N \cdot k} x \in \Ext^{s+2^N\cdot k,t+3\cdot 2^N\cdot k}_{A_*}(\FF_2) \end{equation}
is defined for all $k$. For these classes, being $v_1$-periodic is equivalent to requiring 
$$ v_1^{2^N\cdot k} x \ne 0 $$
for all $k$.  

For $\FF_2$, the only classes in $t - s = 0$ are $v_0^i$, and these are all good (and technically are $v_1$-periodic in the sense that their image in $v_1^{-1}\Ext$ is non-trivial, though they are $v_1$-periodic of ``period $\infty$'').  For $\AA0$, there are no non-trivial classes in $t-s \le 1$.  We may therefore restrict our attention to those classes with $t-s > 0$.

A naive hope would be: ``$x$ is $v_1$-periodic if and only if $x$ is good.''
\emph{This is not true.}  However, something approximating it is.

By Proposition~\ref{prop:v1ext}, any class 
$$ x \in \Ext^{s,t}_{A_*} $$
with $s > 1/3(t-s)$ is automatically $v_1$-periodic.  We shall refer to these classes as being \emph{above the 1/3 line}.

\begin{lem}\label{lem:onethird}
Every class above the $1/3$-line is good.
\end{lem}

\begin{proof}
Since $V^{n,t}$ is a subgroup of $\pi_t bo \wedge \br{\bo}^{\wedge n}$, and $\br{\bo}$ is $3$-connected, it follows that $V^{n,t} = 0$ for $t < 4n$.  The lemma follows from the fact that, in the AKSS, evil classes in $V^{n,t}$ detect elements of $\Ext^{n,t}_{A_*}$. 
\end{proof}

The following is easily checked from the structure of $\mc{C}_{alg}(\FF_2)$ or $\mc{C}_{alg}(\AA0)$.
\begin{lem}\label{lem:v1div}
Suppose $x$ is either an element of $\mc{C}^{n,s,t}_{alg}(\FF_2)$ or $\mc{C}^{n,s,t}_{alg}(\AA0)$.  Then $v_1^{-4k}x$ exists if and only if $4k \le s$.
\end{lem}

\begin{rmk} Note that among good classes in the AKSS, notions of $v_1$-periodicity extend to all the pages in the following manner. There are variants of the AKSS for computing $\Ext^{*,*}_{A(N)_*}(\AA0)$ and $\Ext^{*,*}_{A(N)_*}(\FF_2)$. Thus, we can define $N=N(x)$ as in \eqref{eq:xxxxNxxx} on any page of the AKSS. Moreover $v_1^{2^N}$ is an actual element of $\Ext_{A(N)_*}(\FF_2)$, and the AKSS for $\Ext^{*,*}_{A(N)_*}(\AA0)$ and  $\Ext^{*,*}_{A(N)_*}(\FF_2)$ are spectral sequences of modules over the AKSS for $\Ext^{*,*}_{A(N)_*}(\FF_2)$.  Thus multiplication by $v_1^{2^N}$ commutes with differentials in these AKSS's.
\end{rmk}

The following observation is crucial --- it implies that good classes cannot detect $v_1$-torsion classes in Ext.
\begin{lem}\label{lem:v1torsion}
Suppose that $x$ and $y$ are good classes in the AKSS for $\FF_2$ or $\AA0$, and suppose that there is a non-trivial differential $d^{akss}_r(x) = y$.  Let $N = N(x)$.
\begin{enumerate}
\item If $v_1^{-2^N\cdot k} y$ exists then $v_1^{-2^N\cdot k}x$ exists.
\item For all such $k$, 
$$d_r^{akss}(v_1^{-2^N\cdot k}x) = v_1^{-2^N\cdot k}y. $$
\end{enumerate} 
\end{lem}

\begin{proof}
We use the comparison with the AKSS for computing $\Ext^{*,*}_{A(N)_*}(\FF_2)$ (respectively $\Ext^{*,*}_{A(N)_*}(\AA0)$). By Theorem~\ref{thm:AdamsP}, this spectral sequence is isomorphic to the AKSS for $\Ext^{*,*}_{A_*}(\FF_2)$ (respectively $\Ext^{*,*}_{A_*}(\AA0)$) in a range which includes both $x$ and $y$.

Suppose inductively that we have proven the lemma for $r < r_0$.  We need to prove it for $r = r_0$.  Since we have
$$ s(y) \le s(x) $$
it follows from Lemma~\ref{lem:v1div} that if $v_1^{-2^N\cdot k} y$ exists, then $v_1^{-2^N\cdot k} x$ exists.  Suppose it is not the case that  
$$ d_r(v_1^{-2^N\cdot k} x) = v_1^{-2^N\cdot k} y. $$
Then $v_1^{-2^N\cdot k} x$ must support a shorter non-trivial differential
$$   d_{r'+\beta\epsilon}(v_1^{-2^N\cdot k} x) = z $$
for $r' < r$.  By the inductive hypothesis $z$ must be evil (because if it was good we would have $d_{r'}(x) = v_1^{2^N\cdot k} z$).  In particular, since $s(z) = 0$, we have
$$ s(x)-2^{N}\cdot k-r'+1 = 0. $$
Since $v_1^{-2^N\cdot k} y$ exists, we have
$$ s(y) \ge 2^N\cdot k. $$
Finally, we have
$$ s(x) - r+1 = s(y). $$
From all these equations we deduce $r \le r'$, a contradiction.
\end{proof}

\begin{thm}\label{thm:dichotomyAA0}
Suppose $x$ is a non-trivial class in $\Ext^{s,t}_{A_*}(\AA0)$.
\begin{enumerate}
\item
If $x$ is $v_1$-torsion, it is evil.

\item 
Suppose $x$ is $v_1$-periodic, and let $N = N(x)$.  Suppose $k$ is taken large enough so that  $v_1^{2^N\cdot k} x$ lies above the $1/3$-line.  Suppose $y$ is a class
which detects $v_1^{2^N\cdot k} x$ in the AKSS ($y$ is necessarily good).  The class $x$ is good if and only if 
\[ s \ge n(y). \]
\end{enumerate}
\end{thm}

\begin{proof}
Suppose $x$ is $v_1$-torsion, but is detected by a good class $\td{x}$ in the AKSS.  Let $N = N(x)$ and suppose that $k$ is chosen so that $v_1^{2^N\cdot k} x$ lies above the 1/3-line and is trivial.  Then 
$v_1^{2^N\cdot k} \td{x}$ is killed by a good class in the AKSS.  Lemma~\ref{lem:v1torsion} then implies that $\td{x}$ is the target of a differential, which violates the non-triviality of $x$.  

Now suppose that $x$ is $v_1$-periodic, detected by an evil class $\td{x}$ in the AKSS.  Again we use the fact that both $x$ and $y$ lie in a range where the AKSS for $\Ext_{A_*}(\AA0)$ is isomorphic to the AKSS for $\Ext_{A(N)_*}(\AA0)$.
The AKSS for $\Ext^{*,*}_{A(N)_*}(\AA0)$ is a spectral sequence of modules over the AKSS for $\Ext^{*,*}_{A(N)_*}(\FF_2)$. Since $\td{x}$ is evil, we have  
\[ v_1^{2^N \cdot k}\td{x}  = 0. \]
We deduce that multiplication by $v_1^{2^N}$ must increase $n$ (hidden extension).  Thus
$$ n(\td{x}) < n(y). $$
Suppose that 
$s \ge n(y).$  Since we have
$$ n(\td{x})+s(\td{x}) = s \ge n(y) $$
we deduce that $s(\td{x}) > 0$.  This violates the assumption that $\td{x}$ is evil, so we deduce that $\td{x}$ is good.

Suppose that $x$ is good, detected by $\td{x}$ in the AKSS.  Then by Lemma~\ref{lem:v1torsion}, $v_1^{2^N\cdot k}\td{x}$ detects $v_1^{2^N\cdot k}x$.  
Then
$$ s = s(\td{x}) + n(\td{x}) = s(\td{x}) + n(y) \ge n(y). $$
\end{proof}

\begin{thm}[Dichotomy Principle]\label{thm:dichotomy}
Suppose $x$ is a non-trivial class in $\Ext^{s,t}_{A_*}(\FF_2)$.
\begin{enumerate}
\item
If $x$ is $v_1$-torsion, it is evil.

\item 
Suppose $x$ is $v_1$-periodic, and let $N = N(x)$.  Suppose $k$ is taken large enough so that  $v_1^{2^N\cdot k} x$ lies above the $1/3$-line.  Suppose $y$ is a class
which detects $v_1^{2^N\cdot k} x$ in the AKSS ($y$ is necessarily good).  The class $x$ is good if and only if 
$$ s \ge n(y). $$
\end{enumerate}
\end{thm}

\begin{proof}
This theorem follows from Theorem~\ref{thm:dichotomyAA0} using the fact that for $t \ne s$, there is an isomorphism
$$ \Ext^{s-1,t}_{A_*}(\AA0) \xrightarrow[\cong]{\bar{\partial}} \Ext^{s,t}_{A_*}(\FF_2). $$
so we have $\bar{\partial}(x') = x$.  Suppose $x$ is $v_1$-torsion (then so is $x'$).  By Theorem~\ref{thm:dichotomyAA0}, $x'$ is evil.  The map $\bar{\partial}$ takes $v_1$-torsion evil classes to evil classes.

Suppose now that $x$ is good, with $y$ detecting $v_1^{2^N\cdot k}x$.
First suppose that $y$ is $v_0$-good. Then we have
$$ y = \partial'(y') $$
so
$$ n(y)-1 = n(y'). $$
and so 
$$ s \ge n(y) \Leftrightarrow s-1 \ge n(y') \Leftrightarrow x' \: \mr{good}. $$
The theorem follows from the fact that $x'$ is good if and only if $x$ is good.

Now suppose that $y$ is $h_1$-good, so we have $y = \partial(y')$ with 
$$ n(y) = n(y'). $$
This would seem to pose a problem for $s = n(y)$; in this case 
$$ s-1 < n(y') $$
and so $x'$ is evil by Theorem~\ref{thm:dichotomyAA0}. We claim  such an $x'$ is detected by an evil class $\td{x}'$ with
$$ \td{x} := \partial'(\td{x}') $$
$h_1$-good, detecting $x$.
Indeed, suppose not.  Then $\td{x}$ is evil.  Since
$$ n(\td{x}') = s-1, $$
we have $n(\td{x}) = s$. In the AKSS for $\Ext_{A(N)}(\FF_2)$, there cannot be a $v_1^{2^N\cdot k}$-extension from filtration $s+\epsilon$ to filtration $s$.
\end{proof}

\subsection*{The topological AKSS}

There is a topological analog of the AKSS, which refines the $\bo$-ASS just as the (algebraic) AKSS constructed in the beginning of this section refines the $\bo$-MSS.

Fix a splitting of the short exact sequence
\begin{equation}\label{eq:topSES} 
0 \rightarrow V^{n,*} \rightarrow \E{bo}{}^{n,*}_1 \xrightarrow{g} \mc{C}^{n,*} \rightarrow 0.
\end{equation}
Just as in the algebraic case, we will regard the evil subcomplex $V^{n,*}$ as being in filtration $n + \epsilon$.
The result is a topological AKSS:
$$ \{ \E{akss}{}^{n+\alpha\epsilon,t}_{r + \beta \epsilon} \} \Rightarrow \pi_{t-n}^s $$
with differentials 
\begin{align*}
d^{akss}_{r-\epsilon}:  & \E{akss}{}_{r - \epsilon}^{n+ \epsilon,t} \rightarrow \E{akss}{}_{r- \epsilon}^{n+r,t}, \\
d^{akss}_{r}:  & \E{akss}{}_{r}^{n+ \alpha\epsilon,t} \rightarrow \E{akss}{}_{r}^{n+r +\alpha\epsilon, t}, \\
d^{akss}_{r+\epsilon}:  & \E{akss}{}_{r+\epsilon}^{n, t} \rightarrow \E{akss}{}_{r+ \epsilon}^{n+r +\epsilon, t}.
\end{align*}
The $E_1$-term takes the form
$$\E{akss}{}^{n+\alpha \epsilon, t}_1 = 
\begin{cases}
\mc{C}^{n,t}, & \alpha = 0, \\
V^{n, t}, & \alpha = 1. 
\end{cases}
$$
The $d_1$-differential 
$$ d^{akss}_{1}: \E{akss}{}^{n+\alpha \epsilon, t}_1 \rightarrow \E{akss}{}^{n+ 1+\alpha\epsilon, t}_1
$$
is given by
$$d_1 = 
\begin{cases}
d_1^{good}, & \alpha = 0, \\
d_1^{evil}, & \alpha = 1. 
\end{cases}
$$
We therefore have
$$\E{akss}{}^{n+\alpha \epsilon, t}_{1+\epsilon} = 
\begin{cases}
H^{n,t}(\mc{C}), & \alpha = 0, \\
H^{n, t}(V), & \alpha = 1.
\end{cases}
$$
The only nonzero $d_{1+\epsilon}$-differentials are of the form
$$ H^{n, t}(\mc{C}) = \E{akss}{}^{n, t}_{1+\epsilon} \xrightarrow{d_{1+\epsilon}} \E{akss}{}^{n+1 + \epsilon, t}_{1+\epsilon} = H^{n+1, t}(V), $$
for which we have 
$$ d_{1+\epsilon} = \partial $$
where $\partial$ is the connecting homomorphism of (\ref{eq:LES}).  It turns out all of these differentials can be derived from the algebraic AKSS.

\begin{lem}
For $n = 0,1$, the differentials
$$ d_{1+\epsilon}: \E{akss}{}^{n, t}_{1+\epsilon} \xrightarrow{d_{1+\epsilon}} \E{akss}{}^{n+1 + \epsilon, t}_{1+\epsilon} $$
are trivial.  For $n \ge 2$, they are determined by the following commutative diagram (see Corollary~\ref{cor:HC}).
$$
\xymatrix{
\E{akss}{}^{n, t}_{1+\epsilon} \ar[d] \ar[r]^-{d_{1+\epsilon}} &
\E{akss}{}^{n+1 + \epsilon, t}_{1+\epsilon} \ar@{=}[d] \\
\E{akss}{alg}^{n,{0}, t}_{1+\epsilon}  \ar[r]_-{d^{alg}_{1+\epsilon}} &
\E{akss}{alg}^{n+1 + \epsilon, {0}, t}_{1+\epsilon}
}
$$
\end{lem}

\begin{proof}
Topologically, this connecting homomorphism derives from applying $\pi_*$ to the composite
\begin{equation}\label{eq:topconnect}
\bigvee_{\abs{I} = n} \Sigma^{4\norm{I}}b_I \hookrightarrow \bo \wedge \br{\bo}^n \rightarrow \br{\bo}^{n+1} \rightarrow \bo \wedge \br{\bo}^{n+1} \rightarrow \bigvee_{\abs{I} = n+1} \Sigma^{4\norm{I}} HV_I.
\end{equation}
The first statement follows from the fact that the only elements in $H^{n,*}(\mc{C})$ for $n = 0, 1$ and $* > 0$ have Adams filtration greater than $0$, and therefore cannot map topologically to elements of Adams filtration $0$.
The second statement follows from the fact that the algebraic connecting homomorphism derives from the ASS edge homomorphism of the composite (\ref{eq:topconnect}):
$$
\mc{C}_{alg}^{n,0,*} = \E{ass}{}^{0,*}\left(\bigvee_{\abs{I} = n} \Sigma^{4\norm{I}}b_I\right) \rightarrow \E{ass}{}^{0,*}\left(\bigvee_{\abs{I} = n+1} \Sigma^{4\norm{I}}HV_I\right) = V^{n+1,*}.
$$
\end{proof}

The $E_2$-term of the $\bo$-ASS is deduced from the short exact sequence
$$ 0 \rightarrow \E{akss}{}^{n+\epsilon,t}_2 \rightarrow \E{bo}{}^{n,t}_2 \rightarrow \E{akss}{}^{n,t}_2 \rightarrow 0.
$$

At this point we can deduce Mahowald's Vanishing Line Theorem.  We only sketch the proof to emphasize the conceptual origin of this vanishing line without getting lost in the details.

\begin{thm}[Mahowald \cite{Mahowaldbo}]\label{thm:vanishing}
There is a $C$ so that 
$$ \E{bo}{}_2^{n,t} = 0 $$
for $n > 1/5(t-s)+C$. 
\end{thm}

\begin{proof}[Sketch of proof]
Since $H^{*,*}(\mc{C})$ has a $1/5$-vanishing line by Corollary~\ref{cor:HCvanishing}, it suffices to establish that the cohomology of the evil complex $H^{*,*}(V)$ has a $1/5$-vanishing line.  Suppose that $x$ is a nontrivial class in $H^{*,*}(V)$.  By the Dichotomy Principle (Theorem~\ref{thm:dichotomy}), there are three possibilities.
\begin{enumerate}
\item $x$ detects a $v_1$-torsion class of $\Ext$ in the algebraic AKSS.
\item $x$ detects a $v_1$-periodic class of $\Ext$, which in the $v_1$-periodic $\bo$-MSS, is detected in $v_1^{-1}\E{bo}{alg}^{n,s,t}_2$ with $s < 0$.
\item $x = \partial_{alg}(y)$, for $y \in H^{*,*}(\mc{C})$.
\end{enumerate}
Recall (Remark~\ref{rmk:onefifth}) that $\Ext$ is entirely $v_1$-periodic above a ``periodicity line'' of slope $1/5$ (and above this periodicity line the $\Ext$-groups are isomorphic to the $v_1^{-1}\Ext$-groups). 
Thus if we are in case (1), we deduce that $x$ must lie below this periodicity line.  
The same inequalities used in the proof of Corollary~\ref{cor:HCvanishing} also prove that there is a $C'$ so that the maps
$$ H^{n,s,t}(\mc{C}_{alg}) \rightarrow H^{n,s,t}(v_1^{-1}\mc{C}_{alg}) $$
are isomorphisms for
\begin{equation}\label{eq:goodperiodicity}
 n+s > 1/5(t-s-n) + C'.
\end{equation} 
Thus in case (2), we deduce that $x$ must lie below this line of slope $1/5$.  In case (3) for $M \gg 0$, the class $v_1^{M}y$ is either non-trivial, or is the source or target of a differential in the algebraic AKSS.  In the former case, $y$ cannot lie above the periodicity line.  In the latter case, Lemma~\ref{lem:v1torsion} implies that $y$ must lie below both the 1/5 line given by (\ref{eq:goodperiodicity}).
\end{proof}

The differentials in the topological AKSS determine and are determined by the differentials in the $\bo$-ASS, with lengths dictated by whether the sources and targets of the $\bo$-ASS differentials are good or evil.  Unlike the algebraic case, in the topological case there are no dimensional restrictions: in principle good or evil classes can each kill either good or evil classes.  

\begin{rmk}
There is no Dichotomy Principle in the topological AKSS.  Many $v_1$-torsion elements of $\pi_*^s$ are detected by good classes in the $\bo$-ASS (e.g. $\nu^2, \kappa, \theta_{3}, \bar{\kappa}, \ldots$).  In fact, Mahowald showed in \cite{Mahowaldbo} that a non-trivial class in $\pi_*^s$ is $v_1$-periodic if and only if it has $\bo$-filtration $\le 1$.
\end{rmk}


\section{Computation of the algebraic $\bo$-resolution}\label{sec:algcomp}

In this section we will compute the algebraic $\bo$-resolution, or more specifically, the (algebraic) AKSS, through dimension $42$.  We use the known computation of $\Ext_{A_*}(\FF_2)$ through this range (see, for example, the May spectral sequence computation of \cite{Tangora} or the computer computation of \cite{Bruner},\cite{Brunertable}), as well as our computation of $H^{*,*,*}(\mc{C}_{alg})$ (Theorem~\ref{thm:HCalg}) to deduce the groups $H^{*,*,*}(V)$, and the subsequent differentials.  Key to this is the determination of which classes of $\Ext_{A_*}(\FF_2)$ are good and which are evil.  This is done with the Dichotomy Prinicple (Theorem~\ref{thm:dichotomy}).  A prerequisite to applying the Dichotomy Principle is the determination of which elements in $\Ext_{A_*}(\FF_2)$ in our range are $v_1$-periodic, and which are $v_1$-torsion.

To this end we begin this section with an analysis of $v_1$-periodicity and torsion in our range.  We then explicitly write out $H^{*,*,*}(\mc{C}_{alg})$ and apply the Dichotomy Principle to determine which classes in $\Ext_{A_*}(\FF_2)$ are good and which are evil.  We then prove a couple of convenient lemmas which relate the $d_2$ differentials in the AKSS to $v_1^4$-multiplication.  We then do a  stem by stem computation of the AKSS through dimension $42$. 

\subsection*{$v_1$-periodicity and $v_1$-torsion in low degrees}

In order to invoke the Dichotomy Principle to determine which classes in $\Ext_{A_*}(\FF_2)$ are good and which classes are evil in low degrees, it is necessary to determine which classes in this range are $v_1$-periodic, and which are $v_1$-torsion.

Ideally, this would be accomplished by actually having a complete computation of $v_1^{-1}\Ext_{A_*}(\FF_2)$.  To date, this has not been done.  However, Davis and Mahowald have computed $v_1^{-1}\Ext_{A_*}(H_*Y)$ where $Y = M(2) \wedge C\eta$ \cite{DavisMahowaldv1}.  Note that Andrews \cite{Andrews} has computed $v_1^{-1}\Ext_{A_*}(\FF_p)$ for $p$ odd.

\begin{rmk}
The best available tool to compute $v_1^{-1}\Ext_{A_*}(\FF_2)$ is probably the localized algebraic $\bo$-resolution
\begin{equation}\label{eq:v1algbo}
 v_1^{-1}H^{*,*,*}(\mc{C}_{alg}) \Rightarrow v_1^{-1}\Ext_{A_*}(\FF_2).
 \end{equation}
The only difficulty is that this spectral sequence has many long differentials (as will be demonstrated in our low degree calculations in this section).  Nevertheless, it seems plausible that with enough care the spectral sequence (\ref{eq:v1algbo}) could be completely computed.  The odd primary analog of (\ref{eq:v1algbo}) actually collapses, giving a convenient alternative approach to Andrews' computation of $v_1^{-1}\Ext_{A_*}(\FF_p)$ for $p$ odd.
\end{rmk}

In the absence of a complete understanding of $v_1^{-1}\Ext$, we instead manually identify the kernel of the homomorphism
(a.k.a. the $v_1$-torsion)
$$ \Ext^{s,t}_{A_*}(\FF_2) \rightarrow v_1^{-1}\Ext^{s,t}_{A_*}(\FF_2) $$
in low degrees.  Figure~\ref{fig:v1Extchart} shows an Ext chart in low degrees (courtesy of Perry's Ext software).  The chart is broken into regions which indicate for which $N$ multiplication by $v_1^{2^N}$ is well-defined (c.f. Proposition~\ref{prop:v1ext}).  The circled classes span the $v_1$-torsion (see the following proposition).  The classes decorated with triangles represent the only classes in this range which are $v_1$-periodic, but evil --- this will be explained in the next subsection.

\begin{figure}
\includegraphics[angle=90,height=\textheight]{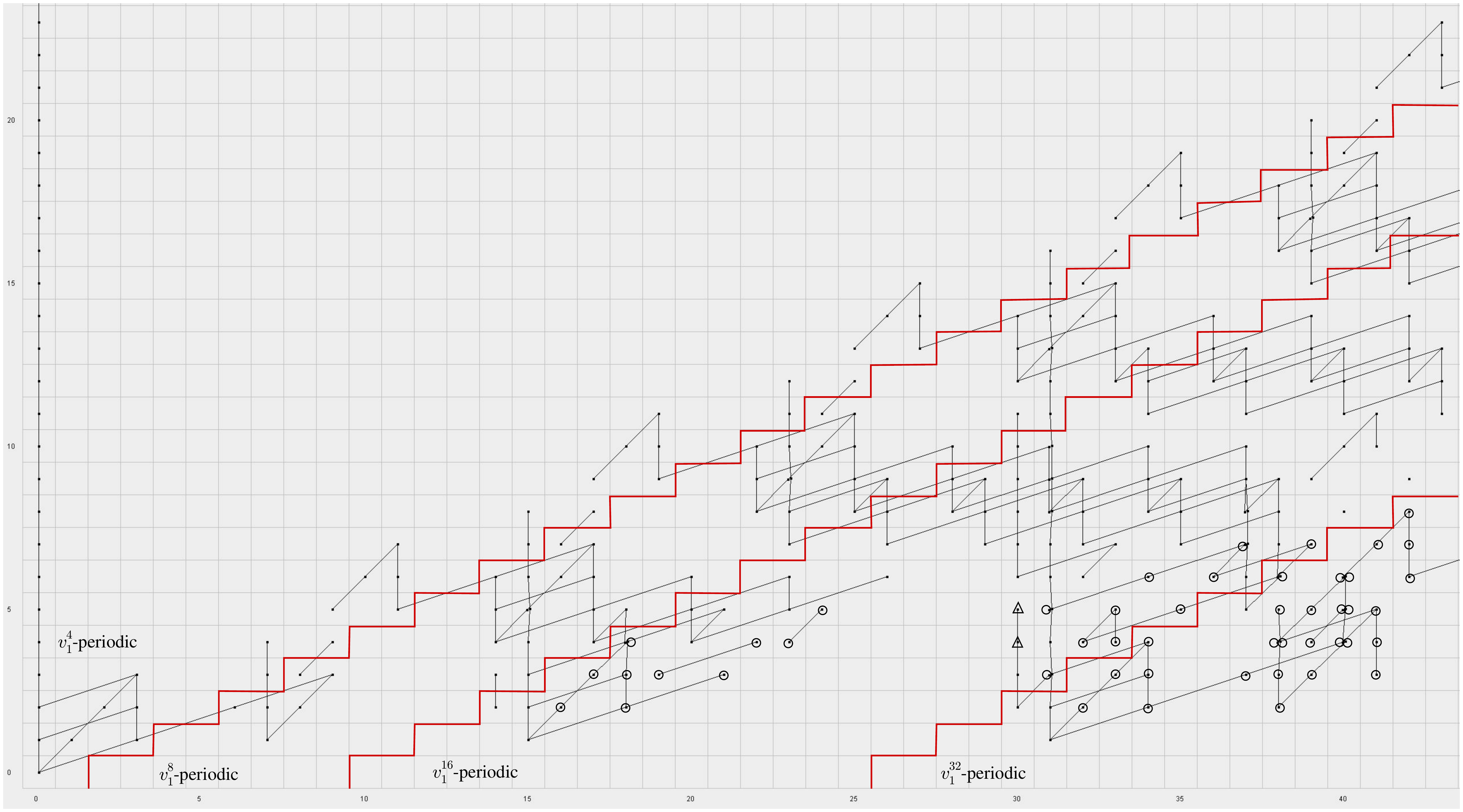}
\caption{$v_1$-periodicity in $\Ext_{A_*}(\FF_2)$.}\label{fig:v1Extchart}
\end{figure}

\begin{prop}\label{prop:v1Extchart}
The $v_1$-torsion in $\Ext^{s,t}_{A_*}(\FF_2)$ for $0\leq t-s\leq 43$ is spanned by the circled classes in Figure~\ref{fig:v1Extchart}.
\end{prop}

\begin{proof}
This represents a tedious analysis of available Ext data, the highlights of which we summarize here.
Basically, one must check that the images of the uncircled classes are linearly independent in $v_1^{-1}\Ext_{A_*}(\FF_2)$, and that the circled classes are $v_1$-torsion.

The vast majority of the uncircled classes are $v_1^4$-periodic.  This can be verified by checking that for such classes $x$, the image of the iterated Adams $P$ operator $P^N x$ is non-trivial, for $N$ sufficiently large that $P^N x$ lies above the $1/3$-line.  These verifications can be done using Bruner's tables \cite{Bruner}.  In many instances, the process is expedited by simply observing that the corresponding classes map non-trivially to $\Ext_{A(2)_*}$, where everything is $v_1^4$-periodic.

The only $v_1$-periodic classes this technique does not apply to are those in the $v_0$-towers $v_0^jh_i$ in $t-s = 2^i-1$, the ``broken'' $v_0$-towers in $t-s = 2^i-2$, and the $v_0$-towers with bottoms in bidegrees $(t-s,s) = (37, 5)$ and $(38,6)$.

In the case of the $v_0$-towers $v_0^jh_i$, the tops of the towers are actually $v_1^4$-periodic (by the $P$-operator argument).  That implies that the images of the top of the towers in $v_1^{-1}\Ext_{A_*}(\FF_2)$ are non-trivial.  It must therefore be the case that all of the classes $v_0^jh_i$ are non-trivial.

In the case of the broken $v_0$ towers in $t-s = 2^i-2$, we employ a bit of a cheat.
The idea is that in the localized ASS
$$ v_1^{-1}\Ext_{A_*}(\FF_2) \Rightarrow v_1^{-1}\pi_* S $$
the tops of the broken $v_0$-towers in dimensions $s2^i-2$ are the targets of differentials on the $v_0$-towers in dimensions $s2^i-1$.  If differentials of the same length happen in the unlocalized ASS, then the targets of these differentials must be non-trivial under the map of ASS's:
$$
\xymatrix{
\Ext_{A_*}(\FF_2) \ar[d] \ar@{=>}[r] &  \pi_* S \ar[d]
\\
v_1^{-1}\Ext_{A_*}(\FF_2) \ar@{=>}[r] &  v_1^{-1}\pi_* S
}
$$
We illustrate this principle with an example: we will show the class $v_0h_3^2$ is $v_1^8$-periodic.  In the unlocalized ASS, there is a differential
$$ d_2 (v_0^4 Q') = v_0^5 i^2 $$
(where $Q'$ is in degree $(47,13)$ and $i^2$ in $(46,14)$)
and this differential must map to a non-trivial $d_2$-differential in the localized ASS since we are working in a region where the map
$$ \Ext_{A_*}(\FF_2) \rightarrow v_1^{-1}\Ext_{A_*}(\FF_2) $$
is an isomorphism.  In the localized ASS, this differential interpolates back to a differential
$$ d_2 (h_4 ) = v_1^{-16} v_0^5 i^2. $$
Since in the unlocalized ASS we have
$$ d_2  (h_4)  = v_0 h_3^2 $$
we deduce that $v_0 h_3^2$ must map to $v_1^{-16} v_0^5 i^2$ in $v_1^{-1}\Ext$, and therefore in $\Ext$ we have
$$ v_1^{16} v_0 h_3^2 = v_0^5 i^2. $$
Since $v_0^5 i^2$ is above the 1/3 line, we deduce that $v_0 h_3^2$ is $v_1$-periodic.  Because it is located in the region of $\Ext$ where multiplication by $v_1^8$ is well-defined, we deduce that in fact $v_0h_3^2$ is $v_1^8$-periodic.

In the case of the towers in $(t-s,s) = (37, 5)$ and $(38,6)$, we deduce that they are $v_1$-periodic as follows.  It suffices to show the tops of these towers $v_0^5 x$ and $v_0^3 y$ are $v_1$-periodic.
We have
\begin{align*}
v_0^5 x & = {h_2^2 d_0e_0}, \\
v_0^3 y & = h_2^2 l.
\end{align*}
The classes $h_2 d_0$ and $h_2 l$ have already been shown to be $v_1^4$-periodic, and we have
\begin{align*}
v_0^8 P^2 x' & = {h_2^2 P^4 d_0e_0}, \\
v_0^8 R'_1 & = h_2 P^4 h_2 l,
\end{align*}
for $P^2 x' $ in $(69,18)$ and $R_1'$ in $(70,17)$.
Moreover, we are working in a range (c.f. Theorem~\ref{thm:AdamsP} where the composite
\begin{equation}\label{eq:A4}
\Ext^{s,t}_{A_*} \rightarrow \Ext^{s-1,t}_{A_*}(\br{A \mmod A(0)}_*) \rightarrow \Ext_{A(4)_*}(\br{A \mmod A(0)}_*)
\end{equation}
is an isomorphism.  Since $v_1^{16} \in \Ext_{A(4)}(\FF_2)$, and the composite (\ref{eq:A4}) is $h_2$-linear, we deduce that
\begin{align*}
v_1^{16} v_0^5 x & = v_0^8 P^2 x', \\
v_1^{16} v_0^3 y & = v_0^8 R'_1.
\end{align*}
Since $v_0^8 P^2 x'$ and $v_0^8 R'_1$ lie above the 1/3 line, they are $v_1$-periodic.  It follows that $v_0^5 x$ and $v_0^3 y$ are both $v_1$-periodic.

We now explain how to check that the circled classes are $v_1$-torsion. The classes which are divisible by $h_i$ for $i \ge 4$ can be handled with the following trick, which we illustrate in the case of $i = 4$.  Using the fact that $h_4$ is $v_1$-periodic, we can deduce that
$$ v_1^{16} h_4 = v_0^4 Q'. $$
It follows that for all $x = h_4 y$,
$$ v_1^{16}x = v_1^{16} h_4 y = v_0^4 Q' y. $$
Therefore, if $y$ is annihilated by $v_0^4$, then $v_1^{16} h_4 y = 0$.
It follows that $v_1^{16}$ times the classes
$$ h_4 h_1, \: h_4 h_2, h_4c_0 $$
are all zero. So the same is true of their $v_0$, $h_1$, and $h_2$-multiples.
The same trick shows that $v_2^{32} h_5 y = 0$ is true for all of the relevant values of $y$.

The remaining cases, such as $c_1$, $n$, etc. are handled by observing that the only classes which could detect their $v_1^{16}$ (or $v_1^{32}$, as appropriate) multiples either (1) have non-trivial products with $v_0$ or $h_1$ which the original class does not have, or (2) map non-trivially to $v_1^{-1}\Ext_{A(2)_*}(\FF_2)$ (and the original class maps trivially into this Ext group.)  The case of $f_1$ in $(t-s,s) = (40,4)$ is somewhat tricky.  The only candidate to detect $v_1^{32}f_1$ has a non-trivial product with $d_0$, whereas $f_1d_0 = 0$.  The same argument shows the class $c_2$ in $(41,3)$ is $v_1$-torsion.
\end{proof}

\subsection*{The cohomology of $\mc{C}_{alg}$}
Throughout this section, by AKSS we refer to the \emph{algebraic} AKSS. The cohomology groups $H^{n,s,t}(\mc{C}_{alg})$ can be read off from Theorem~\ref{thm:HCalg}.  The result, in low degrees, is depicted in Figure~\ref{fig:HCchart}. In this chart, the $x$-axis denotes $t-(s+n)$, and the $y$-axis denotes $s+n$.  The $\bo$-filtration $n$ is encoded by a color given in Figure~\ref{fig:colorbo}.
Application of the Dichotomy Principle (Theorem~\ref{thm:dichotomy}) gives the following.

\begin{figure}
\begin{tabular}{ | c | c | }
  \hline
  $n$ & color \\
  \hline
  \hline
  0 & black  \\
  \hline
  1 & {\color{blue}blue}  \\
  \hline
  2 & {\color{red}red}   \\
  \hline
   3 & {\color{orange}orange}  \\
  \hline
   4 & {\color{limegreen}green} \\
  \hline
     5 & {\color{cyan}cyan} \\
  \hline
     6 & {\color{lavenderrose}pink} \\
  \hline
       7 & {\color{darkmagenta}purple} \\
  \hline
       8 & {\color{goldenpoppy}yellow} \\
  \hline
       9 & {\color{seagreen}forest green} \\
  \hline
\end{tabular}
\caption{The $\bo$-filtration.}
\label{fig:colorbo}
\end{figure}

\begin{prop}
The classes in Figure~\ref{fig:v1Extchart} which are evil are precisely those classes which are marked with a circle (case (1) of Theorem~\ref{thm:dichotomy}) or a triangle (case (2) of Theorem~\ref{thm:dichotomy}).
\end{prop}

\subsection*{The propagation of good : $v_1$--periodic differentials}

We now begin our stemwise computation of the algebraic AKSS
$$ \E{akss}{alg}_{1+\epsilon}^{n+\alpha\epsilon, s,t} \Rightarrow \Ext^{s+n,t}_{A_*}(\FF_2). $$
The full chart of this spectral sequence in the range we will be considering (starting at the $E_{1+\epsilon}$ page) is shown in Figure~\ref{fig:AAKSSchart}.  

\begin{figure}
\includegraphics[angle=90,height=\textheight]{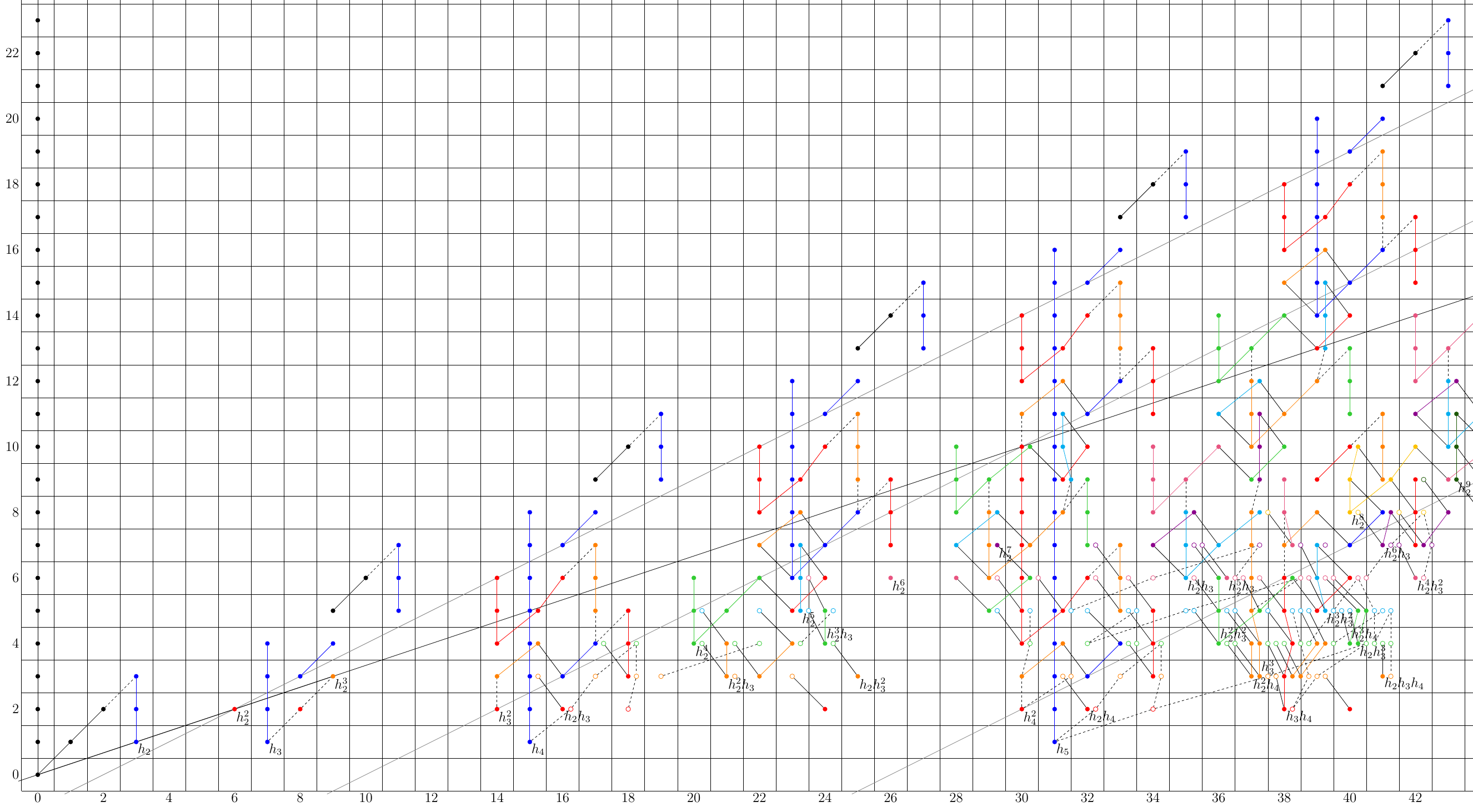}
\caption{The algebraic AKSS. 
}\label{fig:AAKSSchart}\label{fig:HCchart}.
\end{figure}

\begin{not*}We name the classes in Figure~\ref{fig:AAKSSchart} by
\[(x,y: n),\]
where $(x,y)=(t-(s+n),s+n)$ is the Adams coordinate and $n$ is the $\bo$--filtration. We use the same notation, but add a superscript to denote evil classes, i.e., $(x,y: n)^{ev}$. We call this the \emph{nature} of the class.
If multiple classes have the same nature, we distinguish them by a subscript $(x,y: n)_k$, respectively $(x,y: n)_k^{ev}$. The subscript $k$ denotes that this is the $k$'th class from the left in our chart with this nature, counting evil and good separately. 
\end{not*}

We will use the following lemma.
\begin{lem}\label{lem:v14}
Let $a$ and $b$ be elements on the $E_2$-page of the AKSS, which are in the Adams coordinate $(x, y)$ and $(x-1, y+1)$. Suppose firstly that $v_0  a =0$ and $v_1^4 a \neq 0$ on the $E_2$-page of the AKSS, and secondly that the element $v_0^3 {\color{blue}(7,1:1)}$ annihilates all elements in the Adams coordinates $(x+1, y)$ and $(x, y+1)$. Then $d_2(a) = b$ implies that $d_2(v_1^4  a) = v_1^4  b$.
\end{lem}
\begin{proof}
We have a $d_1$ differential $d_1(v_1^4) = v_0^4 {\color{blue}(7,1:1)}$. Therefore, on the $E_2$-page of the AKSS, we have the following Massey products with zero indeterminacies:
\[v_1^4 a = \left<a, v_0, v_0^3{\color{blue}(7,1:1)}\right>,\]
\[v_1^4 b = \left<b, v_0, v_0^3{\color{blue}(7,1:1)}\right>.\]
The first condition implies that the Massey products are well-defined. Note that by $v_0$-linearity, $v_0  a =0$ and $d_2(a) = b$ imply that $v_0  b = 0$. The second condition implies that the Massey products have zero indeterminacies. Note that the Adams bidegree of $v_1^4$ is zero starting from the $E_2$-page. Therefore, the lemma follows from Leibniz's rule of Massey products.
\end{proof}

Finally, recall from Proposition~\ref{prop:v1ext} that $N(x) \leq n$ if, for $(t-s,s)$ the Adams coordinates of $x$, we have $0<t-s<2s+2^{n+1}-6$. This is illustrated in Figure~\ref{fig:v1Extchart}.

We begin by establishing families of periodic differentials between good classes. By the Dichotomy Principle
there can be no evil classes above the 1/3-line. This reduces the possibilities and allows us to make easy arguments. The differentials established in this region can then be ``pulled back'' using Lemma~\ref{lem:v1torsion}. Further, note that Lemma~\ref{lem:v1torsion} implies that the relevant classes must survive long enough for the differentials stated below to occur.

\begin{prop}\label{prop:blueoranage}
There is a $v_1^4$--periodic family of differentials starting with
\[d_2( {\color{blue}(16,3:1)}) = {\color{orange}(15,4:3)}\]
and a $v_1^8$--periodic family starting with
\[d_2( {\color{blue}(23,6:1)} ) = {\color{orange}(22,7:3)} .\]
\end{prop}
\begin{proof}
$\Ex$ is zero in degree $(38,15)$. Therefore, the class ${\color{orange}(38,15:3)}$ cannot survive. This forces the differential $d_2( {\color{blue}(39,14:1)}) = {\color{orange}(38,15:3)}$. By $h_1$--linearity, this implies that $d_2( {\color{blue}(40,15:1)}) = {\color{orange}(39,16:3)}$. The classes $ {\color{blue}(40,15:1)}$ and $ {\color{orange}(39,16:3)}$ are $v_1^4$--periodic on the $E_{1+\epsilon}$--term of the AKSS. By Lemma~\ref{lem:v14}, this differential is $v_1^4$--periodic (and, therefore, $v_1^8$--periodic). The classes $ {\color{blue}(39,14:1)}$ and ${\color{orange}(38,15:3)}$ are $v_1^8$--periodic on the $E_{1+\epsilon}$--term. The reoccurrence of $d_2( {\color{blue}(40,15:1)}) = {\color{orange}(39,16:3)}$ and $h_1$--linearity force the reoccurrence of $d_2( {\color{blue}(39,14:1)}) = {\color{orange}(38,15:3)}$. 
Therefore, from this point on, the $v_1^4$, respectively $v_1^8$, multiples of these differentials always occur. We apply Lemma~\ref{lem:v1torsion} with $N=4$ and appropriate values of $k$ to pull-back these differentials. For example, using $k=1$, the differential $d_2( {\color{blue}(48,19:1)}) = {\color{orange}(47,20:3)}$ implies that $d_2( {\color{blue}(16,3:1)}) = {\color{orange}(15,4:3)}$ and the differential $d_2( {\color{blue}(55,22:1)}) = {\color{orange}(54,23:3)}$ implies that $d_2( {\color{blue}(23,6:1)} ) = {\color{orange}(22,7:3)}$. 
\end{proof}

\begin{prop}\label{prop:blueoranage1}
There is a $v_1^4$--periodic family of differentials starting with
\[d_2( {\color{blue}(32,3:1)}) = {\color{orange}(31,4:3)}.\]
\end{prop}
\begin{proof}
Since $\Ex$ is zero in degree $(71,24)$, we must have $d_2( {\color{blue}(72,23:1)}) = {\color{orange}(71,24:3)}$. The classes are $v_1^4$--periodic, so as in Proposition~\ref{prop:blueoranage}, we apply Lemma~\ref{lem:v14} and Lemma~\ref{lem:v1torsion} with $N=5$ and appropriate values of $k$ to conclude the result.
\end{proof}

\begin{prop}\label{prop:oragnecyan}
There is a $v_1^4$--periodic family of differentials starting with
\[d_2( {\color{orange}(29,6:3)}) = {\color{cyan}(28,7:5)} \ \ \ \text{and} \ \ \ d_2( {\color{orange}(30,7:3)}) = {\color{cyan}(29,8:5)}  .\]
\end{prop}
\begin{proof}
$\Ex$ is zero in degree $(44,15)$. This forces the differential $d_2( {\color{orange}(45,14:3)}) = {\color{cyan}(44,15:5)}$. Then, $h_1$--linearity forces the differential $d_2( {\color{orange}(46,15:3)}) = {\color{cyan}(45,16:5)}$. These classes are $v_1^4$--periodic so Lemma~\ref{lem:v14} implies the periodicity of the differential. Finally,  Lemma~\ref{lem:v1torsion} with $N=4$ allows us to pull back the differentials.
\end{proof}

\begin{prop}\label{prop:cyanmagenta}
There is a $v_1^4$--periodic family of differentials starting with
\[d_2( {\color{cyan}(35,6:5)}) = {\color{darkmagenta}(34,7:7)} \ \ \ \text{and} \ \ \ d_2( {\color{cyan}(36,7:5)}) = {\color{darkmagenta}(35,8:7) }  .\]
\end{prop}
\begin{proof}
This is the same argument as in the proof of Proposition~\ref{prop:oragnecyan}, using the fact that
$\Ex$ is zero in degree $(66,23)$.
\end{proof}

\begin{prop}\label{prop:greenred}
There are $v_1^4$--periodic families of differentials starting with
\[d_2( {\color{red}(23,5:2)}) = {\color{limegreen}(22,6:4)} \ \ \ \text{and} \ \ \  d_3( {\color{red}(24,6:2)} ) = {\color{cyan}(23,7:5)}.\]
\end{prop}

\begin{proof}
$\Ex$ is zero in degree $(38,14)$. Therefore, ${\color{limegreen}(38,14:4)}$ cannot survive and this forces the differential $d_2( {\color{red}(39,13:2)}) = {\color{limegreen}(38,14:4)}$.
As in Figure~\ref{fig:AA0}, let $x =  {\color{limegreen}(36,12:4)}$, $\partial y =  {\color{limegreen}(37,13:4)}$, $\partial h_1y =  {\color{limegreen}(38,14:4)}$, $\partial w = {\color{red}(39,13:2)}$ and $\partial'z = {\color{cyan}(39,13:5)}$ and note that $\partial'z = h_2x$.
The differential
$d_2( \partial w) = \partial h_1 y $ having been established, part (2)  of Theorem~\ref{thm:AA0} implies that $d_2(w) = h_1y$. By $h_1$--linearity, $d_2(h_1w) = h_1^2 y = v_0^2 z$. Then part (1) of Theorem~\ref{thm:AA0} implies that $d_3(\partial h_1w) = \partial' v_0^2 z = v_0^2 h_2x$.

Note that once the $d_2$--differential is established, the $d_3$--differential is a direct consequence of Theorem~\ref{thm:AA0}. By Lemma~\ref{lem:v14}, the $d_2$-differential is $v_1^4$--periodic, and hence, we have the same periodicity for the $d_3$--differential.

Finally, we use Lemma~\ref{lem:v1torsion} with $N=4$ to pull back these differentials. 
\end{proof}

\begin{figure}
\includegraphics[height=0.5\textheight]{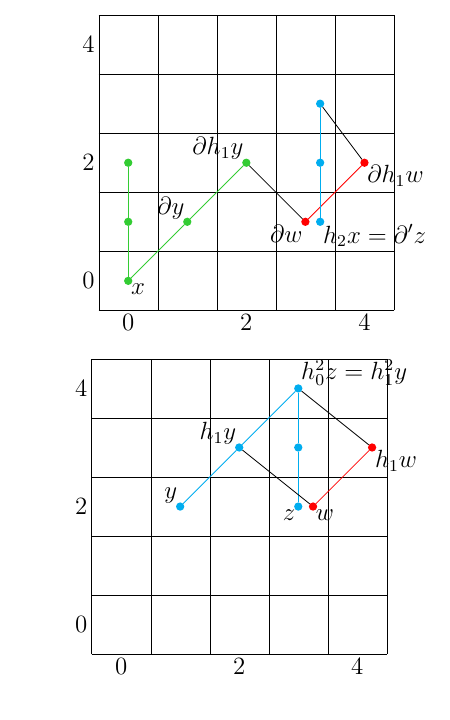}
\caption{An application of Theorem~\ref{thm:AA0}: a $d_2$--differential in the AKSS for $\AA0$ (bottom) implies a $d_3$--differential in the AKSS for $\FF_2$ (top).}\label{fig:AA0}
\end{figure}

\begin{prop}
There are $v_1^4$--periodic families of differentials starting with
\[d_2( {\color{limegreen}(29,5:4)}) = {\color{lavenderrose}(28,6:6)} \ \ \ \text{and} \ \ \  d_3( {\color{limegreen}(30,6:4)}) = {\color{darkmagenta}(29,7:7)} .\]
\end{prop}
\begin{proof}
This is the same argument as in Proposition~\ref{prop:greenred}, starting with the fact that $\Ex$ is zero in degree $(52,18)$, which forces the $d_2$--differential $d_2( {\color{limegreen}(53,17:4)}) = {\color{lavenderrose}(52,18:6)}$.
\end{proof}

\begin{prop}
There are $v_1^4$--periodic families of differentials starting with
\[d_2( {\color{red}(39,5:2)}) = {\color{limegreen}(38,6:4)} \ \ \ \text{and} \ \ \  d_3( {\color{red}(40,6:2)} ) = {\color{cyan}(39,7:5)}.\]
\end{prop}
\begin{proof}
This is the same argument as above, starting with the fact that $\Ex$ is zero in degree $(78,26)$, so that $d_2( {\color{red}(79,25:2)}) = {\color{limegreen}(78,26:4)}$. When applying Lemma~\ref{lem:v1torsion}, we use $N=5$.
\end{proof}

\begin{prop}\label{prop:pinkyellow}
There are $v_1^4$--periodic families of differentials starting with
\[d_2( {\color{lavenderrose}(43,9:6)}) = {\color{goldenpoppy}(42,10:8)} \ \ \ \text{and} \ \ \  d_3( {\color{lavenderrose}(44,10:6)}) = {\color{seagreen}(43,11:9)} .\]
\end{prop}
\begin{proof}
This is the same argument as in Proposition~\ref{prop:greenred}, starting with the fact that $\Ex$ is zero in degree $(66,22)$, which forces the $d_2$--differential $d_2({\color{lavenderrose}(67,21:6)}) = {\color{goldenpoppy}(66,22:8)}$.
\end{proof}

\begin{prop}\label{prop:blueyellow}
There is a $v_1^{4}$--periodic family of differentials starting with
\[d_7( {\color{blue}(41,8:1)}) = {\color{goldenpoppy}(40,9:8)} .\]
\end{prop}
\begin{proof}
These classes are $v_1^4$--periodic on the $E_{1+\epsilon}$--term of the AKSS, but they lie in the $v_1^{16}$--periodic region of $\Ex$. So we proceed with appropriate care.

The fact that $\Ex$ is zero in degrees $(64,21)$, $(72,25)$, $(80,29)$ and $(88,33)$ forces differentials
\begin{align*}
d_7( {\color{blue}(65,20:1)}) &= {\color{goldenpoppy}(64,21:8)} &
d_7( {\color{blue}(73,24:1)}) &= {\color{goldenpoppy}(72,25:8)} \\
d_7( {\color{blue}(81,28:1)}) &= {\color{goldenpoppy}(80,29:8)} &
d_7( {\color{blue}(89,32:1)}) &= {\color{goldenpoppy}(88,33: 8)}.
\end{align*}
We are in a range where the AKSS for $\Ext_{A(4)_*}(\FF_2)$ is isomorphic to that of $\Ext_{A_*}(\FF_2)$. Since $v_1^{16} \in
\Ext_{A(4)_*}(\FF_2)$, this differential is $v_1^{16}$--linear in the AKSS for $\Ext_{A(4)_*}(\FF_2)$, and thus, also in that of $\Ext_{A_*}(\FF_2)$. Using  Lemma~\ref{lem:v1torsion}, we can pull back the differentials as claimed.
\end{proof}

\begin{prop}\label{prop:orangeyellowred}
There is a $v_1^{4}$--periodic family of differentials starting with
\[d_5( {\color{orange}(41,9:3)}) = {\color{goldenpoppy}(40,10:8)} .\]
\end{prop}

\begin{proof}
This is an argument similar to Proposition~\ref{prop:blueyellow}.
$\Ex$ is zero in degrees $(65,21)$, $(73,25)$, $(81,29)$ and $(89,33)$. For degree reasons, the only way the AKSS can realize this is if
\begin{align*}
d_5( {\color{orange}(65,21:3)}) &= {\color{goldenpoppy}(64,22:8)} & d_6( {\color{red}(66,20:2)}) &= {\color{goldenpoppy}(65,21:8)} \\
d_5( {\color{orange}(73,25:3)}) &= {\color{goldenpoppy}(72,26:8)} & d_6( {\color{red}(74,24:2)}) &= {\color{goldenpoppy}(73,25:8)} \\
d_5( {\color{orange}(81,29:3)}) &= {\color{goldenpoppy}(80,30:8)} & d_6( {\color{red}(82,28:2)}) &= {\color{goldenpoppy}(81,29:8)} \\
d_5( {\color{orange}(89,33:3)}) &= {\color{goldenpoppy}(88,34:8)} & d_6( {\color{red}(90,32:2)}) &= {\color{goldenpoppy}(89,33:8)}.
\end{align*}
Now apply Lemma~\ref{lem:v1torsion} with $N=4$.
\end{proof}

\begin{rmk}
We will see that the only differentials between good classes which have not been accounted for in our range are
\begin{align*}
d_5( {\color{red}(42,7:2}) &= {\color{darkmagenta}(41,8:7)}  \\
d_6( {\color{red}(42,8:2)}) &= {\color{goldenpoppy}(41,9:8)} .
\end{align*}
The same methods as above using $v_1^{32}$--periodicity would apply, but would require studying $\Ex$ in the stems $\geq 100$ so we have decided to use direct arguments (see Proposition~\ref{prop:exception}).
\end{rmk}

\subsection*{The calm before the storm : Stems 0-32}
In the following subsections, we turn to differentials involving evil classes and the few good differentials our previous analysis missed. We make extensive use of Figure~\ref{fig:v1Extchart}.

\ \\
\noindent
{\bf Stem 0-14}

\begin{prop}
There are no non-trivial differentials $d_r$ for $r \geq 1+\epsilon$ in the AKSS in the range $0 \leq t-(s+n) \leq 14$. In this range, $  \E{akss}{alg}^{n+\epsilon,s,t}_{1+\epsilon} =0$ and
\[H^{n,s,t}(\mc{C}_{alg}) \cong  \E{akss}{alg}^{n,s,t}_{1+\epsilon} \cong  \E{akss}{alg}^{n,s,t}_{\infty} .\]
\end{prop}
\begin{proof}
From Figure~\ref{fig:v1Extchart}, all classes in $\Ext_{A_*}^{s,t}(\FF_2)$ for $0\leq t-s\leq 15$ are detected by good and there is a bijection between $H^{n,s,t}(\mc{C}_{alg})$ and $\Ext_{A(2)_*}^{s,t}(\FF_2)$ in this range.
\end{proof}

\ \\
\noindent
{\bf Stems 15-21}

\begin{prop}
There are differentials
\begin{align*}
d_{1+\epsilon}({\color{red}(16,2:2)}) &= {\color{orange}(15,3:3)^{ev} } , &
d_{2+\epsilon}({\color{red}(18,3:2)})&={\color{limegreen}(17,4:4)^{ev}}.\end{align*}
\end{prop}
\begin{proof}
The only class in $\Ext_{A_*}(\FF_2)$ in degree $(16,2)$ and $(18,3)$ are detected by evil, therefore, ${\color{red}(16,2:2)}$ and ${\color{red}(18,3:2)}$ cannot survive, forcing these differentials.
\end{proof}

\begin{prop}
There are differentials
\begin{align*}
d_{1+\epsilon}({\color{orange}(21,3:3)})  &= {\color{limegreen}(20,4:4)^{ev}} , & d_{2+\epsilon}({\color{orange}(21,4:3)})  &= {\color{cyan}(20,5:5)^{ev}}.\end{align*}
\end{prop}

\begin{proof}
It follows from Figure~\ref{fig:v1Extchart} that $h_2^2h_4$ is detected by $\color{orange}(21,3:3)^{ev}$. All classes in $\Ex$ in degree $(21,3)$ are accounted for and $\Ex$ is zero in $(21,4)$. Both ${\color{orange}(21,3:3)}$ and ${\color{orange}(21,4:3)}$ must die, and the only possibility is for them to kill evil.
\end{proof}

\ \\
\noindent
{\bf Stem 22-25}

\begin{prop}\label{prop:stem22_1}
There are differentials
\begin{align*}
d_{1+\epsilon}({\color{orange}(22,3:3)}) &= {\color{limegreen}(21,4:4)^{ev}} , & d_{2+\epsilon}({\color{orange}(23,4;3)}) &= {\color{cyan}(22,5;5)^{ev}}.\end{align*}
\end{prop}
\begin{proof}
The first differential follows from the fact that ${\Ex}=0$ in bidegree $(22,3)$. The class $h_4c_0$ is detected by ${\color{limegreen}(23,4:4)^{ev}}$.  It follows that all classes of $\Ex$ in $(23,4)$ have been accounted for, and this forces the differential on ${\color{orange}(23,4;3)}$.
\end{proof}

\begin{prop}
There are differentials
\begin{align*}
d_{1+\epsilon}( {\color{red}(24,2;2)})&={\color{orange}(23,3:3)^{ev}}, & d_{1+\epsilon}( {\color{limegreen}(24,4;4)})&={\color{cyan}(23,5:5)^{ev}}, \\
d_{2+\epsilon}({\color{limegreen}(24,5;4)} )&={\color{lavenderrose}(23,6:6)^{ev}}, &
d_{1+\epsilon}( {\color{orange}(25,3;3)})&={\color{limegreen}(24,4:4)^{ev}}.
\end{align*}
\end{prop}
\begin{proof}
After taking into account the good differentials and the restrictions imposed by Figure~\ref{fig:v1Extchart} all classes in $\Ex$ in stem $24$ and $25$ have been accounted for. These are the only possibilities left.
\end{proof}

\ \\
\noindent
{\bf Stem 26-27}

\begin{prop}
There are no non-trivial differentials $d_r$, $r \geq 1+\epsilon$ in the AKSS with source or target satisfying $t-(s+n) =26, 27$. In these stems, $\E{akss}{alg}^{n+\epsilon,s,t}_{1+\epsilon} =0$ and
$H^{n,s,t}(\mc{C}_{alg}) \cong   \E{akss}{alg}^{n,s,t}_{\infty} $.
\end{prop}

\ \\
\noindent
{\bf Stem 28-32}

\begin{prop}
There are differentials
\begin{align*}
d_{3+\epsilon}({\color{red}(30,4:2)}) &= {\color{cyan}(29,5:5)^{ev}} , &
d_{4+\epsilon}({\color{red}(30,5:2)}) &= {\color{lavenderrose}(29,6:6)^{ev}}.\end{align*}
\end{prop}
\begin{proof}
By Figure~\ref{fig:v1Extchart}, both $v_0^2h_4^2$ and $v_0^3h_4^2$ are detected by evil classes ${\color{limegreen}(30,4:4)^{ev}}$ and ${\color{cyan}(30,5:5)^{ev}}$ respectively. This implies that ${\color{red}(30,4:2)}$ and ${\color{red}(30,5:2)}$ do not survive. Taking into account the good differentials already established, this is the only possibility.
\end{proof}

\begin{prop}
$v_0^sh_5$ for $0\leq s \leq 15$ are detected by ${\color{blue}(31, s+1 : 1)}$.
\end{prop}
\begin{proof}
By Figure~\ref{fig:v1Extchart}, $h_5$ is good so must be detected by ${\color{blue}(31, 1 : 1)}$. By $v_0$--linearity, the whole tower consists of permanent cycles. For  degree reasons they cannot be targets of differentials and the claim follows.
\end{proof}

\begin{prop}
There are differentials
\begin{align*}
d_{4+\epsilon}({\color{red}(31,5:2)}) &= {\color{lavenderrose}(30,6:6)^{ev}}, & d_{1+\epsilon}({\color{red}(32,2 : 2)})&={\color{orange}(31,3 : 3)^{ev}_1}.\end{align*}
\end{prop}
\begin{proof}
Since $n$ is detected by an evil class, all elements of $\Ex$ in degrees $(31,5)$ have been accounted for, forcing the first differential. Since $h_1h_5$ is detected by evil, ${\color{red}(32,2 : 2)}$ cannot survive. The $d_{1+\epsilon}$ is the only possibility.
\end{proof}

\subsection*{The proliferation of evil: Stems 33-42}
The developing phenomena in the remaining stems is that all good classes of low Adams filtration die killing evil classes, and the non-zero elements of $\Ex$ are detected by evil.

\ \\
\noindent
{\bf Stems 33-34}

\begin{prop}
There are differentials
\begin{align*}d_{4+\epsilon}({\color{orange}(33,6 : 3)})&={\color{darkmagenta}(32, 7 : 7)^{ev}}  \\
d_{3+\epsilon}({\color{orange}(33,5 : 3)})&={\color{lavenderrose}(32, 6 : 6)^{ev}} \\
 d_{4+\epsilon}({\color{blue}(33,4 : 1)})&={\color{cyan}(32, 5 :5)^{ev}}.
\end{align*}
\end{prop}

\begin{proof}The class
${\color{orange}(33,7 : 3)}$ detects $h_1q$ and, in $\Ex$, it is not divisible by $v_0$. Hence, ${\color{orange}(33,6 : 3)}$ and ${\color{orange}(33,5 : 3)}$ cannot survive and these differentials are the only possibilities.
No element of $\Ex$ in $(33,4)$ is detected by a good class. This forces the $d_{4+\epsilon}$--differential.
\end{proof}

\begin{prop}
There are differentials
\begin{align*}
d_{2+\epsilon}({\color{red}(34,3 : 2)})&={\color{limegreen}(33, 4 : 4)_1^{ev} },\\
d_{3+\epsilon}({\color{red}(34,4 : 2)})&={\color{cyan}(33, 5 : 5)_1^{ev} }, \\
d_{4+\epsilon}({\color{red}(34,5 : 2)})&={\color{lavenderrose}(33, 6 : 6)^{ev} }.
\end{align*}
\end{prop}
\begin{proof}
$\Ex$ in degrees $(34,3)$, $(34,4)$ and $(34,5)$ is either zero, or its elements are detected by evil classes. This forces these differentials.
\end{proof}

\ \\
\noindent
{\bf Stems 35-43}

\begin{prop}\label{prop:detectx}
The class $x$ is detected by ${\color{orange}(37,5 : 3)}$.
\end{prop}
\begin{proof}
Both $x$ and $v_0x$ are detected by good classes, and for $v_0x$, the only possibility is  ${\color{orange}(37,6 : 3)}$. Since there cannot be an exotic $v_0$-extension from $ {\color{limegreen}(37,5 : 4)} $ to ${\color{orange}(37,6 : 3)}$, we must have that $x$ is detected by ${\color{orange}(37,5 : 3)}$.
\end{proof}

\begin{prop}
There are differentials
\begin{align*}
d_{1+\epsilon}({\color{limegreen}(36,4 : 4)}) &={\color{cyan}(35,5 : 5)_2^{ev}}, \\
d_{2+\epsilon}({\color{limegreen}(36,5 : 4)}) &={\color{lavenderrose}(35,6 : 6)^{ev}}, \\
d_{3+\epsilon}({\color{limegreen}(36,6 : 4)}) &={\color{darkmagenta}(35,7 : 7)_1^{ev}}.
\end{align*}
\end{prop}
\begin{proof}
$\Ex$ in degrees $(36,4 )$ and $(36,5)$ is zero and the only non-zero element of $\Ex$ in $(36,6)$ is detected by an evil class. For degree reasons, we must have the first three differentials.
\end{proof}

\begin{prop}
For $i = 1,2$, there are differentials
\begin{align*}
d_{1+\epsilon}({\color{orange}(37,3 : 3)_i}) &= {\color{limegreen}(36,4 : 4)_i^{ev}}, \\
 d_{2+\epsilon}({\color{orange}(37,4 : 3)_i}) &= {\color{cyan}(36,5 : 5)_i^{ev}}.
 \end{align*}
\end{prop}
\begin{proof}
In these bi-degrees, $\Ex$ is zero or detected by evil. For degree reasons, we must have these differentials.
\end{proof}

\begin{prop}\label{prop:lotsofevil}
There are differentials
\begin{align*}
d_{1+\epsilon}({\color{red}(38,2 : 2)}) &= {\color{orange}(37,3 : 3)^{ev}_2},  &d_{2+\epsilon}({\color{red}(38,3 : 2)}) &= {\color{limegreen}(37,4 : 4)_1^{ev}}, \\
d_{3+\epsilon}({\color{red}(38,4: 2)}) &= {\color{cyan}(37,5 : 5)^{ev}} , & d_{4+\epsilon}({\color{red}(38,5: 2)}) &= {\color{lavenderrose}(37,6: 6)^{ev}}
\end{align*}
and, for $i=1,2$, differentials
\begin{align*}
d_{1+\epsilon}(\gdm{orange}{38}{3}{3}{i}) &=  \evm{limegreen}{38}{4}{i+1}, & d_{2+\epsilon}(\gdm{orange}{39}{4}{3}{i}) &=  \evm{cyan}{38}{5}{i+1}.
\end{align*}
\end{prop}
\begin{proof}
All classes of $\Ex$ in the stems $38$ and $39$ for $y \leq 5$ are detected by evil. No good class can survive. For both families, the $d_{1+\epsilon}$'s and the $d_{2+\epsilon}$'s are the only possibilities. By $v_0$--linearity, ${\color{red}(38,4: 2)}$ is a $d_{2+\epsilon}$--cycle, hence it cannot kill ${\color{limegreen}(37,5 : 4)}$. Therefore, we must have the claimed $d_{3+\epsilon}$. The $d_{4+\epsilon}$ is also the only possibility.
\end{proof}

\begin{prop}
There are differentials
\begin{align*}
d_{1+\epsilon}({\color{lavenderrose}(36,6 : 6)}) &={\color{darkmagenta}(35,7 : 7)_2^{ev}}, & d_{2+\epsilon}({\color{limegreen}(37,5 : 4)}) &={\color{lavenderrose}(36,6 :6)_2^{ev}}.
\end{align*}
\end{prop}
\begin{proof}
The class cannot ${\color{lavenderrose}(36,6 : 6)}$ survive since the only non-zero element of $\Ex$ in $(36,6)$ is detected by evil. The only other possibility is that $d_{2}({\color{limegreen}(37,5 : 4)}) = {\color{lavenderrose}(36,6 : 6)}$. However, this would imply that $d_{2}({\color{limegreen}(45,9 : 4)}) = {\color{lavenderrose}(44,10 : 6)}$, contradicting Proposition~\ref{prop:pinkyellow}.

For the $d_{2+\epsilon}$--differential, by Proposition~\ref{prop:detectx}, $x$ has already been accounted for, hence, 
${\color{limegreen}(37,5 : 4)}$ cannot survive. Proposition~\ref{prop:lotsofevil} eliminates the only other possibility.
\end{proof}

\begin{prop}\label{prop:cyan39}
There are differentials
\begin{align*}
d_{1+\epsilon}(\gdr{40}{2}{2}{red}) &= \evm{orange}{39}{3}{2}, & d_{1+\epsilon}(\gdr{41}{3}{3}{orange} )&= \evm{limegreen}{40}{4}{3} \end{align*}
and, for $i=1,2$,
\begin{align*}
d_{1+\epsilon}(\gdr{39}{5}{5}{cyan}) &= \evm{lavenderrose}{38}{6}{2}, & d_{1+\epsilon}( \gdm{limegreen}{40}{4}{4}{i}) &= \evm{cyan}{39}{5}{i+1}\\
 d_{2+\epsilon}(\gdr{39}{6}{5}{cyan}) &= \ev{darkmagenta}{38}{7}, & d_{2+\epsilon}( \gdm{limegreen}{40}{5}{4}{i}) &= \evm{lavenderrose}{39}{6}{i}.
\end{align*}
\end{prop}
\begin{proof}
$\Ex$ is zero in $(40,2)$ and this justifies the first differential.
$\Ex$ in $(39,5)$ and $(39,6)$ is either detected by evil or zero. Hence, the sources of these differentials cannot survive. Taking $\bo$--filtration into account, this is the only possibility. The other differentials are justified in a similar way.
\end{proof}

\begin{prop}
There are differentials
\begin{align*}
d_{1+\epsilon}(\gdr{41}{7}{7}{darkmagenta}) &=\ev{goldenpoppy}{40}{8}, & d_{1+\epsilon}(\gdr{42}{6}{6}{lavenderrose}) &= \evm{darkmagenta}{41}{7}{2},\\
d_{1+\epsilon}(\gdr{42}{7}{7}{darkmagenta}) &=\ev{goldenpoppy}{41}{8}, &
d_{2+\epsilon}(\gdr{43}{8}{7}{darkmagenta}) &=\ev{seagreen}{42}{9} .
 \end{align*}
\end{prop}
\begin{proof}
The only classes in $\Ex$ in degree $(41,7)$, $(42,7)$ and $(42,6)$ are detected by evil. Because of the $\bo$--filtrations, this forces the three $d_{1+\epsilon}$-differentials.

$\Ex$ is zero in degree $(43,8)$. Hence the source of the $d_{2+\epsilon}$ cannot survive. The only other possibility is that $d_{2}(\gdr{44}{7}{5}{cyan}) =\gdr{42}{7}{7}{darkmagenta}$. However, this differential would be $v_1^4$--periodic, and this would contradict the fact that there is a non-zero element in $\Ex$ in degree $(83,28)$.
\end{proof}

\begin{prop}\label{prop:hard}
There is a differential
\[d_{2+\epsilon}(\gdr{38}{7}{6}{lavenderrose}) =\ev{goldenpoppy}{37}{8}. \]
\end{prop}

\begin{proof}
In bidegree $(38,7)$, the element $v_0y$ in $\Ex$, is either detected by $\gdr{38}{7}{6}{lavenderrose}$ or by $\gdr{38}{7}{3}{orange}$. Therefore, one of them supports a nontrivial differential that kills an evil class, and the other one survives.

Suppose that $\gdr{38}{7}{6}{lavenderrose}$ survives in the AKSS. In bidegree $(46,11)$, $d_0l$ is the only element in $\Ex$. Further, the only non-zero element in $\Ex$ in degree $(78,27)$ is $lP^4d_0$. We are in a region where $\Ex \cong \Ext_{A(4)_*}$ and hence $lP^4d_0$ maps to $v_1^{16} d_0l$ in the latter. In particular, $d_0l$ is $v_1$-periodic. This class is above the 1/3-line, and the only element in the AKSS left to detect $lP^4d_0$ is $\gdr{78}{27}{6}{lavenderrose}$. By the Dichotomy Principle, $d_0l$ is good and so must be detected by $\gdr{46}{11}{6}{lavenderrose}$. Finally, in $\E{bo}{alg}^{*,*,*}_{1}$, $v_1^4 \gdr{38}{7}{6}{lavenderrose} = \gdr{46}{11}{6}{lavenderrose}$, which implies that $Pyv_0 = d_0l$ (with zero indeterminacy). However, $Py = 0$ since this bidegree is zero in $\Ex$, a contradiction.
\end{proof}

\begin{prop}\label{prop:exception}
There are differentials
\begin{align*}
d_6( {\color{red}(42,8:2)}) &= {\color{goldenpoppy}(41,9:8)}, & d_5( {\color{red}(42,7:2}) &= {\color{darkmagenta}(41,8:7)} .
\end{align*}
\end{prop}
\begin{proof}
The only class in $\Ex$ in degree $(40,8)$ is detected by ${\color{goldenpoppy}(40,8:8)}$ so by $h_1$--linearity, ${\color{goldenpoppy}(41,9:8)}$ is a permanent cycle. However, $\Ex$ in degree $(41,9)$ is zero. Therefore, the class ${\color{goldenpoppy}(41,9:8)}$ must be hit by a differential. The only possibility is the $d_6$--differential. 

The only class in $\Ex$ in degree $(42,7)$ is detected by evil. Therefore, $\gdr{42}{7}{2}{red}$ must support a differential. If it kills evil, that differential would be a $d_{6+\epsilon}$--differential. We have $v_0\gdr{42}{7}{2}{red} = \gdr{42}{8}{2}{red} $, so such a $d_{6+\epsilon}$ would imply that $d_{6}(\gdr{42}{8}{2}{red} )=0$, contradicting what we have just shown.
The only possibility is this $d_5$--differential.
\end{proof}


\section{Computation of the topological $\bo$-resolution}\label{sec:topcomp}

In this section, we deduce the differentials in the topological $\bo$-based Adams spectral sequence from known computations of the stable homotopy groups of spheres (see \cite{Isaksen} for example). The computation is depicted in Figure~\ref{fig:boASSchart}. We find that certain products in $\pi_\ast$, which are nontrivial extensions in the classical Adams spectral sequence, are products in $E_2$-page of the $\bo$-based Adams spectral sequence, and so they are not exotic extensions. 

\begin{figure}
\includegraphics[angle=90,height=0.9\textheight]{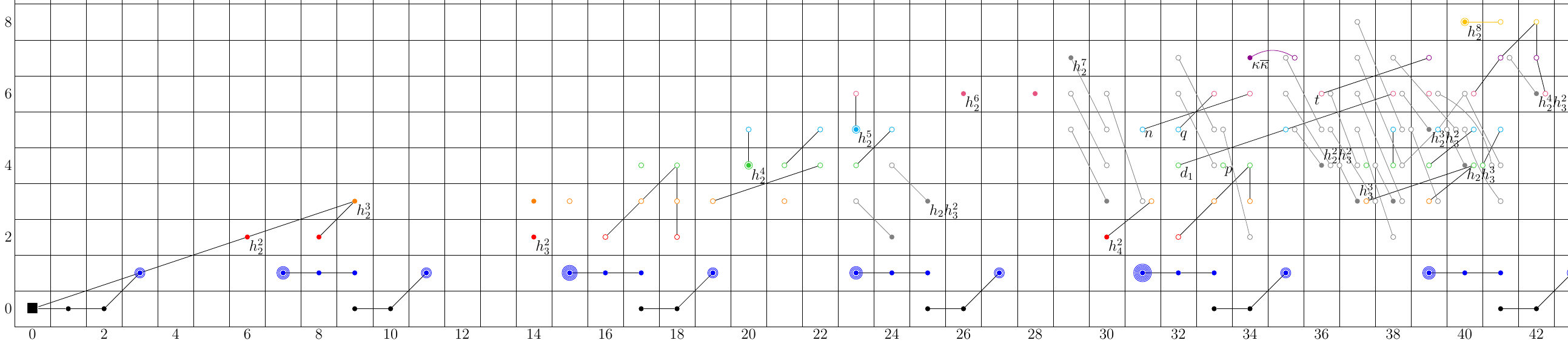}
\caption{The $\bo$ Adams Spectral Sequence. The horizontal axis denotes the topological degree and the vertical axis (and the color) denotes the $\bo$-filtration. 
The $\blacksquare$  denotes $\mathbb{Z}_{(2)}$.  A $\bullet$ denotes a copy of $\mathbb{Z}/2$ detected by a good class and $\circ$ a copy of $\mathbb{Z}/2$ detected by an evil class. 
A \circled{$\bullet$} denotes a copy $\mathbb{Z}/4$, a \circled{\circled{$\bullet$}} a copy of $\mathbb{Z}/8$ etc.. A line between classes which increases topological degree by one is multiplication by $\eta$ and one which increases topological degree by $3$ is multiplication by $\nu$. Gray lines between gray classes show the differentials.
}\label{fig:boASSchart}
\end{figure}

\begin{rem}\label{rem:filt}
In the following computation, we will use the fact that the map from $\bo$ to $H\mathbb{F}_2$ induces a map of spectral sequences
\[\E{bo}{}^{*,*}_2 \to \Ex^{*,*}.\]
In particular, if an element $x \in \pi_*$ is detected in $\Ex$ by a class of Adams filtration $s$, then it must be detected in $\E{bo}{}^{*,*}_2 $ by a class of $\bo$--filtration $n \leq s$. 
\end{rem}

It is straightforward to see that there are no differentials up to stem 28.

\ \\
\noindent
{\bf Stems 29-31}

\noindent

\begin{prop}
The element $\gdr{30}{2}{2}{red}$ survives and detects $\theta_4$.
\end{prop}

\begin{proof}
The element $\gdr{30}{2}{2}{red}$ maps to $h_4^2$ in $\Ex$, which detects $\theta_4$. Since there are no elements with $\bo$-filtration lower than $\gdr{30}{2}{2}{red}$, by Remark \ref{rem:filt}, $\gdr{30}{2}{2}{red}$ survives and detects $\theta_4$.
\end{proof}

\begin{prop} 
There are differentials
$d_{2+\epsilon}(\gdr{30}{3}{3}{orange}) =  \evr{29}{5}{cyan}$, $d_2(\evr{30}{4}{limegreen}) =  \evr{29}{6}{lavenderrose}$ and $d_{2-\epsilon}(\evr{30}{5}{cyan}) =  \gdr{29}{7}{7}{darkmagenta}$.
\end{prop}

\begin{proof}
Since $\pi_{29}=0$, none of the three elements in the 29-stem can survive. This forces these $d_2$--differentials.
\end{proof}

\begin{prop}
One of $\evrm{31}{3}{orange}{1}$ and $\evrm{31}{3}{orange}{2}$ supports a $d_3$ differential that kills $\evr{30}{6}{lavenderrose}$, the other one survives and detects $\eta\theta_4$. Without loss of generality, we adopt the convention that $d_3(\evrm{31}{3}{orange}{1}) = \evr{30}{6}{lavenderrose}$, and $\evrm{31}{3}{orange}{2}$ survives and detects $\eta\theta_4$.
\end{prop}

\begin{proof}
Since $\pi_{30}=\mathbb{Z}/2$, $\evr{30}{6}{lavenderrose}$ must be hit by a differential. The only possibility is that 
it be hit by
one of $\evrm{31}{3}{orange}{1}$ and $\evrm{31}{3}{orange}{2}$ by a $d_3$--differential. The other one must survive and detect $\eta\theta_4$, since the Adams filtration of $\eta\theta_4$ is $s=3$.
\end{proof}

\begin{cor}
The element $\evr{31}{5}{cyan}$ survives and detects $\{n\}$.
\end{cor}

\ \\
\noindent
{\bf Stems 32-34}

\noindent

\begin{prop}
The elements $\evr{32}{2}{red}$, $\evr{32}{4}{limegreen}$, $\evr{32}{5}{cyan}$ and $\evr{33}{3}{orange}$ survive and detect $\eta_5$, $\{d_1\}$, $\{q\}$ and $\eta\eta_5$ respectively. One of the elements $\evrm{33}{4}{limegreen}{1}$ and $\evrm{33}{4}{limegreen}{2}$, say $\evrm{33}{4}{limegreen}{1}$, survives and detects $\nu\theta_4$. (Note that in the classical Adams spectral sequence, $p$ detects $\nu\theta_4$.)
\end{prop}

\begin{proof}
This follows from Remark~\ref{rem:filt}.
\end{proof}

\begin{prop}There is a differential
$d_2(\evrm{33}{4}{limegreen}{2})=\evr{32}{6}{lavenderrose}$.
One of the elements $\evrm{33}{5}{cyan}{1}$ and $\evrm{33}{5}{cyan}{2}$, say $\evrm{33}{5}{cyan}{1}$, supports a $d_2$ differential that kills $\evr{32}{7}{darkmagenta}$. That is, $d_2(\evrm{33}{5}{cyan}{1})=\evr{32}{7}{darkmagenta}$.
\end{prop}

\begin{proof}
Since all classes in $\pi_{32}$ are accounted for, this is the only possibility to kill $\evr{32}{6}{lavenderrose}$ and $\evr{32}{7}{darkmagenta}$.
\end{proof}

\begin{prop}
There is a differential $d_3(\evr{34}{2}{red})=\evrm{33}{5}{cyan}{2}$ and $\evr{33}{6}{lavenderrose}$ survives to detect $\eta\{q\}$.
\end{prop}
\begin{proof}
Since $h_2h_5$ is detected by $\evr{34}{2}{red}$ and $h_0 p$ by $\evrm{33}{5}{cyan}{2}$ in the algebraic $\bo$ spectral sequence, the $d_3$ differential $d_3(h_2h_5) = h_0p$ in the classical Adams spectral sequence implies the first claim and $\evr{33}{6}{lavenderrose}$ is the only class left to detect $\eta\{q\}$. 
\end{proof}

\begin{cor}
The elements $\evr{34}{3}{orange}$, $\evr{34}{4}{limegreen}$, $\evr{34}{6}{lavenderrose}$, $\gdr{34}{7}{7}{darkmagenta}$ survive and detect $\{h_0h_2h_5\}$, $\{h_0^2h_2h_5\}$, $\nu\{n\}$, $\kappa\overline{\kappa}$ respectively. (Note that it is $d_0g$ that detects $\kappa\overline{\kappa}$ in the classical Adams spectral sequence.)
\end{cor}

\ \\
\noindent
{\bf Stems 35-36}

\noindent

\begin{prop}
The element $\evr{35}{5}{cyan}$ survives and detects $\nu\{d_1\}$. One of the elements $\evrm{35}{7}{darkmagenta}{1}$ and $\evrm{35}{7}{darkmagenta}{2}$, say $\evrm{35}{7}{darkmagenta}{1}$, survives and detects $\eta\kappa\overline{\kappa}$.
\end{prop}

\begin{proof}
The first claim follows from Remark~\ref{rem:filt}. Since $\kappa\overline{\kappa}$ is detected by $\gdr{34}{7}{7}{darkmagenta}$ with $\bo$-filtration 7, $\eta\kappa\overline{\kappa}$ has $\bo$-filtration at least 7. (Note that $\eta$ has $\bo$-filtration 0). Therefore, the only possibility is that one of the elements $\evrm{35}{7}{darkmagenta}{1}$ or $\evrm{35}{7}{darkmagenta}{2}$ detects $\eta\kappa\overline{\kappa}$.
\end{proof}

\begin{prop}
$d_2(\evr{36}{4}{limegreen})=\evr{35}{6}{lavenderrose}$.
One of the elements $\evrm{36}{5}{cyan}{1}$ and $\evrm{36}{5}{cyan}{2}$, say $\evrm{36}{5}{cyan}{1}$, supports a $d_2$ differential that kills $\evrm{35}{7}{darkmagenta}{2}$. That is, $d_2(\evrm{36}{5}{cyan}{1})=\evrm{35}{7}{darkmagenta}{2}$.
\end{prop}

\begin{proof}
Since all classes in $\pi_{35}$ are accounted for, this is the only possibility to kill $\evr{35}{6}{lavenderrose}$ and $\evrm{35}{7}{darkmagenta}{2}$.
\end{proof}

\begin{prop}
The element $\evrm{37}{3}{orange}{1}$ survives and detects $\{h_2^2h_5\}$.
\end{prop}

\begin{proof}
This follows from Remark~\ref{rem:filt}.
\end{proof}

\begin{prop}
The element $\evrm{36}{6}{lavenderrose}{1}$ survives and detects $\{t\}$ and there are differentials 
\begin{align*}
d_2(\evrm{37}{3}{orange}{2}) &= \evrm{36}{5}{cyan}{2}, & d_2(\evrm{37}{4}{limegreen}{1}) &= \evrm{36}{6}{lavenderrose}{2}.\end{align*}
\end{prop}

\begin{proof}
The element $\evrm{36}{6}{lavenderrose}{1}$ maps to $t$ in $\Ex$. Since $t$ is not a boundary in the classical Adams spectral sequence, $\evrm{36}{6}{lavenderrose}{1}$ is also not a boundary. Therefore, $\evrm{36}{6}{lavenderrose}{1}$ survives and detects $\{t\}$. The other two differentials are the only possibilities left.
\end{proof}

\ \\
\noindent
{\bf Stems 37-38}

\noindent

\begin{prop}\label{prop:length38}
The elements $\evrm{38}{4}{limegreen}{1}$, $\evrm{38}{5}{cyan}{1}$ survive and detect $\{h_0^2h_3h_5\}$, $\{h_0^3h_3h_5\}$ respectively.
\end{prop}

\begin{proof}
The elements $\evr{38}{2}{red}$, $\evr{38}{3}{orange}$ and $\evrm{38}{4}{limegreen}{2}$ map to $h_3h_5$, $h_0h_3h_5$ and $e_1$ in $\Ex$. Since $h_3h_5$, $h_0h_3h_5$ and $e_1$ 
support non-trivial
differentials 
in the classical Adams spectral sequence, so do $\evr{38}{2}{red}$, $\evr{38}{3}{orange}$ and $\evrm{38}{4}{limegreen}{2}$
in the $\bo$-Adams spectral sequence. By Remark~\ref{rem:filt} and filtration reasons, $\evrm{38}{4}{limegreen}{1}$ and $\evrm{38}{5}{cyan}{1}$ survive and detect $\{h_0^2h_3h_5\}$ and $\{h_0^3h_3h_5\}$ respectively.
\end{proof}

\begin{prop}
The element $\evr{38}{6}{lavenderrose}$ survives and detects $\nu^2\{d_1\}$.
\end{prop}

\begin{proof}
On one hand, $\nu^2\{d_1\}$ has Adams filtration 6, and therefore $\bo$-Adams filtration at most 6. On the other hand, $\nu$ and $\nu\{d_1\}$ have $\bo$-Adams filtration 1 and 5, therefore $\nu^2\{d_1\}$ has $\bo$-Adams filtration at least 6. Therefore, $\nu^2\{d_1\}$ has $\bo$-Adams filtration 6, and the claim then follows from the fact that $\evr{38}{6}{lavenderrose}$ is the only element left in $\bo$-Adams filtration 6 in this stem.
\end{proof}

\begin{prop}\label{prop:seagreend3}
$d_3(\evrm{38}{5}{cyan}{2}) = \evr{37}{8}{seagreen}$. 
\end{prop}

\begin{proof}
In $\pi_{37}$, all classes have Adams filtration at most 5, and therefore $\bo$-Adams filtration at most 5. The target element $\evr{37}{8}{seagreen}$ has $\bo$-Adams filtration 8, therefore must be killed. Since $\evr{38}{6}{lavenderrose}$ survives, by Proposition~\ref{prop:length38}, the only possibility left to kill $\evr{37}{8}{seagreen}$ is an element in $\bo$-Adams filtration 5, say $\evrm{38}{5}{cyan}{2}$.
\end{proof}

\begin{prop}
$d_2(\evrm{39}{3}{orange}{1}) = \evrm{38}{5}{cyan}{3}$.
\end{prop}

\begin{proof}
Every class in $\pi_{38}$ has already been accounted for, therefore $\evrm{38}{5}{cyan}{3}$ must either support a differential or get killed. Suppose it is not killed. Then it can only support a $d_2$ differential that kills $\evr{37}{7}{darkmagenta}$. It follows that the elements $\evr{38}{2}{red}$, $\evr{38}{3}{orange}$, and $\evrm{38}{4}{limegreen}{2}$ kill elements in $\bo$-Adams filtration 6, 5, 4 respectively. This leaves only one element in stem 37. However, $\pi_{37}=\mathbb{Z}/2\oplus\mathbb{Z}/2$, a contradiction. Therefore, we must have $d_2(\evrm{39}{3}{orange}{1}) = \evrm{38}{5}{cyan}{3}$.
\end{proof}

\begin{prop}
The elements $\evrm{39}{3}{orange}{2}$, $\evr{39}{4}{limegreen}$, $\evrm{39}{5}{cyan}{1}$, $\evrm{40}{4}{limegreen}{1}$, $\evrm{40}{4}{limegreen}{2}$ and $\evrm{40}{5}{cyan}{1}$ survive and detect $\sigma\eta_5$, $\{h_5c_0\}$, $\sigma\{d_1\}$, $\eta\sigma\eta_5$, $\{f_1\}$ and $\epsilon\eta_5$ respectively. (Note that $\sigma\eta_5$, $\sigma\{d_1\}$ and $\epsilon\eta_5$ are detected by $h_1h_3h_5$, $h_1e_1$ and $h_1h_5c_0$ in the classical Adams spectral sequence respectively).
\end{prop}

\begin{proof}
This follows from Remark~\ref{rem:filt}.
\end{proof}

\begin{prop}There is a differential
$d_2(\evrm{39}{5}{cyan}{2}) = \evr{38}{7}{darkmagenta}$.
\end{prop}

\begin{proof}
Since all classes in $\pi_{38}$ are accounted for, this is the only possibility left.
\end{proof}

\begin{prop}There is a differential
$d_3(\evrm{38}{4}{limegreen}{2}) = \evr{37}{7}{darkmagenta}$.
\end{prop}

\begin{proof}
Since $\evrm{38}{4}{limegreen}{2}$ maps to $e_1$ in $\Ex$, it must support a $d_2$ or $d_3$ differential. Suppose that $d_2(\evrm{38}{4}{limegreen}{2}) =\evr{37}{6}{lavenderrose}$. This would force a differential $d_4(\evr{38}{3}{orange}) = \evr{37}{7}{darkmagenta}$ since $\evr{37}{7}{darkmagenta}$ must be hit by a differential and this is the only possibility. (The element $\evr{38}{2}{red}$ cannot support a $d_5$ differential, since its image in $\Ex$ supports a $d_4$ differential.) This would imply that $\evr{37}{7}{darkmagenta}$ maps to $ d_4(h_0h_3h_5) = h_0^2x \neq h_2^2n$. This is a contradiction since $\evr{37}{7}{darkmagenta}$ maps to $h_2^2n$ in $\Ex$. Therefore, we must have $d_3(\evrm{38}{4}{limegreen}{2}) = \evr{37}{7}{darkmagenta}$.
\end{proof}

\begin{prop}There is a differential
$d_3(\evr{38}{3}{orange}) = \evr{37}{6}{lavenderrose}$.
\end{prop}

\begin{proof}
The element $\evr{37}{6}{lavenderrose}$ in $\bo$-Adams filtration 6 must be killed, since all classes in $\pi_{37}$ have Adams filtration at most 5. The other possibility to kill it is by a $d_4$ differential: $d_4(\evr{38}{2}{red}) = \evr{37}{6}{lavenderrose}$. This would imply that $\evr{37}{6}{lavenderrose}$ maps to $d_4(h_3h_5) = h_0x$ in $\Ex$, which is not the case since it maps to the target of an algebraic $d_{1+\epsilon}$--differential (see Proposition~\ref{prop:lotsofevil}).
\end{proof}

\begin{prop}There is a differential
$d_3(\evr{38}{2}{red}) = \evr{37}{5}{cyan}$. The element $\evrm{37}{4}{limegreen}{2}$ survives and detects $\sigma\theta_4$. (Note that $\sigma\theta_4$ is detected by $x$ in the classical Adams spectral sequence.)
\end{prop}

\begin{proof}
The other possibility is that $\sigma\theta_4$ is detected by $\evr{37}{5}{cyan}$. Since $\sigma\theta_4$ has Adams filtration 5, this implies that the element $\evr{37}{5}{cyan}$ maps to $x$ in $\Ex$, which would contradict Proposition~\ref{prop:detectx} and Proposition~\ref{prop:lotsofevil}.
\end{proof}

\ \\
\noindent
{\bf Stems 39-42}

\noindent

\begin{prop}
There are differentials
\begin{align*}
d_2(\evr{41}{3}{orange}) &= \evrm{40}{5}{cyan}{2}, & d_2(\evr{41}{4}{limegreen}) &= \evrm{40}{6}{lavenderrose}{1}.
\end{align*} 
The elements  $\evr{41}{5}{cyan}$, $\evrm{40}{6}{lavenderrose}{2}$, $\evr{41}{7}{darkmagenta}$, $\gdr{40}{8}{8}{goldenpoppy}$, $\evr{41}{8}{goldenpoppy}$, $\evr{42}{6}{lavenderrose}$, $\evr{42}{7}{darkmagenta}$ and $\evr{42}{8}{goldenpoppy}$ survive and detect $\{f_1\}$, $\{Ph_1h_5\}$, $\eta\{Ph_1h_5\}$, $\overline{\kappa}^2$, $\eta\overline{\kappa}^2$, $\{Ph_2h_5\}$, $2\{Ph_2h_5\}$ and $4\{Ph_2h_5\}$ respectively.
\end{prop}

\begin{proof}
This is the only possibility left.
\end{proof}

\begin{prop}
The elements $\evrm{39}{6}{lavenderrose}{1}$ and $\evr{39}{7}{darkmagenta}$ survive and detect $\{u\}$ and $\nu\{t\}$. Further, there is a differential $d_2(\evrm{40}{4}{limegreen}{3}) = \evrm{39}{6}{lavenderrose}{2}$.
\end{prop}

\begin{proof}
On one hand, $\nu\{t\}$ has Adams filtration 7, and therefore $\bo$-Adams filtration at most 7. On the other hand, $\nu$ and $\{t\}$ have $\bo$-Adams filtration 1 and 6, therefore $\nu\{t\}$ has $\bo$-Adams filtration at least 7. 
That $\evr{39}{7}{darkmagenta}$ detects $\nu\{t\}$ follows from the fact that this is the only element in $\bo$-Adams filtration 7.
All classes in $\pi_{40}$ are accounted for. Therefore, the element $\evrm{40}{4}{limegreen}{3}$ must support a differential, and this is the only possibility.
\end{proof}

We finish with a few remarks on the $E_\infty$-page.

\begin{rmk}[Stems 8-9]
Recall that $\epsilon \in \pi_8$ is defined to be the unique class with Adams filtration 3. Also recall that we have a relation in $\pi_9$: 
$$\nu^3 = \eta \epsilon + \eta^2\sigma.$$
In $\pi_9$, $\nu^3$ is detected by $\gdr{9}{3}{3}{orange}$, since $\nu^2$ is detected by $\gdr{6}{2}{2}{red}$. It follows from the above relation that $\epsilon + \eta\sigma$ is detected by $\gdr{8}{2}{2}{red}$. Since $\eta\sigma$ is detected by $\gdr{8}{1}{1}{black}$, we have that $\epsilon$ is also detected by $\gdr{8}{1}{1}{black}$. This is interesting since $\epsilon$ and $\eta\sigma$ are detected by different elements in the classical Adams spectral sequence and $\epsilon$ is not divisible by $\eta$.
\end{rmk}

\begin{rmk}[Stem 14]
Recall that $\kappa \in \pi_{14}$ is defined to be the unique class with Adams filtration 4. 

The class $\gdr{14}{2}{2}{red}$ detects both $\sigma^2$ and $\sigma^2+\kappa$, since both $\sigma^2$ and $\sigma^2+\kappa$ have Adams filtration 2, and therefore $\bo$-Adams filtration at most 2. Therefore, $\gdr{14}{3}{3}{orange}$ detects $\kappa$.
\end{rmk}

\begin{rmk}[Stem 22]
The class $\evr{22}{3}{orange}$ detects both $\nu\overline{\sigma}$ and $\nu\overline{\sigma}+\eta^2\overline{\kappa}$, since both $\nu\overline{\sigma}$ and $\nu\overline{\sigma}+\eta^2\overline{\kappa}$ have Adams filtration 4, and therefore $\bo$-Adams filtration at most 4. Therefore, $\evr{22}{5}{cyan}$ detects $\eta^2\overline{\kappa}$.

Recall that in $\pi_{22}$, we have a relation:
$$\eta^2\overline{\kappa} = \epsilon \kappa = (\epsilon + \eta\sigma) \kappa,$$
which is both a nontrivial $\eta$-extension from $\eta\overline{\kappa}$ and a nontrivial $\kappa$-extension from $\epsilon + \eta\sigma$. However, in the $\bo$-Adams spectral sequence, $\eta^2\overline{\kappa}$ has $\bo$-Adams filtration 5, while $\epsilon + \eta\sigma$ and $\kappa$ have $\bo$-Adams filtration 2 and 3. Therefore, this relation is present in the $E_2$-page of the $\bo$-Adams spectral sequence.
\end{rmk}

\bibliographystyle{amsalpha}
\bibliography{boASS6}
\end{document}